\documentclass[a4paper, 11pt, oneside, reqno, svgnames]{amsart}
\usepackage[a4paper]{geometry}
\usepackage[utf8]{inputenc}
\usepackage[english]{babel}
\usepackage[T1]{fontenc}
\usepackage{textcomp}

\usepackage{txfonts}

\usepackage{apptools}

\usepackage[mathscr]{euscript}

\usepackage{tikz-cd}
\usepackage{quiver}
\usepackage{todonotes}

\usepackage{amsmath}
\usepackage{amsfonts}
\usepackage{amsthm}
\usepackage{latexsym}
\usepackage{amssymb}
\usepackage{xcolor}
\definecolor{blue(munsell)}{rgb}{0.0, 0.5, 0.69}
\definecolor{random}{rgb}{0.3, 0.8, 0.3}

\usepackage{hyperref}
\hypersetup{hypertexnames = false, bookmarksdepth = 2, bookmarksopen = true, colorlinks=true, linkcolor = black, citecolor = blue(munsell), urlcolor = blue(munsell), pdfstartview={XYZ null null 1}}

\makeatletter
\def\l@subsection{\@tocline{2}{0pt}{2pc}{5pc}{}}
\makeatother

\newenvironment{psmallmatrix}{\left(\begin{smallmatrix}}{\end{smallmatrix}\right)}

\numberwithin{equation}{section}

\DeclareMathOperator{\Hom}{Hom}
\DeclareMathOperator{\Homdg}{\underline{Hom}}
\DeclareMathOperator{\Ext}{Ext}

\DeclareMathOperator{\Fun}{Fun}
\DeclareMathOperator{\Ob}{Ob}
\DeclareMathOperator{\coker}{coker}

\DeclareMathOperator{\Mod}{Mod}

\DeclareMathOperator{\Img}{Im}

\DeclareMathOperator{\Fundg}{Fun_{dg}}

\DeclareMathOperator{\pretr}{pretr}
\DeclareMathOperator{\cone}{cone}

\DeclareMathOperator{\id}{id}

\DeclareMathOperator{\colim}{colim}
\DeclareMathOperator{\holim}{holim}
\DeclareMathOperator{\hocolim}{hocolim}

\DeclareMathOperator{\Inj}{Inj}
\DeclareMathOperator{\Proj}{Proj}
\DeclareMathOperator{\GInj}{GInj}
\DeclareMathOperator{\GProj}{GProj}
\DeclareMathOperator{\WTriv}{WTriv}

\DeclareMathOperator{\pshdg}{dgm}

\DeclareMathOperator{\dercomp}{\mathsf{D}}
\DeclareMathOperator{\dercompdg}{\mathsf{D}_{\mathrm{dg}}}

\DeclareMathOperator{\QFun}{\dercomp^{\mathrm{qf}}}

\DeclareMathOperator{\RHom}{\mathbb R \underline{\Hom}}
\DeclareMathOperator{\lotimes}{\overset{\mathbb L}{\otimes}}
\newcommand{\cat}{\mathscr}

\newcommand{\opp}[1]{{#1}^{\mathrm{op}}}
\newcommand{\kat}{\mathsf}

\newcommand{\stalk}[1]{S\!_{#1}\,}

\newcommand{\derivator}{\mathscr D}

\newcommand{\Set}{\kat{Set}}
\newcommand{\Cat}{\kat{Cat}}
\newcommand{\CAT}{\kat{CAT}}
\newcommand{\Hqe}{\kat{Hqe}}

\newcommand{\wholim}[2]{\holim\nolimits^{#2}{\! #1}}

\newcommand{\whocolim}[2]{\hocolim\nolimits^{#2} {\! #1}}

\newcommand{\qis}{\overset{\mathrm{qis}}{\approx}}

\newcommand{\basering}[1]{\mathbf{#1}}

\newcommand{\dgm}[2]{{_{#1}} \mathrm{dgm}{_{#2}}}
\newcommand{\rdgm}[1]{\dgm{}{#1}}
\newcommand{\ldgm}[1]{\dgm{#1}{}}

\newtheorem{theorem}{Theorem}[section]
\newtheorem*{theorem*}{Theorem}
\newtheorem*{variant*}{Variant}
\newtheorem{proposition}[theorem]{Proposition}
\newtheorem{corollary}[theorem]{Corollary}

\newtheorem{lemma}[theorem]{Lemma}

\theoremstyle{remark}
\newtheorem{remark}[theorem]{Remark}
\newtheorem*{remark*}{Remark}
\newtheorem{example}[theorem]{Example}
\newtheorem{notation}[theorem]{Notation}

\theoremstyle{definition}
\newtheorem{definition}[theorem]{Definition}

\AtAppendix{\counterwithin{theorem}{section}}

\title{The derivator of a dg-category}

\author{Francesco Genovese}
\address[Francesco Genovese]{Praha, Česká republika.}
\email{fg.anisama@gmail.com}

\author{Chiara Sava} 
\address[Chiara Sava]{Univerzita Karlova, Matematicko-fyzik\'{a}ln\'{i} fakulta, Katedra Algebry, Sokolovsk\'{a} 49/83, 186 75 Praha 8, Česká republika.}
\email{sava@karlin.mff.cuni.cz}

\makeatletter
\let\@wraptoccontribs\wraptoccontribs
\makeatother

\contrib[with an appendix by]{Jan Šťovíček}

\address[Jan Šťovíček]{Univerzita Karlova, Matematicko-fyzik\'{a}ln\'{i} fakulta, Katedra Algebry, Sokolovsk\'{a} 49/83, 186 75 Praha 8, Česká republika.}
\email{stovicek@karlin.mff.cuni.cz}


\thanks{
C.~S. was supported by the Charles University Grant Agency (GA UK) project n° 222923.
J.~\v{S}. was supported by the Czech Science Foundation grant GA\v{C}R~23-05148S}

\begin{document}
\begin{abstract}
  In this work, we construct the stable derivator associated to a homotopically complete and cocomplete dg-category by explicitly defining homotopy Kan extensions via suitable weighted homotopy limits and colimits in dg-categories. By restricting the domain of the derivator to finite direct categories, we obtain a well-defined derivator even for pretriangulated dg-categories. This definition enables an explicit description of the derivator associated to a weakly idempotent complete Frobenius exact category, leading to a more direct characterization in terms of Gorenstein projective (equivalently, Gorenstein injective) diagrams.
\end{abstract}

\maketitle

\tableofcontents

\section{Introduction}
\emph{Derived categories} (of rings or schemes) are nowadays a classical topic in mathematics, with many applications to geometry and algebra. They are defined by localizing the categories of complexes of modules over the given ring (or sheaves over the given scheme) along \emph{quasi-isomorphisms}, that is, maps inducing isomorphisms in cohomology. Derived categories are not in general very well-behaved from a categorical point of view: for instance, the usual limits and colimits of complexes (e.g. kernels and cokernels) will not in general yield corresponding limits and colimits inside the derived category. Instead of those, we should look for \emph{homotopy limits and colimits}, namely, variants of the notion of limits and colimits of complexes which have the crucial property of being \emph{stable under quasi-isomorphisms}.

The notion of \textit{triangulated structure}, making derived categories the typical examples of \emph{triangulated categories}, partially addresses the above problem. Indeed, we may define some homotopy limits and colimits in a derived category, such as \emph{mapping cones}; unfortunately, these homotopy limits and colimits will be given \emph{just as objects, not as functors}. Such lack of functoriality is one of the main technical drawbacks of working just with derived categories -- and, more in general, with triangulated categories: it makes the theory a useful tool for many purposes but ultimately deficient.

In order to recover functorial homotopy limits and colimits, derived categories and triangulated categories just do not carry enough information. It turns out that they are almost always “shadows” of more complicated, higher categorical structures, where all homotopy limits and colimits can be defined and behave functorially as expected. There are many types of such higher structures, and choosing one or the other is typically a matter of situational convenience. It is also expected that all of such higher categorical theories ``enhancing'' triangulated categories are, in some sense, equivalent. 

A popular choice of ``enhancements'' of triangulated categories is given by \emph{differential graded (dg) categories}. These are categories enriched in complexes of modules over a fixed base commutative ring, and they possess a homotopy theory induced by the homotopy theory of complexes themselves \cite{tabuada-quillendg} \cite{toen-dgcat-invmath}. The notion of \emph{pretriangulated dg-category} can be given \cite{bondal-kapranov-enhanced}, and if $\cat A$ is such, then its \emph{homotopy category} $H^0(\cat A)$ carries a natural structure of triangulated category. Inside $\cat A$ we may define functorial homotopy limits and colimits, in particular functorial mapping cones. Recent results \cite{canonaco-stellari-neeman-dgenh-all} \cite{canonoaco-stellari-dgenh-grothendieck} ensure that essentially all derived categories admit an essentially unique dg-categorical enhancement -- namely, a pretriangulated dg-category such that its $H^0$ is equivalent to the given derived category.

Dg-categories are a popular means of enhancing triangulated categories, but there are other possible choices. In this work, we will be interested in \emph{stable derivators}. Derivators were first introduced by Grothendieck \cite{grothendieck-derivateurs} and Heller \cite{heller-homotopy} \cite{heller-stablehomotopies} and more recently developed in \cite{groth-derivators}; stable derivators are specific derivators tailored for derived categories (and stable homotopy theory). If $R$ is a ring, we may define a stable derivator which enhances the derived category $\dercomp(R)$ as the following $2$-functor:
\begin{align*}
    I \mapsto \dercomp(\Mod(R)^I),
\end{align*}
where $I$ is any small category, and $\Mod(R)^I$ denotes the (abelian) category of functors $I \to \Mod(R)$. We see that by taking $I=e$ (the terminal category), we recover $\dercomp(R)$ itself, but keeping track of all derived categories of functors $I \to \Mod(R)$ gives a natural framework allowing us to define and work with homotopy Kan extensions, hence homotopy limits and colimits.

Having both dg-categories and stable derivators as possible enhancements for derived categories, we come to a natural question: can such enhancements be compared? The goal of this paper is to give a first positive definition to that question, by directly constructing, for a given (pretriangulated, with suitable homotopical completeness properties) dg-category, an associated derivator:
\begin{theorem*}[see Theorem \ref{theorem:derivator_dg}]
    Let $\cat A$ be a pretriangulated dg-category admitting all homotopy limits and colimits. Then, the $2$-functor
    \[
    I \mapsto \QFun(\basering k[I], \cat A)
    \]
    defines a strong and stable derivator. Here $\basering k$ is a base commutative ring, $\basering k[I]$ denotes the free $\basering k$-linear category generated by the small category $I$, and $\QFun(\basering k[I], \cat A)$ denotes the category of \emph{quasi-functors} (intuitively: homotopically coherent morphisms of dg-categories) from $\basering k[I]$ to $\cat A$.
\end{theorem*}
The key tool in the proof of this result is the description of \emph{homotopy limits and colimits} and hence \emph{homotopy Kan extensions} in dg-categories, given in Section \ref{subsection:hokandg}. The main idea behind the definition of such homotopy (co)limits is to view them as particular weighted colimits (as in enriched category theory), where the weight undergoes a suitable resolution in order to ensure stability under quasi-isomorphisms. This theory of homotopy limits and colimits will also allow us to give a characterization of pretriangulated dg-categories as precisely those dg-categories admitting \emph{finite} homotopy limits and colimits (Section \ref{subsection:pretr_dg_hocolim}, Proposition \ref{proposition:pretr_dgcat_finite_holims}). With this, we obtain a variant of the above theorem:
\begin{variant*}[see Theorem \ref{theorem:derivator_dg_finite}]
    Let $\cat A$ be a (non empty) pretriangulated dg-category. Then, the $2$-functor
    \[
    I \mapsto \QFun(\basering k[I], \cat A),
    \]
    defined on the $2$-category $\opp{\kat{FinDir}}$ whose objects are \emph{finite direct categories}, defines a strong and stable derivator.
\end{variant*}

As a rather immediate corollary of the above theorem (see Corollary \ref{corollary:derivator_ring}), we can prove that, for a given ring $R$, by taking $\cat A= \dercompdg(R)$ a dg-categorical enhancement of $\dercomp(R)$, the $2$-functors
\begin{align*}
    I & \mapsto \dercomp(\Mod(R)^I), \\
    I &\mapsto \QFun(\basering k[I], \dercompdg(R))
\end{align*}
are equivalent, giving us the desider connection between the dg-categorical and the derivator-enhancement of the derived category $\dercomp(R)$.

The derivator associated to a dg-category is also useful in another application. Given a Frobenius exact category $\cat F$, we know that the stable category $\underline{\cat F}$ is triangulated and, as a fact of the matter, almost all triangulated categories in algebra occur in this way (they are often called \emph{algebraic triangulated categories}, see e.g.\ \cite[\S7]{krause-chicagonotes}). So this is another very common and general ``enhancement''.
It is well known that the stable category $\underline{\cat F}$ admits a dg-enhancement given by 
$\operatorname{Ch^{dg}_{ac}}(\operatorname{Proj}(\cat F))$,
the dg-category of acyclic complexes of projectives objects of $\cat F$ \cite{keller-vossieck-derived}. Our aim is to give an explicit description of the derivator associated to this dg-category, i.e. the derivator of a Frobenius exact category $\cat F$, when  we restrict the domain of the derivator to finite direct categories. For technical reasons, we assume that $\cat F$ is weakly idempotent complete (i.e.\ that every section in $\cat F$ has a cokernel, or equivalently that every retraction has a kernel, \cite[\S7]{buehler}).

\begin{theorem*}[see Theorem~\ref{theorem:derivator-Frobenius}]
    Let $\cat F$ be a weakly idempotent complete Frobenius exact category and let $I$ a finite direct category. The following equivalence holds
        \begin{equation*}
         \QFun(\basering k[I], \operatorname{Ch^{dg}_{ac}}(\operatorname{Proj}(\cat F))) \cong
         \operatorname{K_{ac}}(\Proj(\cat{F}^I)) \cong
         \operatorname{K_{ac}}(\Inj(\cat{F}^I)).
     \end{equation*}
Here $\operatorname{K_{ac}}(\Proj(\cat{F}^I))$ is the category of acyclic chain complexes of projectives objects in the exact category of functors $I \to \cat F$, and similarly $\operatorname{K_{ac}}(\Inj(\cat{F}^I))$ is the category of acyclic chain complexes of injective objects in the exact category of functors $I \to \cat F$
    \end{theorem*}

In the appendix, we give a more direct description of the derivator of a (weakly idempotent complete) Frobenius exact category. The passage through acyclic complexes, although necessary to apply our tools, may feel artificial in hindsight, and it appears useful to obtain a better understanding of homotopy (co)limits (or homotopy Kan extensions) in this context.

The main result of the Appendix can be viewed from two perspectives. For readers familiar with Quillen model categories and Reedy model structures, we recall that $\underline{\cat F}$ identifies with the localization $\cat F[W^{-1}]$, where $W$ is the class of all stable isomorphisms (i.e.\ the morphisms in $\cat F$ which become invertible in $\underline{\cat F}$). This localization is underlined by two model structures on $\cat F$, one where cofibrations are precisely the inflations and one where fibrations are precisely the deflations. Then our result essentially summarizes properties of the corresponding Reedy model structures on diagram categories, but the existence of the Reedy model structures is rather non-trivial in our context, since $\cat F$ is in general far from being complete or cocomplete.

For readers familiar with representation theory, the explanation is that $\cat F^I$ is typically not Frobenius, but it is what we call a Gorenstein exact category (in the sense that it is homologically similar to module categories of Iwanaga-Gorenstein rings~\cite{iwanaga-part1,iwanaga-part2}, see also~\cite[Chapter~11, Definition~4.1]{sauter-habilitation}).
In such an exact category, there are two natural Frobenius exact full subcategories, the categories of Gorenstein projective diagrams $\GProj(\cat F^I)$ and Gorenstein injective diagrams $\GInj(\cat F^I)$, which coincide precisely with the cocycles of complexes in $\operatorname{Ch_{ac}}(\Proj(\cat{F}^I))$ and $\operatorname{Ch_{ac}}(\Inj(\cat{F}^I))$, respectively. In fact, the categories of Gorenstein projective and Gorenstein injective diagrams were studied in various contexts in representation theory. The role of our category $\cat F$ was usually (but not always) taken by the category of finitely generated modules over a finite-dimensional algebra and various restrictions were usually imposed on the shape $I$, ranging from~\cite{jorgensen-kato-symmetric-AB} \cite{li-zhang-construction-GP} (for $I=[1]$, the two-element poset $(0<1)$) through \cite{xiong-zhang-GP-triangular} \cite{zhang-mono-GP} (for finite ordinals $I=[n]=(0<1<\cdots<n)$) and \cite{enochs-estrada-garciarozas_injective-reps} \cite{eshraghi-hafezi-salarian-total-acyclicity} \cite{luo-zhang_monic-GP} (for $I$ freely generated by a finite acyclic quiver) to \cite{shen-GP-tensor} (which in principle encompasses all finite direct categories).
We have no ambition for completeness of this list of references since Gorenstein homological algebra is an active and popular topic which was and is studied by many authors.
We just remark that our situation is on the one hand much more general (in particular $\cat F$ is an exact category, so does not even have all finite (co)limits), but on the other hand we assume a strong homological condition of $\cat F$ being Gorenstein.

\begin{theorem*}[see Theorem~\ref{theorem:derivator_Frobenius_direct}]
Let $\cat F$ be a weakly idempotent complete Frobenius exact category and let $\derivator_{\cat F} \colon \opp{\kat{FinDir}} \to \kat{CAT}$ be the derivator as in the previous theorem, i.e.\ given by
\[
\derivator_{\cat F}(I) = \QFun(\basering k[I], \operatorname{Ch^{dg}_{ac}}(\operatorname{Proj}(\cat F))).
\]
Then for any finite direct category $I$, we have equivalences
\[
\derivator_{\cat F}(I)\cong \cat F^I[W_I^{-1}]\cong \underline{\GProj(\cat F^I)} \cong \underline{\GInj(\cat F^I)},
\]
where $W_I$ consists of the morphisms of diagrams $f\colon X\to Y$ in $\cat F^I$ such that $f_i\colon X_i\to Y_i$ is a stable isomorphism in $\cat F$ for each $i\in I$.

Moreover, for any functor $u\colon I\to J$ between finite direct categories, the obvious exact restriction functor $f^*\colon \cat F^J\to\cat F^I$ has a partially defined exact left adjoint $u_!\colon \GProj(\cat F^I)\to\GProj(\cat F^J)$ and a partially defined exact right adjoint $u_*\colon\GInj(\cat F^I)\to\GInj(\cat F^J)$. These partially defined adjoints induce, respectively, the left and the right adjoint to the restriction functor $u^*\colon \derivator_{\cat F}(J)\to\derivator_{\cat F}(I)$ (i.e.\ the left and right Kan extensions along $u$ in $\derivator_{\cat F}$), as indicated by the following diagram.
\[
\begin{tikzcd}
\underline{\GProj(\cat F^I)} \ar[d, bend right, swap, pos=0.45, "u_!"] & \derivator_{\cat F}(I) \arrow[l, swap, <-, "\simeq"] \arrow[r, <-, "\simeq"] & \underline{\GInj(\cat F^I)}  \ar[d, bend left, pos=0.45, "u_*"] \\
\underline{\GProj(\cat F^J)} & \derivator_{\cat F}(J) \arrow[u, "u^*"] \arrow[l, swap, <-, "\simeq"] \arrow[r, <-, "\simeq"] & \underline{\GInj(\cat F^J)}
\end{tikzcd}
\]
\end{theorem*}

To conclude the Introduction, let us also point out yet another application related to the recent paper~\cite{gao-kuelshammer-kvamme-psaroudakis-monoII} by Gao, K\"ulshammer, Kvamme and Psaroudakis. For a category $I$ freely generated by a finite acyclic quiver, the authors constructed a certain epivalence (i.e.\ a full, essentially surjective and isomorphism reflecting functor). In our context, this construction gives none other than the diagram functor (in the sense of~\cite[p.~323]{groth-derivators})
\[ \underline{\GProj(\cat F^I)} \cong \derivator_{\cat F}(I) \to \derivator_{\cat F}(e)^I \cong \underline{\cat F}^I, \]
which is known to be an epivalence for any strong and stable derivator. Indeed, for $I=[1]$ (i.e.\ the quiver $\bullet\!\to\!\bullet$), this is immediate from the definition of a strong derivator (see the axiom (Der5) in Section~\ref{section:derivators}), and for (necessarily finite direct) categories generated by other finite acyclic quivers, this fact follows by the argument for~\cite[Proposition~A.5]{keller-nicolas-weight-str}. So all in all, we have implicitly proved a variant of~\cite[Theorem~A]{gao-kuelshammer-kvamme-psaroudakis-monoII} for weakly idempotent complete Frobenius exact categories.

\begin{remark*}
Throughout the paper, we fix a base commutative ring $\basering k$. Everything will be $\basering k$-linear, unless otherwise specified. We will consistently use \emph{cohomological} notation, so that every complex of modules (also called dg-module) is a cochain complex with differentials of degree $+1$.

We will essentially disregard set-theoretic issues by implicitly fixing suitable Grothendieck universes.
\end{remark*}

\subsection*{Acknowledgement} 

We would like to thank Isaac Bird, John Bourke, Sondre Kvamme, Mattia Ornaghi and Sebastian Opper for helpful discussions and references.

\section{Derivators} \label{section:derivators}
This section aims to be a quick introduction on derivators. More information can be found in \cite{groth-derivators}. 

Let $\Cat$ be the $2$-category of small categories and $\CAT$ the $2$-category of large categories. 

\begin{definition}
    A $2$-functor $\derivator \colon \opp{\Cat} \to \CAT$ is called \emph{prederivator}. 
\end{definition} 

Prederivators form a $2$-category $\mathcal{PDER}$ whose morphisms are pseudonatural transformations and transformations of prederivators are modifications. Given a morphism of small categories $u\colon I \to J$, his image under a prederivator $\derivator$
\[ \derivator(u)=u^* \colon \derivator(J) \to \derivator(I),\] is called \emph{restriction} of $u$. Similarly, given two morphisms $u,v \colon I \to J$, the image of a natural transformation $\alpha: u \to v$ is denoted by 
\[ \derivator(\alpha)=\alpha^*\colon u^* \to v^*. \]
Let now $e$ be the terminal object in $\Cat$, i.e. the category with only one object and the identity morphism. We call $\derivator(e)$ the \emph{underlying category} of the derivator $\derivator$. Consider a small category $I$, we call \emph{evaluation} the restriction along the functor $i\colon e \to I$,  which is the only functor mapping the object of $e$ to the object $i \in I$. Moreover, for  a morphism  in $\derivator(I)$ \[f: X \to Y,\] we denote by \[i^*(f)=f_i \colon X_i \to Y_i\] its image under the restriction $i^*$. 

\noindent We now introduce two crucial functors for derivators and hence for this work. 
\begin{definition}
    Let $u\colon I \to J \in \Cat$ and consider his restriction $u^* \colon \derivator(J) \to \derivator(I)$. When it exists, we denote the left adjoint of $u^*$ by 
    \[ u_! \colon \derivator(I) \to \derivator(J)\] and we call it \emph{left Kan extension}. Similarly, when it exists, we denote by 
     \[ u_* \colon \derivator(I) \to \derivator(J)\] the right adjoint of $u^*$ and we call it \emph{right Kan extension}.
\end{definition}
Particular examples of Kan extensions arise when we consider the unique functor $\pi \colon I \to e$. Indeed, the left adjoint to the restriction $\pi^*$ is the \emph{colimit functor} and the right adjoint is the \emph{limit functor}.
To work with Kan extensions, let us cosider the \emph{slice squares}:

\begin{equation}\label{eq:Der4}
    \begin{tikzcd}
(u/j)  \arrow[r, "p"] \arrow[d, "\pi"'] & I \arrow[d, "u"] &  & (j/u) \arrow[r, "q"] \arrow[d, "\pi"'] & I \arrow[d, "u"] \\
e \arrow[r, "j"']                       & J                &  & e \arrow[r, "j"']                      & J               
\end{tikzcd}
\end{equation}

which come together with canonical transformations $up \stackrel{\alpha}{\longrightarrow} j\pi$ and $j\pi \stackrel{\beta}{\longrightarrow} up$, respectively. Here the \emph{slice category} $(u/j)$ consists of pairs $(i,f)$ where $f\colon u(i) \to j$ is a morphism in $J$. A morphism between two objects $(i,f)$ and $(i',f')$ is a morphism $i \to i'$ in $I$, making the obvious triangle commute in $J$. The functor $p \colon (u/j) \to I $ is the projection onto the first component. The \emph{coslice category} $(j/u)$ and the functor $q$ are defined dually.

\begin{definition} \label{definition:derivator}
    A prederivator $\derivator$ is a derivator if it satisfies the following axioms:
    \begin{itemize}
        \item[(Der1)] $\derivator\colon \opp{\Cat} \to \CAT$ sends products to coproducts. Namely, the canonical map \[\derivator(\coprod I_n) \to \prod\derivator(I_n)\] is an equivalence. In particular, $\derivator(\emptyset)$ is equivalent to the terminal category.
        \item[(Der2)] For any $I \in \Cat$, a morphism $f:X \to Y$ in $\derivator(I)$ is an isomorphism if and only if the morphism $f_i:X_i \to Y_i$ is an isomorphism for any $i \in I$.
        \item[(Der3)] For every functor $u\colon I \to J$, there exist both the left Kan extension $u_!$ and right Kan extension $u_*$ of the restriction $u^*$. 
        \item[(Der4)] For any functor $u:I \to J$ and any object $i \in I$, the canonical transformations
        \[ \pi_!p^* \stackrel{\eta}{\longrightarrow} \pi_!p^*u^*u_! \stackrel{\alpha^*}{\longrightarrow} \pi_!\pi^*i^*u_! \stackrel{\epsilon}{\longrightarrow} i^*u_! \] \[ b^*u_* \stackrel{\eta}{\longrightarrow} \pi_*\pi^*b^*u_* \stackrel{\beta^*}{\longrightarrow} \pi_*q^*u^*u_* \stackrel{\epsilon}{\longrightarrow} \pi_* q^* \]
        associated to the slice squares \eqref{eq:Der4} are isomorphisms. Here we denote by $\eta$ the adjunction units and by $\epsilon$ the counits.
    \end{itemize}
\end{definition}

Derivators forms a $2$-category $\mathcal{DER}$ where a morphism of derivators is a morphism between the underlying prederivators and a natural transformation between two morphisms is a modification.

\begin{example}\label{ex:der}
Let us list some of the main examples of derivators. 
\begin{enumerate} 
    \item Let $\cat B$ be a bicomplete category, the 2-functor 
    \[\derivator_{\cat B}\colon I \mapsto \cat B^I\]
    is what we call the \emph{represented derivator}. Here we can observe that the Kan extensions are the ordinary Kan extensions functors and the underlying category is $\cat B$ itself. 
    \item Let $\cat M$ be a Quillen model category with weak equivalences $\cat W$, by formally inverting the pointwise weak equivalences we get the \emph{homotopy derivator} (see \cite{cisinski-modelderivator})
      \[\derivator_{\cat M}\colon I \mapsto \cat M^I[(\cat W^I)^{-1}].\]
   In particular, let $R$ be a ring, the homotopy derivator $\derivator_R$ of the projective model structure on unbounded chain complexes over $R$, is the derivator which enhances the derived category $\dercomp(R)$, which is the localization of the category of chain complexes at quasi-isomorphisms. 
   \[\derivator_{R}\colon I \mapsto \dercomp(\Mod(R)^I)\]
\item Let $\cat C$ be a complete and cocomplete $\infty$-category, we can associate a derivator to $\cat C$ by considering the homotopy derivator of $\cat C$
   \[\derivator_{R}\colon I \mapsto \operatorname{Ho}(\cat C^{\operatorname{N}(I)})\]
   where $\operatorname{N}(I)$ is the nerve of $I$ and $\operatorname{Ho}(\cat C^{\operatorname{N}(I)})$ is the homotopy category. More information can be found in \cite{groth-ponto-shulman-mayervietoris}.
    \item Let $\derivator$ be a derivator, $I,J \in \Cat$, the functor
    \[\derivator^I\colon J \mapsto \derivator(I\times J) \] is a derivator which we call \emph{shifted derivator}.
    \item The \emph{opposite derivator} is defined as follows
    \[\opp{\derivator}\colon I \mapsto \opp{\derivator(\opp{I})}.\]
\end{enumerate}
\end{example}

\subsection{Stable derivators}
Given a derivator $\derivator$, in order to get a triangulated structure on its underlying category, we need $\derivator$ to be \emph{stable}. To define stable derivators let us first introduce \emph{pointed derivators}.

\begin{definition}
    A derivator $\derivator$ is pointed if $\derivator(e)$ admits a zero object. 
\end{definition}
Let $[1]$ be the poset $(0<1)$ considered as a category and take a commutative square $\square=[1]\times[1]$. There are two inclusion functors of the full subcategories obtained by removing the terminal and the initial object, respectively\[i_\llcorner\colon\llcorner\to\square, \hspace{1cm}i_\urcorner\colon\urcorner\to\square.\]
Their Kan extensions are fully faithful by the following proposition

\begin{proposition}[{\cite[Proposition~1.20]{groth-derivators}}]
    Kan extensions along fully faithful functors are fully faithful. 
\end{proposition}

There are two important classes of fully faithful functors, which allows a characterization of the essential images of their Kan extensions: the \emph{(co)sieves}. If $\derivator$ is a pointed derivator, a fully faithful functor $u\colon I \to J$ in $\Cat$ is called \emph{sieve} if, whenever we have a morphism $j \to u(i)$ in $J$, then $j$ lies in the image of $u$. Dually we can define \emph{cosieves}. 

\begin{proposition}[{\cite[Proposition~3.6]{groth-derivators}}]\label{proposition_(co)sieves}
    Let $\derivator$ be a pointed derivator. \begin{enumerate}
        \item Let $u \colon I \to J$ be a sieve, then the right Kan extension $u_*$ is fully faithful and $X \in \derivator(J)$ lies in the essential image of $u_*$ if and only if $X_j\cong 0$ for all $j \in J\setminus u(I)$. 
        \item Let $u \colon I \to J$ be a cosieve, then the left Kan extension $u_!$ is fully faithful and $X \in \derivator(J)$ lies in the essential image of $u_!$ if and only if $X_j\cong 0$ for all $j \in J\setminus u(I)$.
    \end{enumerate}
\end{proposition}

When this proposition applies we say that homotopy left Kan
extension along cosieves and homotopy right Kan extension along sieves are given by \emph{extension by
zero functors}.

\noindent Given a (pre)derivator $\derivator$, we write $X\in\derivator$ to indicate that there is a small category~$I$ such that $X\in\derivator(I)$.
In particular, a square $X \in \derivator^{\square}$ is called \emph{cartesian} if it lies in the essential image of $(i_\urcorner)_*$ and \emph{cocartesian} if it lies in the essential image of $(i_\llcorner)_!$. 
\begin{definition}
    A pointed derivator is \emph{stable} if the classes of cartesian squares and cocartesian squares coincide. These squares are then called \emph{bicartesian}.
\end{definition}

If a derivator $\derivator$ is stable then so are the shifted derivator $\derivator^I$ and the opposite $\opp{\derivator}$. Moreover, homotopy derivators of stable $\infty$-categories and stable model categories are stable.

We recall that a derivator is \emph{strong} if it satisfies one more axiom
 
\begin{itemize}
    \item [(Der$5$)] For any $I \in \Cat$, the induced functor  $\derivator(I \times [1]) \to \derivator(I)^{[1]}$ is full and essentially surjective.
\end{itemize}

The property of being strong allows to relate properties of stable derivators to the existence of structure on its values. This is showed, for example, in the following theorem. 

\begin{theorem}[{\cite[Theorem~4.16]{groth-derivators}}]
    Let $\derivator$ be a strong and stable derivator and let $I$ be a category. Then $\derivator(I)$ is a triangulated category. 
\end{theorem}

\section{Dg-categories and quasi-functors}

We start by reviewing in \S \ref{subsection:basicsdg} some basic definitions and results on dg-categories. We will also discuss in \S \ref{subsection:hokandg} \emph{homotopy Kan extensions} in the context of quasi-functors -- a feature of dg-category theory which we will be a key element of the main result. The results of this section are -- to the authors' knowledge -- well-known save for the part on homotopy Kan extensions, which nevertheless draws direct inspiration from the well-known extensions of dg-functors (see \cite{canonaco-stellari-internalhoms} \cite{keller-deriving-dgcat}).

\subsection{Basics} \label{subsection:basicsdg}

We assume the reader to be familiar with the basics of differential graded (dg-) categories. A basic reference is \cite{keller-dgcat}. Here, we recollect the definitions and results we will need, while establishing our notation.
\begin{definition} \label{definition:basicsdg}
A \emph{dg-category} $\cat A$ is a category enriched in complexes of $\basering k$-modules. More concretely, it is given by a set of objects $\Ob(\cat A)$ and complexes $\cat A(A,B)$ for any pair $A,B \in \Ob(\cat A)$, with associative and unital compositions compatible with gradings and differentials.

Given a dg-category $\cat A$, we have its \emph{opposite dg-category} $\opp{\cat A}$.

\emph{Dg-functors} are defined in the obvious way and they form a \emph{dg-category of dg-functors} $\Fundg(\cat A, \cat B)$. Moreover, there is a \emph{tensor product} of dg-categories $\cat A \otimes \cat B$ and a natural isomorphism
\begin{equation} \label{equation:dgtensorhom}
\Fundg(\cat A \otimes \cat B, \cat C) \cong \Fundg(\cat A, \Fundg(\cat B, \cat C)),
\end{equation}
for dg-categories $\cat A, \cat B, \cat C$.

By taking respectively $0$-cocycles and zeroth cohomology we may associate to any dg-category $\cat A$ ordinary categories $Z^0(\cat A)$ and $H^0(\cat A)$, called the \emph{underlying category} and the \emph{homotopy category}. Similarly, for any dg-functor $F \colon \cat A \to \cat B$, we have ordinary functors $Z^0(F) \colon Z^0(\cat A) \to Z^0(\cat B)$ and $H^0(F) \colon H^0(\cat A) \to H^0(\cat B)$. We will call a closed degree $0$ morphism $f \colon A \to B$ in $\cat A$ a \emph{(strict) isomorphism} if it is an isomorphism in $Z^0(\cat A)$; we will call it a \emph{homotopy equivalence} if it is an isomorphism in $H^0(\cat A)$.

A \emph{quasi-equivalence} of dg-categories is a dg-functor $F \colon \cat A \to \cat B$ such that for each pair of objects $A,B \in \cat A$ the chain map $F \colon \cat A(A,B) \to \cat B(F(A),F(B))$ is a quasi-isomorphism, and moreover $H^0(F)$ is essentially surjective. We will denote by $\kat{Hqe}$ the localization of the category of (small) dg-categories along quasi-equivalences, and we will say that two dg-categories are \emph{quasi-equivalent} if they are isomorphic in $\kat{Hqe}$.

There is a dg-category $\pshdg(\basering k)=\rdgm{\basering k}$ of dg-modules over our base ring $\basering k$. If $\cat A$ is a dg-category, we define the \emph{dg-category of right $\cat A$-dg-modules}
\begin{equation}
 \pshdg(\cat A) = \rdgm{\cat A} = \Fundg(\opp{\cat A}, \rdgm{\basering k}).
\end{equation}
Dually, we have the \emph{dg-category of left $\cat A$-dg-modules}:
\begin{equation*}
  \pshdg(\opp{\cat A})=  \ldgm{\cat A} = \Fundg(\cat A, \rdgm{\basering k}).
\end{equation*}
We see that left $\cat A$-dg-modules are just right $\opp{\cat A}$-dg-modules.

If $\cat B$ is another dg-category, we also define the \emph{dg-category of $\cat A$-$\cat B$-dg-bimodules}:
\begin{equation}
  \pshdg(\cat A, \cat B) = \dgm{\cat A}{\cat B} = \Fundg(\opp{\cat B} \otimes \cat A, \rdgm{\basering k}).
\end{equation}
Note that, thanks to \eqref{equation:dgtensorhom}, $\cat A$-$\cat B$-dg-bimodules may be identified with dg-functors $\cat A \to \rdgm{\cat B}$. Note moreover that $\cat A$-$\cat B$-dg-bimodules are just right $\cat B \otimes \opp{\cat A}$-dg-modules.

Sometimes, we will implicitly view our base ring $\basering k$ as a dg-category concentrated in degree $0$ with a single object. It is immediate to see that, for any dg-category $\cat A$, the $\cat A$-$\basering k$-dg-bimodules are just the left $\cat A$-dg-bimodules, and the $\basering k$-$\cat A$-dg-modules are just the right $\cat A$-dg-modules.
\end{definition}

\begin{notation}
We will use ``Einstein notation'' for dg-modules and dg-bimodules. Namely, if $F \in \dgm{\cat A}{\cat B}$, we will use the notation:
\begin{equation}
F_A^B = F(B,A) \in \rdgm{\basering k}.
\end{equation}
The lower index is covariant, the upper index is contravariant. In particular, if $F \in \rdgm{\cat A}$ is a right $\cat A$-dg-module, we will simply write $F^A = F(A)$; analogously, if $F \in \ldgm{\cat A}$, we will write $F_A = F(A)$.

Moreover, $F_A$ or $F_A^{-}$ denotes the right $\cat B$-dg-module
\begin{align*}
F_A &\in \rdgm{\cat B}, \\
B &\mapsto F_A^B,
\end{align*}
if $A \in \cat A$ is fixed, and similarly $F^B$ or $F^B_-$ denotes the left $\cat A$-dg-module
\begin{align*}
F^B & \in \ldgm{\cat A}, \\
A &\mapsto F_A^B ,
\end{align*}
if $B \in \cat B$ is fixed.
\end{notation}

\subsubsection*{Hom and tensor of dg-modules}
If $\cat A$ is a dg-category, we will use the notation
\begin{equation}
\Homdg_{\cat A}(-,-)
\end{equation}
to denote the complex of dg-natural transformations of right or left (depending on context) $\cat A$-dg-modules. If $\cat B$ is another dg-category, we will use the notation
\[
\Homdg_{\cat A,\cat B}(-,-)
\]
to denote the complex of dg-natural transformations of $\cat A$-$\cat B$-dg-bimodules.

We can also define a ``parametrized Hom'' as follows. Let $\cat A, \cat B, \cat C$ be dg-categories, and let $F \in \dgm{\cat A}{\cat C}$ and $G \in \dgm{\cat B}{\cat C}$. We may define
\begin{equation}
    \Homdg_{\cat C}(F,G) \in \dgm{\cat B}{\cat A}
\end{equation}
as follows:
\[
\Homdg_{\cat C}(F,G)_B^A = \Homdg_{\cat C}(F_A, G_B).
\]

Now, let $\cat A$ be a dg-category, and let $F \in \rdgm{\cat A}$ and $G \in \ldgm{\cat A}$ be respectively a right and a left $\cat A$-dg-module. We may define the \emph{tensor product} $F \otimes_{\cat A} G$ as the following $\basering k$-dg-module:
\begin{equation}
\begin{split}
    F \otimes_{\cat A} G = \coker \left(  \bigoplus_{A,B \in \cat A} F(B) \otimes \right.&\cat A(A,B)  \otimes G(A)  \left. \to \bigoplus_{C \in \cat A} F(C) \otimes G(C) \right), \\
    x \otimes f \otimes y &\mapsto F(f)(x) \otimes y - (-1)^{|x||f|} x \otimes G(f)(y),
\end{split}
\end{equation}
where $|a|$ denotes the degree of an homogeneous graded symbol $a$.

We can also define a ``parametrized tensor'', as follows. Let $\cat A, \cat B, \cat C$ be dg-categories, and let $F \in \dgm{\cat A}{\cat B}$ and $G \in \dgm{\cat B}{\cat C}$ be respectively an $\cat A$-$\cat B$-dg-bimodule and a $\cat B$-$\cat C$-dg-bimodule. Then, we may define:
\begin{equation}
    F \otimes_{\cat B} G \in \dgm{\cat A}{\cat C}
\end{equation}
as follows:
\[
(F \otimes_{\cat B} G)_A^C = F_A \otimes_{\cat B} G^C.
\]

We have a \emph{tensor-hom adjunction} involving the parametrized Hom and tensor, which we can express as the following natural isomorphisms of complexes:
\begin{equation} \label{equation:tensorhomdgm}
\begin{split}
   \Homdg_{\cat A, \cat C}(F \otimes_{\cat B} G, H) & \cong \Homdg_{\cat A,\cat B}(F, \Homdg_{\cat C}(G,H)) \\
   & \cong \Homdg_{\cat B, \cat C}(G, \Homdg_{\cat A}(F,H)).
\end{split}
\end{equation}
for $F \in \dgm{\cat A}{\cat B}, G \in \dgm{\cat B}{\cat C}, H \in \dgm{\cat A}{\cat C}$. A ``derived version'' of this isomorphism will be investigated later on, in particular in \S \ref{subsection:hokandg} while dealing with homotopy Kan extensions.

\subsubsection*{Yoneda lemma}
For any dg-category $\cat A$, we have the \emph{diagonal bimodule}
\begin{equation} \label{equation:diagonalbimod}
\begin{split}
    h_{\cat A} &= \cat A  \in \dgm{\cat A}{\cat A}, \\
    (h_{\cat A})_A^{A'} &= h_A^{A'} =\cat A_A^{A'} = \cat A(A',A).
\end{split}
\end{equation}
The right and left $\cat A$-modules of the form $h_A =\cat A(-,A)$ and $h^A = \cat A(A,-)$ are respectively called \emph{represented} and \emph{corepresented} by $A \in \cat A$.

The \emph{Yoneda lemma} holds, for a given $F \in \rdgm{\cat A}$:
\begin{equation} \label{equation:yoneda}
    \Homdg_{\cat A}(h_A, F) \cong F(A) = F^A.
\end{equation}
This yields the \emph{Yoneda embedding}, namely, the following fully faithful dg-functor:
\begin{equation}
    \begin{split}
        \cat A &\hookrightarrow \rdgm{\cat A}, \\
        A & \mapsto h_A = \cat A(-,A). 
    \end{split}
\end{equation}

The Yoneda lemma has a ``dual'' version involving the tensor product. Namely, if $A \in \cat A$ and $F \in \ldgm{\cat A}$, we have a natural isomorphism:
\begin{equation} \label{equation:coyoneda}
    h_A \otimes_{\cat A} F \cong F(A) = F_A.
\end{equation}

\subsubsection*{Resolutions}

\emph{Quasi-isomorphisms} of (right) dg-modules over a dg-category $\cat A$ are simply defined as the termwise quasi-isomorphisms; \emph{acyclic} (right) dg-modules are just the termwise acyclic dg-modules. Quasi-isomorphisms of dg-modules and quasi-equivalences of dg-categories will not in general be preserved by the typical constructions we perform, such as tensor products or Homs. A typical strategy to achieve \emph{derived} versions of such constructions (namely, versions which are stable under quasi-isomorphism or quasi-equivalence) is to use \emph{resolutions}. Such resolutions exist for both dg-categories and dg-modules.

\begin{definition}
    Let $\cat A$ be a dg-category. We say that a (right) $\cat A$-dg-module $F$ is \emph{h-projective} if $\Homdg_{\cat A}(F,-)$ preserves acyclic (right) $\cat A$-dg-modules. We say that it is \emph{h-injective} if $\Homdg_{\cat A}(-,F)$ preserves acyclic (right) $\cat A$-dg-modules. We say that it is \emph{h-flat} if $F \otimes_{\cat A} -$ preserves acyclic (left) $\cat A$-dg-modules.

    An \emph{h-projective resolution} of $F$ is a quasi-isomorphism $Q(F) \to F$ with $Q(F)$ h-projective. An \emph{h-injective resolution} of $F$ is a quasi-isomorphism $F \to R(F)$ with $R(F)$ h-injective. An \emph{h-flat resolution} of $F$ is a quasi-isomorphism $S(F) \to F$ with $S(F)$ h-flat. 
\end{definition}
\begin{remark}
Obviously, if a given dg-module $F$ is itself h-projective, h-injective or h-flat, the identity morphisms can be taken as h-projective, h-injective or h-flat resolutions.
\end{remark}

We know that h-projective, h-injective and h-flat resolutions of dg-modules always exist. Moreover, h-projectivity implies h-flatness. As anticipated above, we may also resolve dg-categories:
\begin{definition}
    Let $\cat A$ be a dg-category. We say that $\cat A$ is \emph{h-projective} if the hom complexes $\cat A(A,B)$ are h-projective $\basering k$-dg-modules for all $A,B \in \cat A$. We say that $\cat A$ is \emph{h-flat} if the hom complexes $\cat A(A,B)$ are h-flat $\basering k$-dg-modules. for all $A,B \in \cat A$.
\end{definition}
We know that h-projective and h-flat resolutions of dg-categories always exist (see \cite[Lemma B.5]{drinfeld-dgquotients}). Clearly, h-projectivity implies h-flatness also for dg-categories. 

We now list some well-known useful results involving \emph{termwise} h-projectivity and h-injectivity.

\begin{lemma}
    Let $\cat A$ and $\cat B$ be dg-categories, and let $F \in \dgm{\cat A}{\cat B}$ be an $\cat A$-$\cat B$-dg-bimodule. Assume that $\cat A$ is h-projective and $F$ is h-projective. Then, $F_A$ is h-projective as a right $\cat B$-dg-bimodule for all $A \in \cat A$.

    Moreover, assume that $\cat A$ is h-flat and $F$ is h-injective. Then, $F_A$ is h-injective as a right $\cat B$-dg-bimodule for all $A \in \cat A$.

    Furthermore, assume that $\cat A$ is h-flat and $F$ is h-flat. Then, $F_A$ is h-flat as a right $\cat B$-dg-bimodule for all $A \in \cat A$.
\end{lemma}
\begin{proof}
    Let us check the first claim. We fix an acyclic dg-module $X \in \rdgm{\cat B}$ and $A \in \cat A$. Then, using the ``dual'' Yoneda lemma \eqref{equation:coyoneda} and the tensor-hom adjunction \eqref{equation:tensorhomdgm}, we have isomorphisms:
    \begin{align*}
        \Homdg_{\cat B}(F_A, X) & \cong \Homdg_{\cat B}(h_A \otimes_{\cat A} F, X) \\
        & \cong \Homdg_{\cat A, \cat B}(F, \Homdg_{\basering k}(h_A, X)).
    \end{align*}
    The $\cat A$-$\cat B$-dg-bimodule $\Homdg_{\basering k}(h_A,X)$ is acyclic, because $\cat A$ is h-projective and $X$ is acyclic. Since $F$ is h-projective, we conclude.

    Let us now prove the second claim. Again, we fix an acyclic dg-module $X \in \rdgm{\cat B}$ and $A \in \cat A$. Then, using the Yoneda lemma  \eqref{equation:yoneda} and the tensor-hom adjunction \eqref{equation:tensorhomdgm}, we have isomorphisms:
    \begin{align*}
        \Homdg_{\cat B}(X,F_A) &\cong \Homdg_{\cat B}(X, \Homdg_{\cat A}(h^A, F)) \\
        & \cong \Homdg_{\cat A, \cat B}(X \otimes_{\basering k} h^A, F).
    \end{align*}
    The $\cat A$-$\cat B$-dg-bimodule $X \otimes_{\basering k} h^A$ is acyclic, because $\cat A$ is h-flat and $X$ is acyclic. Since $F$ is h-injective, we conclude.

    Finally, let us prove the third claim. We fix an acyclic left $\cat B$-dg-module $X$ and an object $A \in \cat A$. We define the following $\cat B$-$\cat A$-dg-bimodule, which we may view as a left $\cat B \otimes \opp{\cat A}$-dg-module:
    \[
    \widetilde{X}^{A'}_B = h_A^{A'} \otimes_{\basering k} X_B.
    \]
    Since $\cat A$ is h-flat, $h_A^{A'} = \cat A(A',A)$ is h-flat, hence $\widetilde{X}$ is acyclic. Next, one can directly prove that:
    \[
    F_A \otimes_{\cat B} X \cong F \otimes_{\cat B \otimes \opp{\cat A}} \widetilde{X}.
    \]
    Since $F$ is h-flat, we conclude.
\end{proof}

The above result ensures the existence of termwise h-projective, h-flat, h-injective resolutions:
\begin{corollary} \label{corollary:termwiseres}
Let $\cat A, \cat B$ be dg-categories and let $F \in \dgm{\cat A}{\cat B}$ be an $\cat A$-$\cat B$-dg-bimodule. If $\cat A$ is h-projective, we have a quasi-isomorphism $Q(F) \to F$ of $\cat A$-$\cat B$-dg-bimodules, such that $Q(F)_A$ is h-projective (hence also h-flat) for all $A \in \cat A$.

Analogously, if $\cat A$ is h-flat, we have a quasi-isomorphism $F \to R(F)$ of $\cat A$-$\cat B$-dg-bimodules, such that $R(F)_A$ is h-injective for all $A \in \cat A$.

Analogously, if $\cat A$ is h-flat, we have a quasi-isomorphism $S(F) \to F$ of $\cat A$-$\cat B$-dg-bimodules, such that $S(F)_A$ is h-flat for all $A \in \cat A$.
\end{corollary}
\begin{remark}
Obviously, if a given $\cat A$-$\cat B$-dg-bimodule $F$ is itself termwise h-projective, h-injective or h-flat, the identity morphisms can be taken as termwise h-projective, h-injective or h-flat resolutions.
\end{remark}
\subsubsection*{Derived categories}
By taking the localization of $H^0(\rdgm{\cat A})$ along quasi-isomorphisms (or, equivalently, the Verdier quotient by acyclic dg-modules) we obtain the \emph{derived category} $\dercomp(\cat A)$ of $\cat A$. For us, $\dercomp(\cat A)$ has the same objects as $\rdgm{\cat A}$. It is a triangulated category.

We say that a dg-module $F \in \dercomp(\cat A)$ is \emph{quasi-representable} if there is an isomorphism in $\dercomp(\cat A)$:
\begin{equation}
\cat A(-,A) \xrightarrow{\approx} F.
\end{equation}

Derived categories can be defined also in connection with dg-bimodules. If $\cat A$ and $\cat B$ are dg-categories, we denote by
\begin{equation} \label{equation:derivedtensor}
    \dercomp(\cat A, \cat B)
\end{equation}
the localization of $H^0(\pshdg(\widetilde{\cat A}, \cat B))$ along quasi-isomorphisms, where $\widetilde{\cat A} \to \cat A$ is an h-flat resolution of $\cat A$. Up to equivalence, it can be also defined as the localization of $H^0(\pshdg(\cat A, \widetilde{\cat B}))$ taking an h-flat resolution $\widetilde{\cat B}$ of $\cat B$, or even as the localization of $H^0(\pshdg(\widetilde{\cat A}, \widetilde{\cat B}))$, taking h-flat resolutions of both $\cat A$ and $\cat B$. This definition can be also expressed by saying that
\[
\dercomp(\cat A, \cat B) = \dercomp(\cat B \lotimes \opp{\cat A}),
\]
where $\cat B \lotimes \opp{\cat A}$ is the \emph{derived tensor product} of $\cat B$ and $\opp{\cat A}$, obtained from the tensor product by taking h-flat resolutions in either variable. This makes sure that our constructions are preserved by quasi-equivalences.

Given any dg-category $\cat A$, we may also define its \emph{derived dg-category} $\dercompdg(\cat A)$ as the full dg-subcategory of $\pshdg(\cat A)$ spanned by all h-projective dg-modules, or (equivalently, up to quasi-equivalence) as the full dg-subcategory of $\pshdg(\cat A)$ spanned by all h-injective dg-modules. We shall nevertheless make the following identification:
\begin{equation} \label{equation:deriveddgcat_H0}
    H^0(\dercompdg(\cat A)) = \dercomp(\cat A).
\end{equation}

\subsubsection*{Derived hom and tensor of dg-modules}
The hom complex and the tensor product of dg-modules can be derived, yielding constructions that are preserved by quasi-isomorphisms. Such derived hom and tensors are well defined just up to quasi-isomorphism. The general strategy to define them is simply taking the ordinary hom and tensor after applying suitable h-projective, h-injective or h-flat resolutions.

Let $\cat A$ be a dg-category, and let $F,G \in \rdgm{\cat A}$ be right $\cat A$-dg-modules. We define the \emph{derived hom} $\RHom_{\cat A}(F,G) \in \dercomp(\basering k)$ (up to isomorphism in that category) as follows:
\begin{equation}
    \begin{split}
   \RHom_{\cat A}(F,G) &= \Homdg_{\cat A}(Q(F),G) \\
   &= \Homdg_{\cat A}(F,R(G)) \\
   &= \Homdg_{\cat A}(Q(F),R(G)),
    \end{split}
\end{equation}
where $Q(F) \to F$ and $G \to R(G)$ are respectively h-projective and h-injective resolutions. We remark that the zeroth cohomology of the derived hom yields the hom in the derived category, namely:
\begin{equation} \label{equation:derivedhom_zerothcohomology}
    H^0(\RHom_{\cat A}(F,G)) \cong \dercomp(\cat A)(F,G),
\end{equation}
for $F, G \in \rdgm{\cat A}$. Moreover, up to a potential little abuse of notation, we will sometimes view the derived homs as the hom-complexes of the derived dg-category $\dercompdg(\cat A)$:
\begin{equation} \label{equation:dercatdg_derivedhom}
    \dercompdg(\cat A)(F,G) = \RHom_{\cat A}(F,G),
\end{equation}
even if $F$ and $G$ are not both h-projective or both h-injective.

We can also define a ``parametrized'' derived hom, using termwise resolutions (see Corollary \ref{corollary:termwiseres}). Hence, we will work with h-projective dg-categories -- this can be always done without loss of generality up to quasi-equivalence. Namely, let $\cat A, \cat B, \cat C$ be h-projective dg-categories and let $F \in \dgm{\cat A}{\cat C}$ and $G \in \dgm{\cat B}{\cat C}$ be dg-bimodules. We define $\RHom_{\cat C}(F,G) \in \dercomp(\cat A,\cat B)$ (up to isomorphism in that category) as follows:
\begin{equation} \label{equation:derivedhomdgmod}
    \begin{split}
        \RHom_{\cat C}(F,G) &= \Homdg_{\cat C}(Q(F),G) \\
        &= \Homdg_{\cat C}(F,R(G)) \\
        &= \Homdg_{\cat C}(Q(F),R(G)),
    \end{split}
\end{equation}
where $Q(F) \to F$ and $G \to R(G)$ are respectively termwise h-projective and h-injective resolutions -- in this case, such that $Q(F)_A \in \rdgm{\cat C}$ is h-projective for all $A \in \cat A$ and $R(G)_B \in \rdgm{\cat C}$ is h-injective for all $B \in \cat B$. In the end, we obtain a functor:
\begin{equation}
\RHom_{\cat C}(-,-) \colon \dercomp(\cat A, \cat C) \otimes \dercomp(\cat B,\cat C) \to \dercomp(\cat A,\cat B). 
\end{equation}
To be more precise, this is an isomorphism class of functors, depending on which h-projective and/or h-injective resolutions one chooses in the (equivalent) definitions given in \eqref{equation:derivedhomdgmod}.

Let us now describe the derived tensor product. Let $\cat A$ be a dg-category and let $F \in \rdgm{\cat A}$ and $G \in \ldgm{\cat A}$ be respectively a right and a left $\cat A$-dg-modules. We define the \emph{derived tensor product} $F \lotimes_{\cat A} G \in \dercomp(\basering k)$ (up to isomorphism in that category) as follows:
\begin{equation}
    \begin{split}
    F \lotimes_{\cat A} G &= S(F) \otimes_{\cat A} G \\
    &= F \otimes_{\cat A} S(G) \\
    &= S(F) \otimes_{\cat A} S(G),
    \end{split}
\end{equation}
where $S(F) \to F$ and $S(G) \to G$ are h-flat resolutions.

We can also define a ``parametrized'' derived tensor product, using termwise resolutions (see Corollary \ref{corollary:termwiseres}). Hence, we will work in this case with h-flat dg-categories -- this can be always done without loss of generality up to quasi-equivalence. Namely, let $\cat A, \cat B, \cat C$ be h-flat dg-categories and let $F \in \dgm{\cat A}{\cat B}, G \in \dgm{\cat B}{\cat C}$. We define the derived tensor product $F \lotimes_{\cat B} G \in \dercomp(\cat A, \cat C)$ (up to isomorphism in that category) as follows:
\begin{equation} \label{equation:derivedtensordgmod}
    \begin{split}
        F \lotimes_{\cat B} G &= S(F) \otimes_{\cat B} G \\
        &= F \otimes_{\cat B} S(G) \\
        &= S(F) \otimes_{\cat B} S(G),
    \end{split}
\end{equation}
where $S(F) \to F$ and $S(G) \to G$ are termwise h-flat resolutions, such that $S(F)_A \in \rdgm{\cat B}$ is h-flat for all $A \in \cat A$ and such that $S(G)^C \in \ldgm{\cat B}$ is h-flat for all $C \in \cat C$. In the end, we obtain a functor:
\begin{equation}
    - \lotimes_{\cat B} - \colon \dercomp(\cat A, \cat B) \otimes \dercomp(\cat B, \cat C) \to \dercomp(\cat A, \cat C).
\end{equation}
To be more precise, this is an isomorphism class of functors, depending on which h-flat resolutions one chooses in the (equivalent) definitions given in \eqref{equation:derivedtensordgmod}.

Taking h-projective dg-categories $\cat A,\cat B$ and $\cat C$, we have a \emph{derived tensor-hom adjunction} involving the parametrized derived hom and tensor, which we can express as the following isomorphisms in $\dercomp(\basering k)$:
\begin{equation} \label{equation:derivedtensorhomdgm}
\begin{split}
   \RHom_{\cat A, \cat C}(F \lotimes_{\cat B} G, H) & \cong \RHom_{\cat A,\cat B}(F, \RHom_{\cat C}(G,H)) \\
   & \cong \RHom_{\cat B, \cat C}(G, \RHom_{\cat A}(F,H)),
\end{split}
\end{equation}
natural in $F \in \dercomp(\cat A,\cat B), G \in \dercomp(\cat B, \cat C), H \in \dercomp(\cat A,\cat C)$. This is obtained directly from the ordinary tensor-hom adjunction \eqref{equation:tensorhomdgm} by introducing the suitable (termwise) resolutions of dg-modules.

\subsubsection*{Derived Yoneda lemma}

Let $\cat A$ be a dg-category and let $A \in \cat A$. Using the Yoneda lemma \eqref{equation:yoneda}, we can immediately show that $h_A$ is an h-projective right $\cat A$-dg-module. Hence, we can write:
\begin{equation} \label{equation:derivedYoneda}
\RHom(h_A,F) \cong \Homdg(h_A,F) \cong F^A = F(A),
\end{equation}
for any $F \in \rdgm{\cat A}$. This can be interpreted as a \emph{derived Yoneda lemma}. Taking zeroth cohomology and recalling \eqref{equation:derivedhom_zerothcohomology}, we get an isomorphism
\begin{equation}
    \dercomp(\cat A)(h_A,F) \cong H^0(F(A)),
\end{equation}
sometimes also called ``derived Yoneda lemma'', and from that the \emph{derived Yoneda embedding}
\begin{equation} \label{equation:derivedyonedaembedding}
    \begin{split}
        H^0(\cat A) &\hookrightarrow \dercomp(\cat A), \\
        A & \mapsto h_A = \cat A(-,A). 
    \end{split}
\end{equation}

Since $h_A$ is h-projective, it is in particular h-flat. Hence, using the ``dual'' Yoneda lemma \eqref{equation:coyoneda} we can write:
\begin{equation} \label{equation:derivedcoYoneda}
h_A \lotimes_{\cat A} F \cong h_A \otimes_{\cat A} F \cong F_A,
\end{equation}
for $F \in \ldgm{\cat A}$. This can be interpreted as a \emph{derived ``dual'' Yoneda lemma}.

\subsection{Pretriangulated dg-categories}\label{sec:pretriang-dg}
Let $\cat A$ be a dg-category. The dg-category of right $\cat A$-dg-modules $\rdgm{\cat A}$ has \emph{shifts} of objects and \emph{mapping cones} of closed degree $0$ morphisms, computed as termwise shifts and mapping cones. This allows us to give a definition of shifts and cones inside the dg-category $\cat A$, by means of a universal property.
\begin{definition} \label{definition:shift_cones_qis}
    Let $\cat A$ be a dg-category, $A \in \cat A$ be an object and $n \in \mathbb Z$. An \emph{$n$-shift} of $A$ is an object $A[n] \in \cat A$ together with an isomorphism
    \begin{equation} \label{equation:shift_quasirep}
    \cat A(-,A[n]) \xrightarrow{\sim} \cat A(-,A)[n]
    \end{equation}
    in $\dercomp(\cat A)$, where $A(-,A)[n]$ is the $n$-shift of the $\cat A$-dg-module $\cat A(-,A)$, computed termwise. Abusing notation, we will call any object $A[n] \in \cat A$ \emph{the} $n$-shift of $A$, uniquely determined up to isomorphism in $H^0(\cat A)$.

    Next, let $f \colon A \to B$ be a closed degree $0$ morphism in $\cat A$. A \emph{cone} of $f$ is an object $\cone(f) \in \cat A$ together with an isomorphism
    \begin{equation} \label{equation:cone_quasirep}
    \cat A(-,\cone(f)) \xrightarrow{\sim} \cat \cone(f_* \colon \cat A(-,A) \to \cat A(-,B))
    \end{equation}
    in $\dercomp(\cat A)$, where $\cone(f_*)$ is the termwise mapping cone of the induced morphism of $\cat A$-dg-modules $f_* \colon \cat A(-,A) \to \cat A(-,B)$. Abusing notation, we will call any object $\cone(f) \in \cat A$ \emph{the} cone of $f$, uniquely determined up to isomorphism in $H^0(\cat A)$.
\end{definition}

\begin{definition}[cf. {\cite{bondal-kapranov-enhanced}}] \label{definition:pretriangulated_dgcat}
    Let $\cat A$ be a (non empty) dg-category. We say that $\cat A$ is \emph{pretriangulated} if for any object $A \in \cat A$ and $n \in \mathbb Z$, the $n$-shift $A[n]$ exists in $\cat A$, and for any closed degree $0$ morphism $f$ in $\cat A$, the cone $\cone(f)$ exists in $\cat A$.
\end{definition}
\begin{remark}
    If $\cat A$ is pretriangulated, it has a zero object up to homotopy, namely, an object $0 \in \cat A$ such that
    \begin{equation} \label{equation:zeroobject}
    \cat A(A,0) \cong \cat A(0,A) \cong 0
    \end{equation}
    in $\dercomp(\basering k)$, for all objects $A \in \cat A$. Thanks to shifts, this is actually equivalent to requiring that $H^0(\cat A)$ has a zero object. Such object can be described, up to isomorphism in $H^0(\cat A)$, as the cone of any identity morphism.
\end{remark}

\begin{remark} \label{remark:pretriangles}
    Notice that we have a sequence of closed degree $0$ morphisms:
    \begin{equation} \label{equation:pretriangle_repr}
        A \xrightarrow{f} B \xrightarrow{j} \cone(f) \xrightarrow{p} A[1],
    \end{equation}
    induced by the morphisms of $\cat A$-dg-modules:
    \begin{equation} \label{equation:pretriangle_dgm}
        \cat A(-,A )\xrightarrow{f_*} \cat A(-,B) \to \cone(f) \to \cat A(-,A)[1],
    \end{equation}
    where the unnamed morphisms are given by the natural inclusion and projection morphisms. Both \eqref{equation:pretriangle_repr} and \eqref{equation:pretriangle_dgm} are called \emph{pretriangles}.
\end{remark}
The homotopy category $H^0(\cat A)$ of a pretriangulated dg-category $\cat A$ has a ``canonical'' structure of triangulated category. The crucial property of pretriangulated dg-categories is that, unlike triangulated categories, they have \emph{functorial shifts and cones}.
\begin{remark} \label{remark:opposite_pretriangulated}
    Let $\cat A$ be a dg-category. It can be proved that $\cat A$ is pretriangulated if and only if $\opp{\cat A}$ is pretriangulated. Indeed:
    \begin{itemize}
        \item If $A[n]$ is a $n$-shift of $A$ in $\cat A$, then $A[-n] \in \cat A$ is an object representing the $n$-shift of $A$ in $\opp{\cat A}$.
        \item If $\cone(f)$ is a cone of a closed degree $0$ morphism $f$ in $\cat A$, then $\cone(f)[-1] \in \cat A$ is an object representing the cone of the corresponding morphism $\opp{f}$ in $\opp{\cat A}$.
    \end{itemize}
\end{remark}
\begin{remark} \label{remark:strongly_pretriangulated}
    We may change the above Definition \ref{definition:shift_cones_qis} by requiring \eqref{equation:shift_quasirep} and \eqref{equation:cone_quasirep} be isomorphisms in $Z^0(\pshdg(\cat A))$ instead of $\dercomp(\cat A)$. We would then define ``strict'' notions of shifts and cones in a dg-category. If a dg-category $\cat A$ has such ``strict'' shifts and cones, we say that it is \emph{strongly pretriangulated}. An obvious example of strongly pretriangulated dg-category is given by the dg-category of $\cat A$-dg-modules $\pshdg(\cat A)$, for any given dg-category $\cat A$.

    If a dg-category is pretriangulated, it can be proved that it is quasi-equivalent to a strongly pretriangulated dg-category. Indeed, we can introduce the \emph{pretriangulated hull} $\pretr(\cat A)$ of any dg-category $\cat A$ as the closure of $\cat A$ inside $\pshdg(\cat A)$ (via the Yoneda embedding) under shifts and mapping cones, up to isomorphism in $Z^0(\pshdg(\cat A))$ (hence, ``strict'' shifts and cones). Then, we can easily check that $\cat A$ is pretriangulated if and only if the embedding $\cat A \hookrightarrow \pretr(\cat A)$ is a quasi-equivalence, and that $\cat A$ is strongly pretriangulated if and onyl if $\cat A \hookrightarrow \pretr(\cat A)$ is a (strict) dg-equivalence. Since $\pretr(\cat A)$ is strongly pretriangulated by construction, we immediately prove the above claim.
\end{remark}

\subsection{Quasi-functors}
Using dg-bimodules, we can describe the suitable ``homotopically coherent'' morphisms between dg-categories, vastly generalizing dg-functors. Such morphisms are called \emph{quasi-functors} and will be the main character of this paper.
\begin{definition}
    Let $\cat A$ and $\cat B$ be dg-categories, and assume that either $\cat A$ or $\cat B$ is h-flat. A \emph{quasi-functor} $F \colon \cat A \to \cat B$ is an $\cat A$-$\cat B$-dg-bimodule $F \in \dercomp(\cat A, \cat B)$, such that for all $A \in \cat A$, there is an isomorphism
    \[
    F_A \cong \cat B(-,\Phi_F(A))
    \]
    in $\dercomp(\cat B)$, for some $\Phi_F(A) \in \cat B$.

    We denote by $\QFun(\cat A, \cat B)$ the full subcategory of $\dercomp(\cat A, \cat B)$ spanned by quasi-functors. We will say that two quasi-functors are \emph{isomorphic} if they are isomorphic in $\QFun(\cat A, \cat B)$ and typically use the symbol $\cong$.
\end{definition}
\begin{remark}
    We have to make sure that the category $\QFun(\cat A, \cat B)$ is independent (up to equivalence) of any choice of h-flat resolution of $\cat A$ or $\cat B$. More precisely: if $\cat A$ and $\cat B$ are any dg-categories, and $\widetilde{\cat A} \to \cat A$, $\widetilde{\cat B} \to \cat B$ are h-flat resolutions, we have to make sure that
    \[
    \QFun(\widetilde{\cat A}, \cat B) \cong \QFun(\cat A, \widetilde{\cat B}).
    \]
    We can actually show that there are equivalences:
    \[
    \QFun(\widetilde{\cat A}, \cat B) \cong \QFun(\widetilde{\cat A}, \widetilde{\cat B}) \cong \QFun(\cat A, \widetilde{\cat B}),
    \]
    induced by the given h-flat resolutions. The first equivalence can be proved adapting \cite[Proposition 3.8 (2)]{canonaco-stellari-internalhoms}, whereas the second equivalence is dealt with in detail in \cite[Proposition 2.3.1]{genovese-lowen-vdb-tstruct-deform}.
\end{remark}
\begin{remark}
    The derived Yoneda embedding \eqref{equation:derivedyonedaembedding} can be enhanced to a quasi-functor, called the \emph{dg-derived Yoneda embedding}:
    \begin{equation} \label{equation:deriveddgyonedaembedding}
        \cat A  \hookrightarrow \dercompdg(\cat A).
    \end{equation}
    If we identify $\dercompdg(\cat A)$ with the full dg-subcategory of $\pshdg(\cat A)$ spanned by h-projective dg-modules, the above dg-derived Yoneda embedding is described by the usual dg-Yoneda embedding, since any representable dg-module $\cat A(-,A)$ is h-projective.
\end{remark}

\begin{remark} \label{remark:RHom_dgcat}
    There is actually a \emph{dg-category of quasi-functors} $\RHom(\cat A, \cat B)$ for given dg-categories $\cat A$ and $\cat B$, well defined up to quasi-equivalence, such that there is an equivalence $H^0(\RHom(\cat A, \cat B)) \cong \QFun(\cat A, \cat B)$. It is defined as the full dg-subcategory of the derived dg-category $\dercompdg(\cat A, \cat B)$ of $\cat A$-$\cat B$-dg-bimodules spanned by the quasi-functors. In particular, if $F, G \in \RHom(\cat A, \cat B)$, we may identify:
    \[
    \RHom(\cat A, \cat B)(F,G) = \RHom_{\cat A, \cat B}(F,G),
    \]
    and
    \[
    \QFun(\cat A, \cat B)(F,G)= H^0(\RHom_{\cat A, \cat B}(F,G)).
    \]
    This is the \emph{internal hom in the homotopy category of dg-categories} $\Hqe$ \cite{toen-dgcat-invmath} \cite{canonaco-stellari-internalhoms}. In particular, we have an isomorphism in $\kat{Hqe}$:
    \begin{equation} \label{equation:RHom_dgcat}
        \RHom(\cat A \lotimes \cat B, \cat C) \cong \RHom(\cat A, \RHom(\cat B, \cat C)),
    \end{equation}
    for dg-categories $\cat A, \cat B$ and $\cat C$. The dg-category $\RHom(\cat A, \cat B)$ is pretriangulated if $\cat B$ is pretriangulated. 
    
    We will almost never need this higher structure in our paper and we will mostly concentrate just on the (ordinary) category of quasi-functors $\QFun(\cat A, \cat B)$. Still, we remark that we have an equivalence:
    \begin{equation} \label{equation:qfun_dercatdg_bimod}
        \QFun(\cat A, \dercompdg(\cat B)) \cong \dercomp(\cat A, \cat B),
    \end{equation}
    obtained by restricting an $\cat A$-$\dercompdg(\cat B)$-dg-bimodule along the dg-derived Yoneda embedding $\cat B \hookrightarrow \dercompdg(\cat B)$ (see \eqref{equation:deriveddgyonedaembedding}). In other words, quasi functors $\cat A \to \dercompdg(\cat B)$ can be viewed as $\cat A$-$\cat B$-dg-bimodules, at least from a ``derived'' point of view. In particular, we have:
    \begin{equation} \label{equation:qfun_dercatdg_dercat}
    \begin{split}
        \QFun(\cat A, \dercompdg(\basering k)) & \cong \dercomp(\opp{\cat A}), \\
        \QFun(\opp{\cat A}, \dercompdg(\basering k)) & \cong \dercomp(\cat A).
    \end{split}
    \end{equation}
\end{remark}

Dg-functors can be viewed as quasi-functors, as follows. If $f \colon \cat A \to \cat B$ is a dg-functor, we have an associated $\cat A$-$\cat B$-dg-bimodule $h_f \in \dgm{\cat A}{\cat B}$:
\begin{equation}
    (h_f)_A^B = \cat B(B,f(A)).
\end{equation}
This is trivially a quasi-functor. The dg-bimodule (and quasi-functor) associated to the identity $1_{\cat A} \colon \cat A \to \cat A$ is the diagonal bimodule  $h_{\cat A} = \cat A \in \dgm{\cat A}{\cat A}$ defined in \eqref{equation:diagonalbimod}:
\begin{equation}
    (h_{\cat A})_A^{A'} = \cat A_{A}^{A'} = \cat A(A',A).
\end{equation}

Quasi-functors can be \emph{composed} using the (parametrized) derived tensor product \eqref{equation:derivedhomdgmod}. Indeed, given quasi-functors $F \colon \cat A \to \cat B$ and $G \colon \cat B \to \cat C$ (with suitable h-flatness assumptions on $\cat A, \cat B$ and $\cat C$), we can define
\begin{equation}
G \circ F = GF = F \lotimes_{\cat B} G \colon \cat A \to \cat C.
\end{equation}
We can easily show that $GF$ is indeed a quasi-functor, and this composition yields a functor
\[
- \circ - \colon \QFun(\cat B, \cat C) \otimes \QFun(\cat A, \cat B) \to \QFun(\cat A, \cat C).
\]
The composition is associative and unital up to isomorphism of quasi-functors, the units being the diagonal bimodules \eqref{equation:diagonalbimod}.

It is worth observing that, given a dg-functor $f \colon \cat A' \to \cat A$ and a quasi-functor $F \colon \cat A \to \cat B$, the composition $F \circ f$ can be written as follows:
\begin{equation} \label{equation:quasifunctor_compose_dgfunctor}
\begin{split}
    F \circ f &= h_f \lotimes_{\cat A} F \\
    &\cong h_f \otimes_{\cat A} F \\
    &\cong F_f
\end{split}
\end{equation}
the restriction of the dg-bimodule $F$ along $f$. We used the derived ``dual'' Yoneda lemma \eqref{equation:derivedcoYoneda} alongside the fact that $(h_f)_{A'} = h_{f(A')}$ is h-flat for all $A' \in \cat A'$.

If $F \colon \cat A \to \cat B$ is a quasi-functor, we may define its \emph{opposite} quasi-functor $\opp{F} \colon \opp{\cat A} \to \opp{\cat B}$ as follows:
\begin{equation} \label{equation:opposite_qfun}
    \opp{F} = \RHom_{\cat B}(F,h_{\cat A}).
\end{equation}
Concretely, we typically take a termwise h-projective resolution $Q(F) \to F$ and we set:
\[
(\opp{F})_A^B = \Homdg_{\cat B}(Q(F)_A, h_B),
\]
as an $\opp{\cat A}$-$\opp{\cat B}$-dg-bimodule. We have the typical properties, up to isomorphism of quasi-functors:
\[
\opp{(\opp{F})} \cong F, \qquad \opp{(GF)} \cong \opp{G}\opp{F}.
\]

Any quasi-functor $F \colon \cat A \to \cat B$ yields an ordinary functor
\begin{equation}
    H^0(F) \colon H^0(\cat A) \to H^0(\cat B).
\end{equation}
Taking $H^0$ is compatible with opposites and compositions:
\[
H^0(GF) \cong H^0(G) H^0(F), \qquad H^0(\opp{F}) \cong \opp{H^0(F)}.
\]

\section{Homotopy (co)limits and homotopy Kan extensions} \label{subsection:hokandg}

In this section, we explore the basics of the theory of homotopy limits and colimits inside a dg-category. We remark that any theory of limits and colimits in dg-category is necessarily a theory of \emph{weighted} limits and colimits, since dg-categories are an example of enriched categories. Then, in order to define \emph{weighted homotopy (co)limits}, we essentially follow the approach of \cite{riehl-cathtpy} (see also \cite[Definition 5.2]{lack-rosicky-homotopylocpres}), which can be summarized by the following slogan: \emph{a weighted homotopy (co)limit is obtained from a weighted (co)limit by applying suitable resolutions}. Finally, we will require the typical representability condition, up to quasi-isomorphism.

We remark that the same definitions of weighted homotopy (co)limits in dg-categories are given in \cite{genovese-lowen-symons-vdb-deformations}, albeit with a significantly different scope. Moreover, a similar theory is developed in \cite{imamura-homotopy-dg}.

\subsection{Weighted homotopy (co)limits}
We will assume that any dg-category is h-projective, or implicitly replace it with an h-projective resolution.

We first define \emph{homotopy weighted limits} very generally as suitable dg-bimodules.
\begin{definition}
    Let $F \colon \cat I \to \cat A$ be a quasi-functor between dg-categories, and let $W \in \dercomp(\opp{\cat I})$ be a left $\cat I$-dg-module, called \emph{weight}. We define:
    \[
    \wholim{F}{W} = \RHom_{\cat I}(W,F) \in \dercomp(\cat A).
    \]

    Recalling \eqref{equation:derivedhomdgmod}, we have:
    \[
    \wholim{F}{W} \cong \Homdg_{\cat I}(Q(W),F) \cong \Homdg_{\cat I}(W, R(F))
    \]
    in $\dercomp(\cat A)$, where $Q(W) \to W$ is an h-projective resolution of the weight $W$, and $F \to R(F)$ is an h-injective resolution of $F$ as an $I$-$\cat A$-dg-bimodule.
\end{definition}

The definition of weighted homotopy limits inside a dg-category $\cat A$ is given via (weak) representability:
\begin{definition}
    Let $F \colon \cat I \to \cat A$ be a quasi-functor between dg-categories, and let $W \in \dercomp(\opp{\cat I})$ be a weight. A \emph{$W$-weighted homotopy limit of $F$} is an object $X \in \cat A$ together with an isomorphism
    \[
    \cat A(-,X) \xrightarrow{\sim} \wholim{F}{W}
    \]
    in $\dercomp(\cat A)$. Abusing notation, we will denote $X = \wholim{F}{W}$. It is uniquely determined up to isomorphism in $H^0(\cat A)$, and will be called \emph{the} $W$-weighted homotopy limit of $F$.
\end{definition}

Weighted homotopy colimits are obtained by duality, using the notion of opposite quasi-functor defined in \eqref{equation:opposite_qfun}:
\begin{definition}
    Let $F \colon \cat I \to \cat A$ be a quasi-functor between dg-categories, and let $W \in \dercomp(\cat I)$ be a weight. A \emph{$W$-weighted homotopy colimit of $F$} is a $W$-weighted homotopy limit of $\opp{F} \colon \opp{\cat I} \to \opp{\cat A}$, viewing $W \in \dercomp(\opp{(\opp{\cat I})})$. 
    
    Unraveling this definition, we see that a $W$-weighted homotopy colimit of $F$ is an object $X \in \cat A$ together with an isomorphism
    \[
    \cat A(X,-) \xrightarrow{\sim} \holim^W \opp{F} = \RHom_{\opp{\cat I}}(W, \opp{F})
    \]
    in $\dercomp(\opp{\cat A})$. Abusing notation, we will denote $X = \whocolim{F}{W}$. It is uniquely determined up to isomorphism in $H^0(\cat A)$, and will be called \emph{the} $W$-weighted homotopy colimit of $F$. 
\end{definition}
\begin{remark}
    Recalling \eqref{equation:qfun_dercatdg_dercat}, a weight $W \in \dercomp(\opp{\cat I})$ (respectively in $\dercomp(\cat I)$ in the case of weighted colimits) can be also viewed as a quasi-functor $\cat I \to \dercompdg(\basering k)$ (respectively a quasi-functor $\opp{\cat I} \to \dercompdg(\basering k)$).
\end{remark}

Let $\cat A$ be a dg-category, and let $\dercompdg(\cat A)$ be its derived dg-category (see \eqref{equation:deriveddgcat_H0}). We may ask whether weighted homotopy limits and colimits exist in $\dercompdg(\cat A)$.  Recalling \eqref{equation:qfun_dercatdg_bimod}, we can view quasi-functors $\cat I \to \dercompdg(\cat A)$ both as right-quasi-representable $\cat I$-$\dercompdg(\cat A)$-dg-bimodules or $\cat I$-$\cat A$-bimodules via restriction along the dg-derived Yoneda embedding of $\cat A$; we will use this identification in the following result:
\begin{proposition}
    Let $\cat I, \cat A$ be dg-categories, let $F \colon \cat I \to \dercompdg(\cat A)$ be a quasi-functor, and let $W \in \dercomp(\opp{\cat I})$ be a weight. Then, the parametrized derived hom complex (viewing $F$ as an $\cat I$-$\cat A$-dg-bimodule):
    \[
    \RHom_{\cat I}(W,F) \in \dercompdg(\cat A),
    \]
    together with the natural Yoneda isomorphism in $\dercomp(\dercompdg(\cat A))$ (viewing $F$ also as an $\cat I$-$\dercompdg(\cat A)$-dg-bimodule):
    \[
    \RHom_\cat A(-,\RHom_{\cat I}(W,F)) \xrightarrow{\sim} \RHom_{\cat I}(W,F),
    \]
    is a $W$-weighted homotopy limit of $F$.
    
    Dually, let $F \colon \cat I \to \dercompdg(\cat A)$ be a quasi-functor and let $W \in \dercomp(\cat I)$ be a weight. Then, the parametrized derived tensor product (viewing $F$ as an $\cat I$-$\cat A$-dg-bimodule):
    \[
    W \lotimes_{\cat I} F \in \dercompdg(\cat A),
    \]
    together with the natural derived tensor-hom adjunction isomorphism in $\dercomp(\opp{\dercompdg(\cat A)})$ (viewing $F$ also as an $\cat I$-$\dercompdg(\cat A)$-dg-bimodule):
    \[
    \RHom_{\cat A}(W \lotimes_{\cat I} F, -) \xrightarrow{\sim} \RHom_{\cat I}(W, \RHom_{\cat A}(F,-)) = \RHom_{\cat I}(W,\opp{F}),
    \]
    is a $W$-weighted homotopy colimit of $F$.
\end{proposition}
\begin{proof}
    The first claim is straightforward. For the dual claim, we just observe that $\opp{F}$ is by definition indeed $\RHom_{\cat A}(F,-)$ as an $\opp{\cat I}$-$\opp{\dercompdg(\cat A)}$-dg-bimodule, if we view $F$ as an $\cat I$-$\dercompdg(\cat A)$-dg-bimodule.
\end{proof}

\subsection{Homotopically (co)complete dg-categories}
With weighted homotopy limits and colimits, it is natural to give the following definition:
\begin{definition}
    Let $\cat A$ be a dg-category. We say that $\cat A$ is \emph{homotopically complete} if for any quasi-functor $F \colon \cat I \to \cat A$ and any weight $W \in \dercomp(\opp{\cat I})$, the $W$-weighted homotopy limit of $F$ exists in $\cat A$.

    Dually, we say that $\cat A$ is \emph{homotopically cocomplete} if for any quasi-functor $F \colon \cat I \to \cat A$ and any weight $W \in \dercomp(\cat I)$, the $W$-weighted homotopy colimit of $F$ exists in $\cat A$.
\end{definition}
\begin{notation}
    We will sometimes simplify terminology and say just ``complete'' or ``cocomplete'' instead of ``homotopically complete'' or ``homotopically cocomplete''.
\end{notation}

We recall that a dg-category $\cat A$ is pretriangulated if it has shifts and cones (Definition \ref{definition:pretriangulated_dgcat}). We can describe such shifts and cones in terms of homotopy weighted limits and colimits.
\begin{lemma} \label{lemma:shift_holim}
    Let $\cat A$ be a dg-category, let $A \in \cat A$ be an object and let $n \in \mathbb Z$ be an integer. Take the quasi-functor defined by:
    \begin{equation*}
        \begin{split}
        \Psi_A \colon \basering k & \to \cat A, \\
        \ast & \mapsto \cat A(-,A),
        \end{split}
    \end{equation*}
    where $\ast$ is the unique object of $\basering k$ viewed as a dg-category. Then, we have:
    \begin{equation}
        \RHom_{\basering k}(\basering k[-n], \Psi_A) \cong \cat A(-,A)[n]
    \end{equation}
    In particular, the $n$-shift of $A$ in $\cat A$ is the $\basering k[-n]$-weighted homotopy limit of $\Psi_A$.
\end{lemma}
\begin{proof}
    We observe that $\RHom_{\basering k}(\basering k[-n], A) \cong \Homdg_{\basering k}(\basering k[-n], A)$, because $\basering k[-n]$ is clearly h-projective. Then, the claim is a straightforward application of the Yoneda lemma.
\end{proof}
\begin{lemma} \label{lemma:cone_holim}
    Let $\cat A$ be a dg-category and let $f \colon A \to B$ be a closed degree $0$ morphism. Take $\Psi_f$ as the following quasi-functor:
    \begin{equation*}
        \begin{split}
            \Psi_f \colon \basering k[[1]] &\to \cat A,\\
            (0 \xrightarrow{e_0} 1) &\mapsto (f_* \colon \cat A(-,A) \to \cat A(-,B)),
        \end{split}
    \end{equation*}
    where $\basering k[[1]]$ is the dg-category freely generated by the diagram $(0 \xrightarrow{e_0} 1)$. Next, take the following left $\basering k[[1]]$-dg-module:
    \begin{equation*}
        \begin{split}
            \Psi_{\mathrm{incl}} \colon \basering k[[1]] &\to \rdgm{k}, \\
            (0 \xrightarrow{e_0} 1) & \mapsto (\mathrm{incl} \colon \basering k \to \cone(1_{\basering k})),
        \end{split}
    \end{equation*}
    where $\mathrm{incl}$ is explicitly defined as follows:
    \begin{equation*}
    \begin{tikzcd}
    \cdots \arrow[r] & 0 \arrow[r] & 0 \arrow[r] \arrow[d] & \basering k \arrow[d, equal] \arrow[r] & 0 \arrow[r] & \cdots \\
    \cdots \arrow[r] & 0 \arrow[r] & \basering k \arrow[r, equal] & \basering k \arrow[r]           & 0 \arrow[r] & \cdots
    \end{tikzcd}
    \end{equation*}
    Then, we have:
    \begin{equation}
        \RHom_{\basering k[[1]]}(\Psi_{\mathrm{incl}}, \Psi_f) \cong \cone(f_* \colon \cat A(-,A) \to \cat A(-,B))[-1].
    \end{equation}
    In particular, the cone of $f \colon A \to B$ shifted by $-1$ in $\cat A$ is the $\Psi_{\mathrm{incl}}$-weighted homotopy limit of $\Psi_f$
\end{lemma}
\begin{proof}
    We will check the following more general claim. Let $G \in \pshdg(\basering k[[1]], \cat A)$ be a $\basering k[[1]]$-$\cat A$-dg-bimodule. Then, we have the following commutative diagram, where rows are isomorphisms of $\cat A$-dg-modules:
    \begin{equation} \label{equation:pretriangles_cone_holim}
        \begin{tikzcd}
    {\Homdg(\Psi_{\mathrm{incl}}, G)} \arrow[d] \arrow[r, "\sim"]       & {\cone(G_{e_0} \colon G_0 \to G_1)[-1]} \arrow[d] \\
    {\Homdg(\Psi_{1_{\basering k}}, G)} \arrow[d] \arrow[r, "\sim"]     & G_0 \arrow[d, "G_{e_0}"]                             \\
    {\Homdg(\Psi_{(0 \to \basering k)}, G)} \arrow[d] \arrow[r, "\sim"] & G_1 \arrow[d]                                        \\
    {\Homdg(\Psi_{\mathrm{incl}}[-1], G)} \arrow[r, "\sim"]             & \cone(G_{e_0} \colon G_0 \to G_1),               
    \end{tikzcd}
\end{equation}
where $\Psi_{(0 \to \basering k)}$ and $\Psi_{1_{\basering k}}$ are the left $\basering k[[1]]$-modules corresponding to the chain maps $0 \to \basering k$ and $1_{\basering k} \colon \basering k \to \basering k$. The left vertical arrows are induced (by precomposition) by the following commutative diagram of $\basering k$-dg-modules:
\begin{equation*}
    \begin{tikzcd}
    {\basering k[-1]} \arrow[d] \arrow[r, "{\mathrm{incl}[-1]}"] & {\cone(1_{\basering k}))[-1]} \arrow[d] \\
    0 \arrow[d] \arrow[r]                                        & \basering k \arrow[d, equal]                  \\
    \basering k \arrow[d, equal] \arrow[r, "1_{\basering k}"]           & \basering k \arrow[d]                  \\
    \basering k \arrow[r, "\mathrm{incl}"]                       & \cone(1_{\basering k}).                
    \end{tikzcd}
\end{equation*}
The vertical arrows are pretriangles, hence we obtain a pretriangle
\begin{equation} \label{equation:weight_is_cone}
\Psi_{\mathrm{incl}}[-1] \to \Psi_{(0 \to \basering k)} \to \Psi_{1_{\basering k}} \to \Psi_{\mathrm{incl}}
\end{equation}
of left $\basering k[[1]]$-modules. From this, we also get a pretriangle
\[
\Homdg(\Psi_{\mathrm{incl}}, G) \to \Homdg(\Psi_{1_{\basering k}}, G) \to \Homdg(\Psi_{(0 \to \basering k)}, G) \to \Homdg(\Psi_{\mathrm{incl}}[-1], G),
\]
and in particular $\Homdg(\Psi_{\mathrm{incl}}[-1], G)$ is isomorphic to the cone of $\Homdg(\Psi_{1_{\basering k}}, G) \to \Homdg(\Psi_{(0 \to \basering k)}, G)$. Now, defining the two middle isomorphisms in \eqref{equation:pretriangles_cone_holim} is straightforward. Thanks to the universal property of the cone, we may find the upper and bottom horizontal isomorphisms making \eqref{equation:pretriangles_cone_holim} commute.

Directly from \eqref{equation:pretriangles_cone_holim}, we deduce that:
\begin{itemize}
    \item $\Psi_{1_{\basering k}}$ and $\Psi_{(0 \to \basering k)}$ are h-projective $\basering k[[1]]$-dg-module, hence $\Psi_{\mathrm{incl}}$ is also h-projective, and
    \[
    \RHom_{\basering k[[1]]}(\Psi_{\mathrm{incl}}, G) \cong \Homdg_{\basering k[[1]]}(\Psi_{\mathrm{incl}}, G).
    \]
    \item Taking $G=\Psi_f$, we can immediately conclude. \qedhere
\end{itemize}
\end{proof}

The above Lemma \ref{lemma:shift_holim} and Lemma \ref{lemma:cone_holim} can be directly dualized, yielding the following:
\begin{lemma} \label{lemma:shift_hocolim}
Let $\cat A$ be a dg-category, let $A \in \cat A$ be an object and let $n \in \mathbb Z$ be an integer. Take the quasi-functor defined by:
\begin{equation*}
        \begin{split}
        \Psi_A \colon \basering k & \to \cat A, \\
        \ast & \mapsto \cat A(-,A),
        \end{split}
    \end{equation*}
    where $\ast$ is the unique object of $\basering k$ viewed as a dg-category. Then, we have:
    \[
    \RHom_{\basering k}(\basering k[n], \opp{\Psi}_A) \cong \cat A(A,-)[-n].
    \]
    In particular, the $n$-shift of $A$ in $\cat A$ is the $\basering k[n]$-weighted homotopy colimit of $\Psi_A$.
\end{lemma}
\begin{proof}
    We observe that $\opp{\Psi}_A$ is the quasi-functor defined by:
\begin{equation*}
        \begin{split}
        \opp{\Psi}_A \colon \opp{\basering k} = \basering k & \to \opp{\cat A}, \\
        \ast & \mapsto \opp{\cat A}(-,A) = \cat A(A,-),
        \end{split}
\end{equation*}
then we may apply Lemma \ref{lemma:shift_hocolim} and get
\[
\RHom_{\basering k}(\basering k[n], \opp{\Psi}_A) \cong \opp{\cat A}(-,A)[-n] \cong \cat A(A,-)[-n].
\]
We conclude that the $(-n)$-shift of $A$ in $\opp{\cat A}$ is the $\basering k[n]$-weighted homotopy limit of $\opp{\Psi}_A$; recalling Remark \ref{remark:opposite_pretriangulated}, this means that the $n$-shift of $A$ in $\cat A$ is the $\basering k[n]$-weighted homotopy colimit of $\Psi_A$.
\end{proof}

\begin{lemma} \label{lemma:cone_hocolim}
 Let $\cat A$ be a dg-category and let $f \colon A \to B$ be a closed degree $0$ morphism. Take $\Psi_f$ as the following quasi-functor:
    \begin{equation*}
        \begin{split}
            \Psi_f \colon \basering k[[1]] &\to \cat A,\\
            (0 \xrightarrow{e_0} 1) &\mapsto (f_* \colon \cat A(-,A) \to \cat A(-,B)),
        \end{split}
    \end{equation*}
    where $\basering k[[1]]$ is the dg-category freely generated by the diagram $(0 \xrightarrow{e_0} 1)$. Next, take the following right $\basering k[[1]]$-dg-module:
    \begin{equation*}
        \begin{split}
            \Psi_{\mathrm{incl}} \colon \opp{\basering k[[1]]} &\to \rdgm{k}, \\
            (1 \xrightarrow{\opp{e}_0} 0) & \mapsto (\mathrm{incl} \colon \basering k \to \cone(1_{\basering k})),
        \end{split}
    \end{equation*}
    where $\mathrm{incl}$ is explicitly defined as follows:
    \begin{equation*}
    \begin{tikzcd}
    \cdots \arrow[r] & 0 \arrow[r] & 0 \arrow[r] \arrow[d] & \basering k \arrow[d, equal] \arrow[r] & 0 \arrow[r] & \cdots \\
    \cdots \arrow[r] & 0 \arrow[r] & \basering k \arrow[r, equal] & \basering k \arrow[r]           & 0 \arrow[r] & \cdots
    \end{tikzcd}
    \end{equation*}
    Then, we have:
    \begin{equation} \label{equation:cone_hocolim}
        \RHom_{\basering k[[1]]}(\Psi_{\mathrm{incl}}, \opp{\Psi}_f) \cong \cone(f^* \colon \cat A(B,-) \to \cat A(A,-))[-1].
    \end{equation}
    In particular, the cone of $f \colon A \to B$ in $\cat A$ is the $\Psi_{\mathrm{incl}}$-weighted homotopy colimit of $\Psi_f$.
\end{lemma}
\begin{proof}
    We can check that $\opp{\Psi}_f$ is the following quasi-functor:
\begin{equation*}
        \begin{split}
            \opp{\Psi}_f = \Psi_{\opp{f}} \colon \opp{\basering k[[1]]} &\to \opp{\cat A},\\
            (1 \xrightarrow{\opp{e}_0} 0) &\mapsto (\opp{f}_* \colon \opp{\cat A}(-,B) \to \opp{\cat A}(-,A)),
        \end{split}
\end{equation*}
where $\opp{f}$ is the morphism $f \colon A \to B$ viewed as a morphism $\opp{f} \colon B \to A$ in $\opp{\cat A}$. Then, directly from Lemma \ref{lemma:cone_holim}, we get \eqref{equation:cone_hocolim}, namely, the cone of $\opp{f}$ shifted by $-1$ in $\opp{\cat A}$ is the $\Psi_{\mathrm{incl}}$-weighted homotopy limit of $\opp{\Psi}_f$. Recalling Remark \ref{remark:opposite_pretriangulated}, this means that the cone of $f$ in $\cat A$ is the $\Psi_{\mathrm{incl}}$-weighted homotopy colimit of $\Psi_f$, as claimed.
\end{proof}

With the above results, we may immediately prove:
\begin{proposition} \label{proposition:hocomplete_pretr}
    Let $\cat A$ be a dg-category. If $\cat A$ is either homotopically complete or homotopically cocomplete, it is pretriangulated.
\end{proposition}
\begin{proof}
    If $\cat A$ is homotopically complete, we may combine Lemma \ref{lemma:shift_holim} and \ref{lemma:cone_holim} to prove that $\cat A$ has shifts and cones. If $\cat A$ is homotopically cocomplete, we may instead apply Lemma \ref{lemma:shift_hocolim} and Lemma \ref{lemma:cone_hocolim}.
\end{proof}

We are now able to prove the following characterization of homotopically complete or cocomplete dg-categories:
\begin{proposition}\label{proposition:homotopically-complete}
    Let $\cat A$ be a dg-category. Then, $\cat A$ is homotopically complete if and only if it is pretriangulated and $H^0(\cat A)$ has arbitrary (small) products.

    Dually, $\cat A$ is homotopically cocomplete if and only if it is pretriangulated and $H^0(\cat A)$ has arbitrary (small) coproducts.
\end{proposition}
\begin{proof}
   Recalling that $\cat A$ is pretriangulated if and only if $\opp{\cat A}$ is pretriangulated, we see the second claim follows from the first by replacing $\cat A$ with its opposite; hence, we concentrate on proving the first claim.

    Let us assume that $\cat A$ is homotopically complete. Thanks to the above Proposition \ref{proposition:hocomplete_pretr}, we know that $\cat A$ is pretriangulated. Let $I$ be a set and let $\{A_i : i \in I\}$ be a family of objects in $\cat A$. We may view $I$ as a discrete category, and let
    \[
    c_I \colon \basering k[I] \to \rdgm{\basering k}
    \]
    be the constant dg-functor, mapping $i \in I$ to $\basering k$ (this is well defined, since $\basering k[I]$ is a dg-category freely generated by a category). We observe that $c_I$ is h-projective, indeed, if $X$ is any left $\basering k[I]$-dg-module, we have a natural isomorphism of complexes:
    \begin{equation} \label{equation:constant_discrete_hproj}
    \Homdg_{\basering k[I]}(c_I, X) \cong \prod_{i \in I} X_i,
    \end{equation}
    then we see that $\Homdg_{\basering k[I]}(c_I, X)$ is acyclic if $X$ is acyclic, being a product of acyclic $\basering k$-dg-modules.
    
    Next, let $F$ the quasi-functor (directly induced by a dg-functor) defined by
    \begin{equation*}
    \begin{split}
    F \colon &\basering k[I] \to \cat A, \\
    i &\mapsto \cat A(-,A_i).
    \end{split}
    \end{equation*}
    We claim that:
    \begin{equation} \label{equation:products_holim}
        \RHom_{\basering k[I]}(c_I, F) \cong \prod_{i \in I} \cat A(-,A_i)
    \end{equation}
    in $\dercomp(\cat A)$. This follows from a similar argument as for the above isomorphism \eqref{equation:constant_discrete_hproj}, combined with the fact that $\RHom$ is identified (up to quasi-isomorphism) with $\Homdg$, since $c_I$ is h-projective. From \eqref{equation:products_holim}, combined with the homotopical completeness hypothesis, we find an object $A \in \cat A$ and an isomorphism
    \[
    \cat A(-,A) \xrightarrow{\sim} \prod_{i \in I} \cat A(-,A_i)
    \]
    in $\dercomp(\cat A)$, from which we immediately deduce an isomorphism
    \[
    H^0(\cat A)(-,A) \xrightarrow{\sim} \prod_{i \in I} H^0(\cat A)(-,A_i),
    \]
    which proves that the product of $\{A_i : i \in I\}$ exists in $H^0(\cat A)$, as we wanted.

    For the converse implication, let us assume that $\cat A$ is pretriangulated and $H^0(\cat A)$ has products. We first see that such products in $H^0(\cat A)$ imply the existence of ``homotopy products'' in $\cat A$, namely, for any set $I$ and any family of objects $\{A_i : i \in I\}$ in $\cat A$, there is an object $A \in \cat A$ and an isomorphism
    \begin{equation} \label{equation:hoproducts}
    \cat A(-,A) \xrightarrow{\sim} \prod_{i \in I} \cat A(-,A_i) \tag{$\ast$}
    \end{equation}
    in $\dercomp(\cat A)$. Indeed, by hypothesis we have an object $A$ and an isomorphism
    \begin{equation} \label{equation:H0products}
    H^0(\cat A(-,A)) \xrightarrow{\sim} \prod_{i \in I} H^0(\cat A(-,A_i)) \tag{$H^0(\ast)$}
    \end{equation}
    of right $H^0(\cat A)$-modules. Since $\cat A$ has shifts, the above isomorphism yields an isomorphism of graded $H^*(\cat A)$-modules:
    \[ 
        H^*(\cat A(-,A)) \xrightarrow{\sim} \prod_{i \in I} H^*(\cat A(-,A_i)). 
    \]
    Hence, thanks to the isomorphism $H^0(\prod_i \cat A(-,A_i)) \cong \prod_i H^0(\cat A(-,A_i))$ and a direct application of the Yoneda lemma, we can lift \eqref{equation:H0products} to an isomorphism \eqref{equation:hoproducts}, as claimed.

    Now, let $F \colon \cat I \to \cat A$ be a quasi-functor and let $W \in \dercomp(\opp{\cat I})$ be a weight. We look for an object $X \in \cat A$ and an isomorphism
    \[
    \cat A(-,X) \xrightarrow{\sim} \RHom_{\cat I}(W,F)
    \]
    in $\dercomp(\cat A)$. We will use that the derived category $\dercomp(\opp{\cat I})$ is compactly generated by the representable dg-modules $\cat I(i,-)$, namely, it is the smallest triangulated subcategory of itself which contains all representables and it is also closed under arbitrary (small) direct sums. Let $\cat W \subseteq \dercomp(\opp{\cat I})$ be the full subcategory of $\dercomp(\opp{\cat I})$ spanned by the following family of objects:
    \[
    \{ M \in \dercomp(\opp{\cat I}) : \RHom_{\cat I}(M,F) \cong \cat A(-,A), \text{in $\dercomp(\cat A)$, for some $A \in \cat A$} \}.
    \]
    We are going to show that $\cat W$ contains the representables $\cat I(i,-)$ for all $i \in \cat I$, it is triangulated and closed under arbitrary direct sums.

    First, let $i \in \cat I$. Then:
    \[
    \RHom_{\cat I}(\cat I(i,-), F) \cong F_i \cong \cat A(-,\Phi_F(i))
    \]
    in $\dercomp(\cat A)$, because $F$ is a quasi-functor. Next, if $M \in \cat W$, we have:
    \[
    \RHom_{\cat I}(M[n],F) \cong \RHom_{\cat I}(M,F)[-n] \cong \cat A(-,A)[-n] \cong \cat A(-,A[-n]),
    \]
    using that $\cat A$ is closed under shifts. Moreover, let
    \[
    M \to M' \to M''
    \]
    be a distinguished triangle in $\dercomp(\opp{\cat I})$, with $M,M' \in \cat W$. Then, we get a distinguished triangle
    \[
    \RHom_{\cat I}(M'', F) \to \RHom_{\cat I}(M',F) \to \RHom_{\cat I}(M, F))
    \]
    in $\dercomp(\cat A)$. Since both $\RHom_{\cat I}(M',F)$ and $\RHom_{\cat I}(M, F)$ are isomorphic to representable $\cat A$-dg-modules in $\dercomp(\cat A)$ and $\cat A$ is pretriangulated, the same is true for $\RHom_{\cat I}(M'', F)$, as claimed. Finally, let $\{M_j : j \in J \}$ be a family of objects in $\cat W$. Then:
    \[
    \RHom_{\cat I}(\bigoplus_j M_j, F) \cong \prod_i \RHom_{\cat I}(M_j, F)
    \]
    in $\dercomp(\cat A)$. Since $\RHom_{\cat I}(M_j, F) \cong \cat A(-,A_j)$ in $\dercomp(\cat A)$ for some $A_j \in \cat A$ for all $j \in J$ and $H^0(\cat A)$ has arbitrary products, we may conclude (using the observation on ``homotopy products'' of the previous paragraph) that $\prod_j \RHom_{\cat I}(M_j, F) \cong \cat A(-, \prod_i A_j)$, as we wanted.
\end{proof}

\subsection{Pretriangulated dg-categories and finite weighted homotopy (co)limits} \label{subsection:pretr_dg_hocolim}

Weighted homotopy limits and colimits can be used to give an interesting characterization of pretriangulated dg-categories in terms of a suitable ``finite'' homotopy (co)completeness. We will restrict our attention to the following specific class of finite diagram categories:
\begin{definition}\label{definition:finite_directed}
    Let $I$ be a finite category, namely, having a finite set of objects and of morphisms. We say that $I$ is a (finite) \emph{direct category} if there is a non-negative integer $n > 0$ and a function  
    \[
    d \colon \operatorname{Ob}(I) \to \{0,\cdots n\}
    \]
    such that, for any non-identity morphism $\alpha \colon i \to j$, we have:
    \[
    d(i) < d(j).
    \]
We denote by $\kat{FinDir}$ the full sub-$2$-category of $\kat{Cat}$ formed by finite direct categories.
\end{definition}

We can prove:
\begin{proposition} \label{proposition:pretr_dgcat_finite_holims}
    Let $\cat A$ be a (non empty) dg-category. Then, the following are equivalent:
    \begin{enumerate}
        \item $\cat A$ is pretriangulated.
        \item For any finite direct category $I$, the dg-category $\cat A$ has all $W$-weighted homotopy limits $\wholim{F}{W}$, where $F \colon \basering k[I] \to \cat A$ is a quasi-functor and $W \in \pretr(\opp{\basering k[I]})$.
        \item For any finite direct category $I$, the dg-category $\cat A$ has all $W$-weighted homotopy colimits $\wholim{F}{W}$, where $F \colon \basering k[I] \to \cat A$ is a quasi-functor and $W \in \pretr(\basering k[I])$.
    \end{enumerate}
\end{proposition}
\begin{proof}
    We recall that $\cat A$ is pretriangulated if and only if $\opp{\cat A}$ is pretriangulated. Moreover, it is clear that a finite category $I$ is a direct category if and only if $\opp{I}$ is a direct category. Hence, by duality we only need to prove $(1) \Leftrightarrow (2)$.

    $(2) \Rightarrow (1)$. We know from Lemma \ref{lemma:shift_holim} that, for any object $A \in \cat A$, the $n$-shift $A[n]$ of $A$ ($n \in \mathbb Z$) can be identified with the $\basering k[-n]$-weighted homotopy limit of $\Psi_A$, defined on $\basering k[[0]]=\basering k$, where $[0]$ is the category having a single object and just the identity morphism, which is clearly a finite direct category. The weight $\basering k[-n]$ clearly lies in $\pretr(\opp{\basering k})$. We therefore conclude that $\cat A$ has all $n$-shifts, for any $n \in \mathbb Z$. Next, we know from Lemma \ref{lemma:cone_holim} that, given a closed degree $0$ morphism $f$ in $\cat A$, the shifted cone $\cone(f)[-1]$ can be identified with the $\Psi_{\mathrm{incl}}$-weighted homotopy limit of $\Psi_f$, defined on the category $\basering k[[1]]$, where $[1]$ is the category with two objects and a single non-identity arrow $0 \to 1$, which clearly is a finite direct category. Moreover, we see from \eqref{equation:weight_is_cone} that $\Psi_{\mathrm{incl}}$ is the cone of a morphism
    \[
    \Psi_{(0 \to \basering k)} \to \Psi_{1_{\basering k}},
    \]
    and it is immediate to see that actually:
    \begin{align*}
    \Psi_{(0 \to \basering k)} &= \basering k[I](1,-), \\
    \Psi_{1_{\basering k}} &= \basering k[I](0,-),
    \end{align*}
    as left $\basering k[I]$-modules, whence we immediately get that $\Psi_{\mathrm{incl}} \in \pretr(\opp{\basering k[[1]]})$. We conclude that the shifted cone $\cone(f)[-1]$ exists in $\cat A$; this, combined with the existence of all shifts, ensures that indeed $\cat A$ is pretriangulated.

    $(1) \Rightarrow (2)$. Let $I$ be a finite direct category, let $F \colon I \to \cat A$ be a quasi-functor and let $W \in \pretr(\opp{\basering k[I]})$. We want to prove that the right $\cat A$-dg-module
    \[
    \Homdg_{\basering k[I]}(W,F)
    \]
    is quasi-representable. We let
    \[
    \mathcal F = \{W \in \pretr(\opp{\basering k[I]}) : \Homdg_{\basering k[I]}(W,F) \text{ is quasi-representable}\}.
    \]
    By definition of the pretriangulated hull $\pretr(\opp{\basering k[I]})$, we will conclude if we show that $\mathcal F$ contains all representable left $\basering k[I]$-dg-modules and is closed under shifts and cones. If $W=\basering k[I](i,-)$ is representable, we have $\Homdg_{\basering k[I]}(W,F) \cong F_i$, which is quasi-representable since $F$ is a quasi-functor. Next, if $W \in \mathcal F$, then $W[n]$ satisfies:
    \[
    \Homdg_{\basering k[I]}(W[n],F) \cong \Homdg_{\basering k[I]}(W,F)[-n],
    \]
    which is quasi-representable, since $\Homdg_{\basering k[I]}(W,F)$ is quasi-representable and $\cat A$ is pretriangulated. Finally, if $f \colon W \to W'$ is a closed degree $0$ morphism in $\mathcal F$, then $\Homdg_{\basering k[I]}(\cone(f)[-1],F)$ sits in the following pretriangle:
    \[
    \Homdg_{\basering k[I]}(W',F) \xrightarrow{f^*} \Homdg_{\basering k[I]}(W,F) \to \Homdg_{\basering k[I]}(\cone(f)[-1],F).
    \]
    Since the first two objects of this pretriangle are quasi-representable and $\cat A$ is pretriangulated, so is the third. We obtain that $\cone(f)[-1] \in \mathcal F$, which we may combine with closure under shifts to conclude that indeed $\mathcal F$ is closed under shifts and cones.
    
\end{proof}
\subsection{Homotopy Kan extensions}

Again, by assumption, any dg category will be h-projective or will implicitly be replaced by an h-projective resolution.
\begin{definition}
    Let $u \colon \cat A \to \cat B$ be a quasi-functor, and let $\cat C$ be a dg-category which admits all shifts of objects (e.g. $\cat C$ is pretriangulated). We have a \emph{restriction} functor
\begin{equation}
\begin{split}
u^* \colon \QFun(\cat B, \cat C) &\to \QFun(\cat A, \cat C), \\
u^*(G) &= G \circ u.
\end{split}
\end{equation}

Let $F \colon \cat A \to \cat C$ be a quasi-functor. A \emph{right homotopy Kan extension of $F$ along $u$} is a quasi-functor
\[
u_*(F) \colon \cat B \to \cat C
\]
together with an isomorphism
\begin{equation} \label{equation:righthoKan}
\QFun(\cat A, \cat C)(u^*(-),F) \cong \QFun(\cat B, \cat C)(-,u_*(F)).
\end{equation}
Dually, a \emph{left homotopy Kan extension of $F$ along $u$} is a quasi-functor
\[
u_{!}(F) \colon \cat B \to \cat C
\]
together with an isomorphism
\begin{equation} \label{equation:lefthoKan}
\QFun(\cat A, \cat C)(u_!(F),-) \cong \QFun(\cat B, \cat C)(F,u_*(-)).
\end{equation}
Clearly, $u_*(F)$ and $u_!(F)$ are uniquely determined in $\QFun(\cat B, \cat C)$, if they exist, and they will be called \emph{the} right or left Kan extensions of $F$ along $u$.
\end{definition}
\begin{remark}
    We required $\cat C$ to have shifts in the above definition, so that the above \eqref{equation:righthoKan} and \eqref{equation:lefthoKan} actually lift to quasi-isomorphisms
    \begin{align*}
        \RHom(\cat A, \cat C)(u^*(-),F) &\cong \RHom(\cat B, \cat C)(-,u_*(F)), \\
        \RHom(\cat A, \cat C)(u_!(F),-) &\cong \RHom(\cat B, \cat C)(F,u_*(-)),
    \end{align*}
    where $\RHom$ here denotes the dg-category of quasi-functors, cf.\ Remark \ref{remark:RHom_dgcat}. We may also drop the requirement that $\cat C$ has shifts, but then we have to replace $\QFun(\cat A, \cat C)$ and $\QFun(\cat B, \cat C)$ with corresponding \emph{graded categories}. In any case, for all the purposes of this work, $\cat C$ will be at least pretriangulated.
\end{remark}
\begin{remark}
If $u_*(F)$ or dually $u_!(F)$ exists for all quasi-functors $F \colon \cat A \to \cat C$, we obtain a functor
\[
u_* \colon \QFun(\cat A, \cat C) \to \QFun(\cat B, \cat C),
\]
which is right adjoint to the restriction $u^*$, or dually
\[
u_! \colon \QFun(\cat A, \cat C) \to \QFun(\cat B, \cat C),
\]
which is left adjoint to the restriction $u^*$. 
\end{remark}

It turns out that we can always define the right Kan extension as a bimodule. Indeed, if $F \colon \cat A \to \cat C$ is a quasi-functor, we may define
\begin{equation}
u_*(F) = \RHom_{\cat A}(u,F) \in \dercomp(\cat B, \cat C).
\end{equation}
For any quasi-functor $G \colon \cat B \to \cat C$, we have the derived tensor-hom adjunction \eqref{equation:derivedtensorhomdgm}:
\[
\RHom_{\cat A, \cat C}(u \lotimes_{\cat B} G, F) \cong \RHom_{\cat B, \cat C}(G, \RHom_{\cat A}(u,F)),
\]
from which we immediately deduce, taking $H^0$ and recalling that $u \lotimes_{\cat B} G$ is indeed the composition of quasi-functors $G \circ u$ (see also Remark \ref{remark:RHom_dgcat}):
\[
\QFun(\cat A, \cat C)(G \circ u,F) \cong \dercomp(\cat B,\cat C)(G, \RHom_{\cat A}(u,F)).
\]
Hence, if $\RHom_{\cat A}(u,F) = u_*(F)$ were a quasi-functor, it would be by definition the right homotopy Kan extension of $F$. It turns out that homotopy completeness is the condition ensuring that right homotopy Kan extensions exist:
\begin{proposition} \label{proposition:righthokan_holim}
    Let $u \colon \cat A \to \cat B$ and let $F \colon \cat A \to \cat C$ be quasi-functors between dg-categories. Then, $u_*(F)$ is a quasi-functor if and only if, for all $B \in \cat B$, the weighted homotopy limit
    \[
    \wholim{F}{u^B}
    \]
    exists in $\cat C$. In that case, the induced functor $\Phi_{u_*(F)} \colon H^0(\cat B) \to H^0(\cat C)$ is given on objects by
    \[
    \Phi_{u_*(F)}(B) = \wholim{F}{u^B}.
    \]
    In this case, the right homotopy Kan extension $u_*(F)$ of $F$ along $u$ exists. 

    In particular, if $\cat C$ is homotopy complete, all right homotopy Kan extensions always exist.
\end{proposition}
\begin{proof}
    Immediate, from the definition of homotopy weighted limit and from the definition of homotopy complete dg-category.
\end{proof}
Left homotopy Kan extensions can be obtained by dualizing right homotopy Kan extensions. Indeed, let $u \colon \cat A \to \cat B$ and $F \colon \cat A \to \cat C$ be quasi-functors. We take their opposites:
\begin{align*}
    \opp{u} \colon \opp{\cat A} &\to \opp{\cat B}, \\
    \opp{F} \colon \opp{\cat A} & \to \opp{\cat C}.
\end{align*}
If $u_!(F)$ exists as a quasi-functor, we have an isomorphism
\[
\QFun(\cat B, \cat C)(u_!(F),-) \cong \QFun(\cat A, \cat C)(F,u^*(-)).
\]
With opposites, this is equivalent to
\[
\QFun(\opp{\cat B}, \opp{\cat C})(-, \opp{u_!(F)}) \cong \QFun(\opp{\cat A}, \opp{\cat C})((\opp{u})^*(-), \opp{F}).
\]
This implies that:
\begin{equation}
    \opp{u_!(F)} \cong (\opp{u})_*(\opp{F}) \colon \opp{\cat B} \to \opp{\cat C}.
\end{equation}
Hence:
\begin{proposition} \label{proposition:lefthokan_hocolim}
   Let $u \colon \cat A \to \cat B$ and let $F \colon \cat A \to \cat C$ be quasi-functors between dg-categories. Then, the left homotopy Kan extension $u_!(F)$ of $F$ along $u$ exists if and only if, for all $B \in \cat B$, the weighted homotopy colimit
    \[
    \whocolim{F}{(\opp{u})^B}
    \]
    exists in $\cat C$. In that case, the induced functor $\Phi_{u_!(F)} \colon H^0(\cat B) \to H^0(\cat C)$ is given on objects by
    \[
    \Phi_{u_!(F)}(B) = \whocolim{F}{(\opp{u})^B}.
    \]
    In this case, the right homotopy Kan extension $u_*(F)$ of $F$ along $u$ exists. 

    In particular, if $\cat C$ is homotopy cocomplete, all left homotopy Kan extensions always exist.
\end{proposition}

\section{The derivator associated to a dg-category}\label{section:main result}

Let $I$ and $J$ be small categories, and let $u \colon I \to J$ be a functor. We have a unique induced $\basering k$-linear functor
\[
\basering k[u] = u \colon \basering k[I] \to \basering k[J]
\]
between the free $\basering k$-linear categories generated by $I$ and $J$. Moreover, we shall view $\basering k[I]$ and $\basering k[J]$ as dg-categories concentrated in degree $0$.
\begin{lemma}
Let $I$ be a small category. Then, the dg-category $\basering k[I]$ is h-projective, namely, it has h-projective hom-complexes.
\end{lemma}
\begin{proof}
Let $i,i'$ be objects of $I$. By definition, $\basering k[I](i,i')$ is the free $\basering k$-module generated by $I(i,i')$. Hence, it is h-projective.
\end{proof}
\begin{remark}
The above lemma, albeit simple, has the relevant consequence that we don't have to derive tensor products of the form $\basering k[I] \otimes \cat A$, for any dg-category $\cat A$.
\end{remark}

\subsubsection*{Restriction of quasi-functors along $u \colon \basering I \to \basering J$}

\begin{definition}
    Let $X \colon \basering k[J] \to \cat A$ be a quasi-functor. Then, we define a quasi-functor $u^*(X) \colon \basering k[I] \to \cat A$ as follows:
    \[
    u^*(X)_i^A = X_{u(i)}^A.
    \]

    Moreover, let $u,v \colon I \to J$ be functors, and let $\alpha \colon u \Rightarrow v$ be a natural transformation. We have an induced morphism of quasi-functors
    \[
    \alpha^* \colon u^*(X) \to v^*(X),
    \]
    defined in the obvious way.
\end{definition}

\subsection{The main result}

\begin{theorem} \label{theorem:derivator_dg}
Let $\cat A$ be a dg-category. We define the following $2$-functor:
\begin{equation} \label{equation:derivatordg_mainformula}
    \begin{split}
        \derivator_{\cat A} \colon \opp{\kat{Cat}} & \to \kat{CAT}, \\
        I & \mapsto \QFun(\basering k[I], \cat A).
    \end{split}
\end{equation}
Then, $\derivator_{\cat A}$ is a prederivator.

Moreover, if $\cat A$ is homotopically complete and cocomplete, $\derivator_{\cat A}$ is a stable derivator.
\end{theorem}

\begin{remark} \label{remark:restriction_elements}
    We have an equivalence
    \[
    H^0(\cat A) \cong \derivator_{\cat A}(e),
    \]
    where $e$ is the terminal category. If $X \colon \basering k[I] \to \cat A$ is a quasi-functor (i.e. an object in $\derivator_{\cat A}(I)$ and $i \in I$ (namely, $i \colon e \to I$), then we will view $i^*(X)$ as an element of $H^0(\cat A)$, identifying it with the quasi-representable right $\cat A$-dg-module given by
    \[
    X_i = X(-,i) \qis \cat A(-,i^*(X))
    \]
\end{remark}

\begin{remark}
Given a dg-category $\cat A$ as above, we can associate to it an $\infty$-category $\cat C$ via the dg-nerve functor \cite[Construction~1.3.1.6]{lurie-higheralgebra}. In principle, this allows us to associate a derivator to the dg-category $\cat A$, as outlined in Example \ref{ex:der} (3). However, the aim of this paper is to provide a more direct connection between dg-categories and derivators with explicit descriptions of the Kan extensions. We will not attempt to prove that the derivator defined here is equivalent to the one arising from the dg-nerve construction, as such a comparison lies beyond the scope of this work.
\end{remark}

If $R$ is a ring, we may consider its derived dg-category $\dercompdg(R)$ and take the associated derivator according to the above formula \eqref{equation:derivatordg_mainformula}. We may also consider the derivator defined by:
\begin{equation} \label{equation:derivatorabelian}
    I \mapsto \dercomp(\Mod(R)^I),
\end{equation}
where $\Mod(R)^I$ denotes the abelian category of functors $I \to \Mod(R)$, see also \cite[Example 1.2 (2)]{groth-derivators}. We may prove the following:
\begin{corollary} \label{corollary:derivator_ring}
    Let $R$ be a ring. Then, the $2$-functors:
    \begin{align*}
        I & \mapsto \QFun(\basering k[I], \dercompdg(R)), \\
        I & \mapsto \dercomp(\Mod(R)^I),
    \end{align*}
    define equivalent derivators.
\end{corollary}
\begin{proof}
    We observe that:
    \begin{align*}
    \Mod(R)^I &= \Fun(I,\Mod(R)) \\
    &\cong \Fun_{\basering k}(\basering k[I],\Mod(R)) \\
    &\cong \Fun_{\basering k}(\basering k[I] \otimes \opp{R}, \Mod(\basering k)) \\
    &=\Mod(\opp{\basering k[I]} \otimes R),
    \end{align*}
    from which we immediately deduce that:
    \[
    \dercomp(\Mod(R)^I) \cong \dercomp(\opp{\basering k[I]} \otimes R),
    \]
    naturally in $I \in \Cat$.
    
    Moreover, we deduce from \eqref{equation:qfun_dercatdg_bimod} the following equivalence, also natural in $I \in \Cat$:
    \begin{align*}
    \QFun(\basering k[I], \dercompdg(R)) &\cong \dercomp(\basering k[I], R) \\
    &= \dercomp(\opp{\basering k[I]} \otimes R).
    \end{align*}
    We thus conclude.
\end{proof}
The above corollary can be viewed as a ``bridge'' between the derivator associated to a dg-category and the derivator associated to the derived category of a ring.

\subsection{The proof of Theorem \ref{theorem:derivator_dg}}

The proof of Theorem \ref{theorem:derivator_dg} is given in the following subsections.

\subsubsection{Der1}
We need to show compatibility with coproducts.
\begin{proposition}
Let $\{I_i : i \in \mathfrak I\}$ be a family of small categories. Then, the natural functor
\[
\derivator_{\cat A}(\coprod_i I_i) \to \prod_i \derivator_{\cat A}(I_i)
\]
is an equivalence.
\end{proposition}
\begin{proof}
First, we observe that the natural dg-functor
\begin{equation} \label{equation:bimod_prod_coprod}
\dgm{\coprod_{i \in I} \basering k[I]}{\cat A} \to \prod_{i \in I} \dgm{\basering k[I_i]}{\cat A}
\end{equation}
is an isomorphism. Indeed, this is just
\[
\Fundg(\coprod_{i \in I} \basering k[I], \rdgm{\cat A}) \xrightarrow{\sim} \prod_{i \in I} \Fundg(\basering k[I_i], \rdgm{\cat A}).
\]

Next, we consider the following diagram:
\[
\begin{tikzcd}[row sep=10ex, column sep=large]
{H^0(\dgm{\coprod_{i \in I} \basering k[I]}{\cat A})} \arrow[d, "\mathrm{loc}"'] \arrow[r, "\sim", bend right] & {\prod_{i \in I}H^0( \dgm{\basering k[I_i]}{\cat A})} \arrow[l, "\sim"', bend right] \arrow[d, "\prod_{i \in I}\mathrm{loc}_i"] \\
{\dercomp(\dgm{\coprod_{i \in I} \basering k[I]}{\cat A})} \arrow[r, dashed, bend right]                        & {\prod_{i \in I}\dercomp(\dgm{\basering k[I_i]}{\cat A}).} \arrow[l, dashed, bend right]                                       
\end{tikzcd}
\]
The upper horizontal arrows are isomorphisms induced by \eqref{equation:bimod_prod_coprod} and its inverse; we remark that products of complexes (hence of dg-bimodules) are exact, and in particular the zeroth cohomology commutes with products. Moreover, these arrows map quasi-isomorphisms in $H^0(\dgm{\coprod_{i \in I} \basering k[I]}{\cat A})$ to products of quasi-isomorphisms in ${\prod_{i \in I}H^0( \dgm{\basering k[I_i]}{\cat A})}$ and vice-versa, by definition; this implies that both lower dashed arrows exist and they are inverse to one another, thanks to the universal property of localizations.

To conclude, we check that the dashed arrows preserve quasi-functors. This follows from the analogous fact regarding the upper isomorphisms, which is straightforward.
\end{proof}

\subsubsection{Der2}
We need to show the following:
\begin{proposition}
    Let $I \in \kat{Cat}$ and let $f \colon X \to Y$ be a morphism in $\derivator_{\cat A}(I)$. Then $f$ is an isomorphism if and only if $i^*(f) \colon i^*(X) \to i^*(Y)$ is an isomorphism in $H^0(\cat A)$ for all $i \in I$.
\end{proposition}
\begin{proof}
    $f$ is a morphism of quasi-functors, and it is immediate that $f$ is an isomorphism in $H^0(\RHom(\basering k[I], \cat A))$, namely a quasi-isomorphism, if and only if its components $f_i \colon X_i \to Y_i$ are quasi-isomorphisms. But such components are identified with the morphisms $i^*(f)$ in $H^0(\cat A)$, recalling Remark \ref{remark:restriction_elements}.
\end{proof}

\subsubsection{Der3}

Let $u \colon I \to J$ be a functor, and let
\[
u^* \colon \derivator_{\cat A}(J) \to \derivator_{\cat A}(I)
\]
be the restriction functor. We want to show that $u^*$ has both a left and right adjoint. Since $\cat A$ is by assumption homotopically complete and cocomplete, the result follows from the general theory of homotopy Kan extensions, in particular Proposition \ref{proposition:righthokan_holim} and Proposition \ref{proposition:lefthokan_hocolim}.

\subsubsection{Der4}
We consider the following setting. Let $u \colon I \to J$ be a functor between categories, and let $j \in J$ be an object. We recall that the category $j/u$ is defined as follows:
\begin{itemize}
    \item Objects are pairs $(i,f)$ where $f \colon j \to u(i)$ is a morphism in $J$.
    \item A morphism $g \colon (i,f) \to (i',f')$ in $j/u$ is a morphism $g \colon i \to i'$ in $I$ such that $u(g) \circ f = f'$.
\end{itemize}

There is a natural ``projection'' functor
\begin{equation}
\begin{split}
    q \colon j / u & \to I, \\
    (i,f) & \mapsto i.
\end{split}
\end{equation}
Let $\pi \colon j / u \to e$ be the terminal functor, and denote by $j \colon e \to J$ the functor selecting the object $j \in J$. There is a ``base change'' natural transformation
\begin{equation}
    \alpha \colon j \circ \pi \to u \circ q.
\end{equation}

We recall Definition \ref{definition:derivator}. Proving the first part of the axiom (Der4) amounts to checking that the induced morphism
\begin{equation} \label{equation:der41}
    \alpha_* \colon  j^* \circ u_* \to \pi_* \circ q^*
\end{equation}
of functors
\[
 j^* \circ u_*, \pi_* \circ q^* \colon \QFun(I, \cat A) \to \QFun(\basering k[e], \cat A) \cong A
\]
is an isomorphism. The other part of (Der4) is completely dual: we consider the category $u/j$ having pairs $(i,f \colon u(i) \to j)$ as objects, the natural projection $p \colon u/j \to I$, the terminal functor $\pi \colon u/j \to e$ and the natural transformation
\[
\alpha \colon p \circ u \to j \circ \pi.
\]
This induces a morphism
\begin{equation} \label{equation:der42}
    \alpha_! \colon \pi_! \circ p^*  \to j^* \circ u_!.
\end{equation}
of functors
\[
\pi_! \circ p^*, j^* \circ u_! \colon \QFun(I, \cat A) \to \QFun(\basering k[e], \cat A) \cong A
\]
Proving the second part of (Der4) amounts to checking that this is an isomorphism. Fortunately:
\begin{lemma}
    If we prove that \eqref{equation:der41} is an isomorphism for all $u \colon I \to J$ and all $j \in J$, the same claim for \eqref{equation:der42} automatically follows.
\end{lemma}
\begin{proof}
    It is a direct application of duality. Indeed, \eqref{equation:der42} is basically \eqref{equation:der41} applied to $\opp{\cat A}$ and $\opp{u} \colon \opp{I} \to \opp{J}$.
\end{proof}

To prove that \eqref{equation:der42}  is an isomorphism, we first unravel everything. If $F \colon I \to \cat A$ is a quasi-functor, then the morphism
\[
\alpha_*(F) \colon \pi_* \circ q^*(F) \to j^* u_*(F)
\]
is a morphism of right $\cat A$-dg-bimodules. For each $A \in \cat A$, it is a morphism in $\dercomp(\basering k)$:
\begin{equation} \label{equation:der4_unraveled}
    \RHom_{\basering k[j/u]}(\Homdg_{\basering k}(\basering k,\pi(-)), F_{q(-)}^A) \to \RHom_{\basering k[I]}(\basering k[J](j,u(-)), F^A).
\end{equation}
The strategy is to prove that the above morphism is an isomorphism \emph{when we replace the derived hom with the regular hom}. Namely:
\begin{lemma} \label{lemma:der4_underived}
    The natural morphism
    \begin{equation} \label{equation:der4_underived}
    \Homdg_{\basering k[j/u]}(\Homdg_{\basering k}(\basering k,\pi(-)), F_{q(-)}^A) \to \Homdg_{\basering k[I]}(\basering k[J](j,u(-)), F^A).
    \end{equation}
    is an isomorphism.
\end{lemma}
\begin{proof}
    This follows from a direct computation.
\end{proof}

Next, we know that the derived hom can be obtained from the regular hom by taking an h-injective dg-module in the second variable. In other words,
\[
\RHom(F,G) \cong \Homdg(F,G)
\]
if we take $G$ to be h-injective. In \eqref{equation:der4_unraveled}, we can certainly take $F$ to be an h-injective $\basering k[I]$-$\cat A$-dg-bimodule, without loss of generality. Then, we discover that we are actually in a very benign setting:
\begin{lemma}[{compare with \cite[Lemma 1.31]{groth-derivators}}] \label{lemma:der4_restriction_hinj}
    If $F$ is h-injective in \eqref{equation:der4_unraveled}, then:
    \begin{enumerate}
        \item $F^A$ is an h-injective $\basering k[I]$-dg-module for all $A \in \cat A$.
        \item Moreover, $F^A_{q(-)}$ is an h-injective left $\basering k[j/u]$-dg-module.
    \end{enumerate}
\end{lemma}
\begin{proof}
(1) Let $G \in \ldgm{\basering k[I]}$ be an acyclic dg-module. With an application of the Yoneda lemma and the tensor-hom adjunction, we have an isomorphism of complexes:
\[
    \Homdg_{\basering k[I]}(G,F^A) \cong \Homdg_{\basering k[I], \cat A}(G \otimes_{\basering k} \cat A(-,A), F).
\]
We recall that, by assumption, $\cat A$ is an h-projective dg-category. In particular, the hom complexes $\cat A(A',A'')$ are h-flat. This immediately implies that the $\basering k[I]$-$\cat A$-dg-bimodule $G \otimes_{\basering k} \cat A(-,A)$ is acyclic. Since $F$ is h-injective, we may conclude.

(2) Let $X \in \ldgm{\basering k[j/u]}$ be an acyclic dg-module. We claim that we have an isomorphism of complexes:
\[
\Homdg_{\basering k[I]}(\bigoplus_{f \colon j \to u(-)} X_{(-,f)}, F^A) \cong \Homdg_{\basering k[j/u]}(X,F^A_{q(-)}).
\]
This can be proved with a direct computation. Next, we see that the direct sum
\[
\bigoplus_{f \colon j \to u(i)} X_{(i,f)}
\]
is acyiclic for all $i \in \basering k[I]$, since $X$ is acyclic. From (1) above we know that $F^A$ is h-injective, so we may immediately conclude.
\end{proof}

Now, taking $F$ to be h-injective in \eqref{equation:der4_underived} and using the above Lemma \ref{lemma:der4_restriction_hinj}, the ``derived isomorphism'' \eqref{equation:der4_unraveled} immediately follows from Lemma \ref{lemma:der4_underived}.

\subsubsection{Der5 (strongness)}
In order to check the strongness property, we first prove a general result about quasi functors defined on $\basering k[[1]]$, where $[1]$ is the category with two objects and a single non-identity morphism:
\[
0 \xrightarrow{e_0} 1.
\]
\begin{proposition} \label{proposition:mor_cat_comparison}
    Let $\cat B$ be any dg-category. Then, the functor
    \begin{equation} \label{equation:functor_H0_underlying}
        \begin{split}
            \QFun(\basering k[[1]], \cat B) & \to \Fun_{\basering k}(\basering k[[1]], H^0(\cat B)), \\
            F & \mapsto H^0(F)
        \end{split}
    \end{equation}
is full and essentially surjective.
\end{proposition}
\begin{proof}
The idea is to identify the dg-category of quasi-functors $\RHom(\basering k[[1]], \cat B)$, up to quasi-equivalence, with the dg-category $\underline{\operatorname{Mor}}(\cat B)$ defined as follows:
\begin{itemize}
    \item Objects are closed degree $0$ morphisms in $\cat B$.
    \item A degree $p$ morphism from $f \colon A \to A'$ to $f' \colon A' \to B'$ is a triple $(u,v,h)$, where $u \colon A \to A'$ and $v \colon B \to B'$ are degree $p$ morphisms and $h \colon A \to B'$ is a degree $p-1$ morphism.
    \item The differentials are defined as follows:
    \[
    d(u,v,h)=(du,dv,dh + (-1)^p(f'u-vf)).
    \]
\end{itemize}
The identification can be made using \cite[Theorem C]{canonaco-ornaghi-stellari-ainf-ring} and finding an explicit isomorphism between the above $\underline{\operatorname{Mor}}(\cat B)$ and the dg-category of strictly unital $A_{\infty}$-functors $\basering k[[1]] \to \cat B$.

Furthermore, we will identify (up to isomorphism) the category $\Fun_{\basering k}(\basering k[[1]], H^0(\cat B))$ with the category of morphisms $\operatorname{Mor}(H^0(\cat B))$ of $H^0(\cat B)$, whose objects are morphisms in $H^0(\cat B)$ and morphisms are given by commutative diagrams in $H^0(\cat B)$.

Under the identifications $\RHom(\basering k[[1]], \cat B) \cong \underline{\operatorname{Mor}}(\cat B)$ and $\Fun_{\basering k}(\basering k[[1]], H^0(\cat B)) \cong \operatorname{Mor}(H^0(\cat B))$, the functor \eqref{equation:functor_H0_underlying} is described as follows:
\begin{equation}
    \begin{split}
        H^0(\underline{\operatorname{Mor}}(\cat B)) &\to \operatorname{Mor}(H^0(\cat B)), \\
        f & \mapsto [f], \\
        [(u,v,h)] & \mapsto ([u], [v]),
    \end{split}
\end{equation}
where $[-]$ denotes equivalence classes of morphisms in $H^0$. It is now easy to show that the above functor is indeed full and essentially surjective.
\end{proof}

We may now prove strongness. We need to show that the functor
\begin{equation} \label{equation:der5}
\mathrm{dia}_{[1], I} \colon \QFun(\basering k[[1] \times I], \cat A) \to \Fun_{\basering k}(\basering k[1], \cat A)
\end{equation}
is full and essentially surjective.

First, we observe that
\begin{align*}
\basering k[[1] \times I] & \cong \basering k[[1]] \otimes \basering k[I] \\
&\cong  \basering k[[1]] \lotimes \basering k[I],
\end{align*}
recalling that any dg-category of the form $\basering k[\cat I]$ is h-projective, hence h-flat, so that the derived tensor product can be identified with the ordinary one. Next, we need to use the quasi-equivalence \eqref{equation:RHom_dgcat}. In particular, taking $H^0$, we obtain:
\[
\QFun(\basering k[[1]] \otimes \basering k[I], \cat A) \cong \QFun(\basering k[[1]], \RHom(\basering k[I], \cat A)).
\]
The above functor $\mathrm{dia}_{[1], I}$ is then given by
\begin{align*}
\QFun(\basering k[[1]], \RHom(\basering k[I], \cat A)) & \to \Fun_{\basering k}(\basering k[[1]], \QFun(\basering k[I], \cat A)), \\
F & \mapsto H^0(F),
\end{align*}
which is full and essentially surjective by Proposition \ref{proposition:mor_cat_comparison} applied to $\cat B = \RHom(\basering k[I], \cat A)$.
\subsubsection{Stability}

Here, we prove that the derivator associated to a (homotopically complete and cocomplete) dg-category $\cat A$ is a \emph{stable} derivator.

Let $\derivator_{\cat A}$ the derivator associated to $\cat A$. We recall that, from a derivator-theoretic perspective, the loop space functor $\Omega$ is defined as the following composition:
\begin{equation} \label{equation:loopspace_derivator}
    \Omega = (\derivator_{\cat A}(e) \xrightarrow{(1,1)_!} \derivator_{\cat A}(\lrcorner) \xrightarrow{(i_{\lrcorner})_*} \derivator_{\cat A}(\square) \xrightarrow{(0,0)^*} \derivator_{\cat A}(e)),
\end{equation}
where $i_{\lrcorner} \colon \lrcorner \hookrightarrow \square$ is the following inclusion of diagrams:
\begin{equation}
    \begin{tikzcd}
                  & {(0,1)} \arrow[d, "e_0"] \\
{(1,0)} \arrow[r, "e_1"] & {(1,1)}          
\end{tikzcd} \hookrightarrow  \begin{tikzcd}
{(0,0)} \arrow[r] \arrow[d] & {(0,1)} \arrow[d, "e_0"] \\
{(1,0)} \arrow[r, "e_1"]           & {(1,1)}          
\end{tikzcd}
\end{equation}

We now recall \cite[Theorem 7.1]{groth-ponto-shulman-mayervietoris}, which asserts that a pointed derivator is stable if and only if the suspension-loop space adjunction $\Sigma \dashv \Omega$ is an adjoint equivalence. Clearly, this is equivalent to $\Omega$ being an equivalence.

\begin{proposition}
    Let $\cat A$ be a homotopically complete and cocomplete dg-category. Then, with the identification $\derivator_{\cat A}(e) =  H^0(\cat A)$, the loop space functor $\Omega$ defined by \eqref{equation:loopspace_derivator} is isomorphic to the $(-1)$-shift $[-1]$. In particular, $\Omega$ is an equivalence and the derivator $\derivator_{\cat A}$ associated to $\cat A$ is stable.
\end{proposition}
\begin{proof}
    Thanks to \cite[Proposition 1.23, Proposition 3.6, see also the paragraph before Definition 3.16]{groth-derivators} we easily deduce that an object $F \in \derivator_{\cat A}(\lrcorner)$, which is a quasi-functor $F \colon \basering k[\lrcorner] \to \cat A$, is in the essential image of $(1,1)_!$ if and only if
    \[
    F_{(0,1)} \cong 0, \qquad F_{(1,0)} \cong 0,
    \]
    in $\dercomp(\cat A)$. Moreover, $(1,1)_!$ is fully faithful, which implies that
    \[
    ((1,1)_!(X))_{(1,1)} = (1,1)^*((1,1)_!(X)) \cong X,
    \]
    naturally in $X \in \derivator_{\cat A}(e) \cong H^0(\cat A)$.

    We take a closer look at the right homotopy Kan extension $(i_{\lrcorner})_*(F)$, for a given quasi-functor $F \colon \basering k[\lrcorner] \to \cat A$. As a quasi-functor $k[\square] \to \cat A$, it is defined by:
    \[
    (i_{\lrcorner})_*(F)_? = \RHom_{\basering k[\lrcorner]}(\basering k[\square](?,i_{\lrcorner}(-)), F).
    \]
    Recalling the definition \eqref{equation:loopspace_derivator} of $\Omega$, we consider the evaluation of the above quasi-functor in $(0,0)$:
    \[
    (i_{\lrcorner})_*(F)_{(0,0)} = \RHom_{\basering k[\lrcorner]}(\basering k[\square]((0,0),i_{\lrcorner}(-)), F).
    \]

    A direct computation shows that we have a short exact sequence
    \begin{equation} \label{equation:exactness_pullback}
     0 \to \basering k[\lrcorner]((1,1),-) \xrightarrow{\begin{psmallmatrix}
         e_0^* \\ -e_1^*
     \end{psmallmatrix}} k[\lrcorner]((0,1),-) \oplus k[\lrcorner]((1,0),-) \to \basering k[\square]((0,0),i_{\lrcorner}(-)) \to 0,
     \end{equation}
     in $Z^0(\ldgm{\basering k[\lrcorner]})$, which induces a distinguished triangle in $\dercomp(\opp{\basering k[\lrcorner]})$. By applying $\RHom_{\basering k[\lrcorner]}(-,F)$, we find a distinguished triangle
     \[
     (i_{\lrcorner})_*(F)_{(0,0)} \to \RHom_{\basering k[\lrcorner]}(k[\lrcorner]((0,1),-) \oplus k[\lrcorner]((1,0),-), F) \to \RHom_{\basering k[\lrcorner]}(\basering k[\lrcorner]((1,1),-), F)
     \]
     in $\dercomp(\cat A)$. Applying the derived Yoneda lemma \eqref{equation:derivedYoneda} we immediately find the distinguished triangle
    \begin{equation} \label{equation:hopullback_hofiber}
         (i_{\lrcorner})_*(F)_{(0,0)} \to F_{(0,1)} \oplus F_{(1,0)} \xrightarrow{(F_{e_0}, -F_{e_1})} F_{(1,1)}.
     \end{equation}
     Now, we take $X \in H^0(\cat A)$ and consider $(1,1)_!(X)$. Thanks to the above remark, we know that
     \[
     ((1,1)_!(X))_{(0,1)} \cong 0, \quad ((1,1)_!(X))_{(1,0)} \cong 0, \quad ((1,1)_!(X))_{(1,1)} \cong X.
     \]
     We immediately conclude that
     \[
     (i_{\lrcorner})_*((1,1)_!(X))_{(0,0)} \cong X,
     \]
     from which we may immediately deduce that the definition \eqref{equation:loopspace_derivator} of $\Omega$ is indeed the functor $X \mapsto X[-1]$, from which the proof is complete.
\end{proof}

\subsection{The derivator of a pretriangulated dg-category, on finite direct categories}

Theorem \ref{theorem:derivator_dg} has a variant, which will allow us to define a derivator associated to any pretriangulated dg-category (without the stronger assumption of general homotopical (co)completeness). We recall the notion of \emph{finite direct category} (cf. Definition \ref{definition:finite_directed}). We may prove:
\begin{theorem} \label{theorem:derivator_dg_finite}
Let $\cat A$ be a pretriangulated dg-category. Then the $2$-functor
\[
    \begin{split}
        \derivator_{\cat A} \colon \opp{\kat{FinDir}} & \to \kat{CAT}, \\
        I & \mapsto \QFun(\basering k[I], \cat A).
    \end{split}
\]
is a derivator (defined on finite direct categories).
\end{theorem}
\begin{proof}
    It follows with the same arguments of the proof of Theorem \ref{theorem:derivator_dg}, but we need to make sure that homotopy Kan extensions along a given functor $u \colon I \to J$ between finite direct categories exist as quasi-functors. We will concentrate on right homotopy Kan extensions, the argument for left homotopy Kan extensions will follow by duality.

    We know from Proposition \ref{proposition:pretr_dgcat_finite_holims} that $\cat A$ admits all $W$-weighted homotopy limits defined on $\basering k[I]$ where $I$ is a finite direct category and $W \in \pretr(\opp{\basering k[I]})$. Thanks to the description of right homotopy Kan extensions in terms of weighted homotopy limits (see Proposition \ref{proposition:righthokan_holim}), we only need to show that, for all $j \in J$, the left $\basering k[I]$-dg-module
    \[
    \basering k[J](j,u(-))
    \]
    lies in $\pretr(\opp{\basering k[I]})$. It is enough to show that it admits a \emph{finite free resolution}.

    Call $M=\basering k[J](j,u(-))$. A free resolution of $M$ is given by the bar complex. The key idea is that, since $I$ is a finite direct category, we may find a close variant of the bar complex which is finite, and every term will also be a finite direct sum of representable left $\basering k[I]$-modules.
    
    Let $d \colon \Ob(I) \to \{0,\cdots, n\}$ be a chosen degree function associated to $I$, and assume that $n$ is actually the maximum value attained by $d$. We define a resolution whose $k$-th term is
    \[
    B_k = \bigoplus_{\substack{i_0, \cdots, i_k \in \Ob(I) \\ d(i_0) < \cdots < d(i_k)}} \basering k[I](i_k,-) \otimes_{\basering k} \basering k[I](i_{k-1},i_k) \otimes_{\basering k} \cdots \otimes_{\basering k} \basering k[I](i_0,i_1) \otimes_{\basering k} M_{i_0}, 
    \]
    if $k \in \{0,\ldots,n\}$. Moreover, we set:
    \[
    B_{n+1} = B_{n+2} = \cdots = 0.
    \]
    The boundary operator
    \[
    \partial_{k} \colon B_{k} \to B_{k-1}
    \]
    is defined as the usual one of the bar complex, with components:
    \[
    \partial_{k}(f_k \otimes \ldots f_0 \ldots \otimes m) = \sum_{j=0}^{k-1} (-1)^j f_k \otimes \cdots \otimes f_{j+1}f_j \otimes \cdots \otimes m + (-1)^k  f_k \otimes \ldots \otimes f_0m. 
    \]
    Moreover, we have a map
    \begin{equation*}
        \begin{split}
            \partial_0 \colon B_0 &\to M, \\
            \partial_0(f_0 \otimes m) &= f_0m.
        \end{split}
    \end{equation*}
    Now, let $m$ be the non-negative integer such that. for all set of objects $\{i_0,\ldots, i_m\}$ in $I$ such that $d(i_0) < \cdots < d(i_m)$, we also have $d(i_m)=n$. Such $m$ exists since $I$ is a finite category. The fact that
    \[
    B_m \xrightarrow{\partial_m} B_{m-1} \xrightarrow{\partial_{m-1}} \cdots \xrightarrow{\partial_0} M
    \]
    is an exact sequence follows from standard arguments. Moreover, we may directly show that
    \[
    0 \to B_m \xrightarrow{\partial_{m}} B_{m-1}
    \]
    is also exact, namely, $\partial_{m}$ is injective (by convention we have $B_{-1}=M$). This fundamentally relies on the assumption that $I$ is a finite direct category. Let us show this in the case that $m=1$; the general case is more complicated but conceptually analogous. We have:
    \begin{align*}
    B_1 &= \bigoplus_{\substack{i_0, i_1 \in \Ob(I) \\ d(i_0) < d(i_1)=1}} \basering k[I](i_1,-) \otimes_{\basering k} \basering k[I](i_0,i_1) \otimes_{\basering k} M_{i_0},\\
    B_0 &= \bigoplus_{i_0 \in \Ob(I)} \basering k[I](i_0,-) \otimes_{\basering k} M_{i_0}.
    \end{align*}
    Let $j \in \Ob(I)$. If $d(j)=1$, then $\basering k[I](i_1, j)=0$ unless $j=i_1$, and $\basering k[I](j, j)\cong \basering k$.  Applying the natural isomorphisms $\basering k \otimes_{\basering k} ? \cong ?$, we obtain:
    \begin{align*}
    B_1 &= \bigoplus_{\substack{i_0 \in \Ob(I) \\ d(i_0) < 1}} \basering k[I](i_0,j) \otimes_{\basering k} M_{i_0},\\
    B_0 &= \bigoplus_{i_0 \in \Ob(I)} \basering k[I](i_0,j) \otimes_{\basering k} M_{i_0}.
    \end{align*}
    We can also write $B_0$ as follows:
    \[
    B_0 = \left(\bigoplus_{\substack{i_0 \in \Ob(I) \\ d(i_0) < 1}} \basering k[I](i_0,j) \otimes_{\basering k} M_{i_0}\right) \oplus M_j.
    \]
    The morphism $\partial_1 \colon B_1 \to B_0$ is defined (on each component) as follows:
    \[
    \partial_1(f \otimes m) = (f \otimes m, -fm) \in \left(\bigoplus_{\substack{i_0 \in \Ob(I) \\ d(i_0) < 1}} \basering k[I](i_0,j) \otimes_{\basering k} M_{i_0}\right) \oplus M_j
    \]
    This is an injective morphism. Moreover, if $j$ is such that $d(j)<1$, we have:
    \[
    B_1 = \bigoplus_{\substack{i_0, i_1 \in \Ob(I) \\ d(i_0) < d(i_1)=1}} \basering k[I](i_1,j) \otimes_{\basering k} \basering k[I](i_0,i_1) \otimes_{\basering k} M_{i_0} = 0,
    \]
    hence $\partial_1$ is injective also in this case.
    
    In the end, the complex $B_\bullet$ yields a finite free resolution of $M$, from which we conclude that indeed $M \in \pretr(\opp{\basering k[I]})$.

\end{proof}

\section{The derivator of a Frobenius exact category}
\label{section:derivator_Frobenius}

In \cite{keller-vossieck-derived} Keller and Vossieck proved that, if $\cat F$ is a Frobenius exact category and $\operatorname{Proj}(\cat F)$ the subcategory of projectives objects of $\cat F$, then the stable category $\underline{\cat F}$ is equivalent to the homotopy category of acyclic complex of objects of $\operatorname{Proj}(\cat F)$
\begin{equation}\label{equation:Frobenius-dg}
  \underline{\cat F}\simeq \operatorname{K_{ac}}(\operatorname{Proj}(\cat F)).
\end{equation}

This homotopy category admits a standard dg-enhancement, namely the full dg-subcategory $\operatorname{Ch^{dg}_{ac}}(\operatorname{Proj}(\cat F))$ of the pretriangulated dg-category $\operatorname{Ch^{dg}}(\operatorname{Proj}(\cat F))$. Note that $\operatorname{Ch^{dg}_{ac}}(\operatorname{Proj}(\cat F))$ is closed under shifts and cones in $\operatorname{Ch^{dg}}(\operatorname{Proj}(\cat F))$ by~\cite[Lemma~10.3]{buehler}, so it is itself pretriangulated and we have
 \[
    H^0(\operatorname{Ch^{dg}_{ac}}(\operatorname{Proj}(\cat F))) = \operatorname{K_{ac}}(\operatorname{Proj}(\cat F)).
    \]
 As a consequence $\operatorname{Ch^{dg}_{ac}}(\operatorname{Proj}(\cat F))$ is also the dg-enhancement of $\underline{\cat F}$.

Thanks to the main Theorem~\ref{theorem:derivator_dg_finite}, we can then explicitly write the derivator associated to the Frobenius category:
\begin{equation} \label{equation:derivatordg_Frobenius}
    \begin{split}
        \derivator_{\cat F} \colon I & \mapsto \QFun(\basering k[I], \operatorname{Ch^{dg}_{ac}}(\operatorname{Proj}(\cat F))).
    \end{split}
\end{equation}

The aim of this section is giving a better description of this derivator \eqref{equation:derivatordg_Frobenius}. Namely, we would like to prove that 
\begin{equation*}
         \QFun(\basering k[I], \operatorname{Ch^{dg}_{ac}}(\operatorname{Proj}(\cat F))) \cong \operatorname{K_{ac}}(\operatorname{Proj}(\cat{F}^I)).
     \end{equation*}
where $\operatorname{Proj}(\cat{F}^I)$ is the category of projective objects in the exact category of functors $I \to \cat F$ with the termwise exact structure.

We denote by $\operatorname{Ch^{dg}_{tc}}(\operatorname{Proj}(\cat F)^{I})$ the pretriangulated full dg-subcategory of $\operatorname{Ch^{dg}_{ac}}(\operatorname{Proj}(\cat F)^{I})$ spanned by the \emph{termwise contractible} complexes of objects of $\operatorname{Proj}(\cat F)^{I}$, namely, complexes $F^\bullet$ such that $F^\bullet_i$ is a contractible complex of objects of $\operatorname{Proj}(\cat F)$, for all $i \in I$.
Here $\operatorname{Proj}(\cat{F})^I$ is the category of functors $I \to \Proj(\cat F)$.
We also denote:
    \[
    H^0(\operatorname{Ch^{dg}_{tc}}(\operatorname{Proj}(\cat F)^{I})) = \operatorname{K_{tc}}(\operatorname{Proj}(\cat F)^{I}).
    \]
We remark that a termwise contractible complex is certainly acyclic but not necessarily contractible as a complex of objects in $\cat F^{I}$.

Given a functor $u\colon I\to J$ in $\kat{Cat}$, we will need the existence of certain Kan extensions $u_!\colon {\cat F}^I\to{\cat F}^J$ or $u_*\colon {\cat F}^I\to{\cat F}^J$. There is a well-known criterion for that.

\begin{proposition}[{\cite[Theorem~X.3.1]{maclane-categories}}]\label{proposition:Kan-extensions-existence}
Let $\cat F$ be any category and $u\colon I\to J$ be a functor in $\kat{Cat}$.
\begin{enumerate}
\item If colimits of all diagrams of shapes $u/j$ for each $j\in J$ exist in $\cat F$, then the restriction functor $u^*\colon {\cat F}^J\to{\cat F}^I$ has a left adjoint $u_!\colon {\cat F}^I\to{\cat F}^J$.
\item If limits of all diagrams of shapes $j/u$ for each $j\in J$ exist in $\cat F$, then the restriction functor $u^*\colon {\cat F}^J\to{\cat F}^I$ has a right adjoint $u_*\colon {\cat F}^I\to{\cat F}^J$.
\end{enumerate}
\end{proposition}

Any additive category $\cat F$ has finite products and coproducts. Although this is a very limited collection of limits and colimits, it is still enough for the existence of Kan extensions in the following situation.

\begin{corollary}\label{corollary:Kan-extensions-existence}
Let $\cat F$ be an additive category and $u\colon I\to J$ be a functor in $\kat{Cat}$.
\begin{enumerate}
\item If the slice category $u/j$ is a disjoint union of finitely many categories with terminal objects for each $j\in J$, then the left Kan extension functor $u_!\colon {\cat F}^I\to{\cat F}^J$ exists.
\item If the slice category $j/u$ is a disjoint union of finitely many categories with initial objects for each $j\in J$, then the right Kan extension functor $u_*\colon {\cat F}^I\to{\cat F}^J$ exists.
\end{enumerate}
\end{corollary}

\begin{proof}
We will only prove part (1), part (2) is dual.
Let $j\in J$ and consider a decomposition $u/j=\coprod_{t=1}^{n_j} K_t$, where $n_j\ge 0$ is an integer and $K_t\in\kat{Cat}$ has a terminal object $k_t\in K_t$ for each $t$. If $X\colon \coprod_{t=1}^{n_j} K_t\to \cat{F}$ is a diagram of the shape $u/j$ in $\cat{F}$, the colimit clearly exists and it is just the finite coproduct $\bigoplus_{t=1}^{n_j}X_{k_t}$. It remains to apply Proposition~\ref{proposition:Kan-extensions-existence}(1).
\end{proof}

\begin{remark}\label{remark:Kan-extensions-existence}
Note that in the situation of Corollary~\ref{corollary:Kan-extensions-existence}(1), we have for each $X\in\cat{F}^I$ and $j\in I$ an isomorphism $u_!(X)_j\cong\coprod_{t=1}^{n_j} X_{k_t}$ that is functorial in $X$, where $(k_t, u(k_t)\to j)$ are the terminal objects of the connected components of the slice category $u/j$. A dual observation applies Corollary~\ref{corollary:Kan-extensions-existence}(2).
\end{remark}

Now we aim to obtain a description of the projective objects in $\cat{F}^I$ with the termwise exact structure.
Recall that an exact category $\cat E$ is \emph{weakly idempotent complete} if any section $s: x \to y$ has a kernel in $\cat E$ or, equivalently, any retraction $r: y \to z$ has a cokernel in $\cat E$ \cite[\S7]{buehler}.

\begin{proposition}\label{proposition:structure_projectives_F^I}
Let $\cat{F}$ be a weakly idempotent complete exact category with enough projectives and $I$ be a finite direct category. Then $X \in \operatorname{Proj}(\cat F^I)$ if and only if\
\[X \cong \bigoplus_{j\in I}j_!Y_j,\]
for some $Y_j \in \Proj{\cat{F}}$.
In particular, projective functors in $\cat{F}^I$ are termwise projective and there is an inclusion of categories:
        \[
        \operatorname{Proj}(\cat F^I) \hookrightarrow \operatorname{Proj}(\cat F)^I.
        \]
\end{proposition}

The proof is based on the following lemma. We call $i \in I$ a \emph{minimal object} if $I(j,i)=\emptyset$ for any $j \ne i$. Any non-empty finite direct category clearly has a minimal object.

\begin{lemma}\label{lemma:splitmono}
Let $\cat{F}$ and $I$ be as in the proposition, let $i\in I$ be minimal and denote $J\coloneqq I\setminus \{i\}$. If 
\[i\colon e \to I, \hspace{1 cm} k\colon J \to I \]
are the obvious inclusion functors and $X \in \operatorname{Proj}(\cat F^I)$, then $i^*X\in\Proj(\cat F)$, the counit of adjunction \[\epsilon^i\colon i_!i^*X \to X\] is a split monomorphism and \[\coker\epsilon^i\cong k_!Y\] for some $Y \in \operatorname{Proj}(\cat F^J).$
\end{lemma}
\begin{proof}
Let us denote $\cat{P}=\Proj(\cat{F})$.
Consider $j \in I$, which we again identify with a functor $j\colon e\to I$. Then both the left adjoint $j_!$ and the right adjoint $j_*$ to $j^*\colon\cat F^I\to\cat F$ exist and are exact by Corollary~\ref{corollary:Kan-extensions-existence} and Remark~\ref{remark:Kan-extensions-existence}. Moreover, if we denote by
\[\epsilon^j:j_!j^*X \to X.\]
the counit of adjunction for $X$, then $j^*\epsilon^j$ is an isomorphism since the functor $j\colon e \to I$ is fully faithful.
Since the restriction $j^*$ is the left adjoint of an exact functor we have in particular that $j^*X \in \cat P$, for any $X \in \operatorname{Proj}(\cat F^I)$.  We can now consider the following map 
\[\oplus\epsilon\colon\bigoplus_{j \in I} j_!j^*X \stackrel{\epsilon^j}{\longrightarrow}X\]
where, since $j^*X \in \cat P$ as we discussed above, $\bigoplus_{j \in I} j_!j^*X \in \operatorname{Proj}(\cat F^I)$ because $j_!$ preserves projective objects. Moreover, since $j^*\epsilon^j$ is an isomorphism for all $j\in I$, $\oplus \epsilon$ is a termwise split deflation. Then, since we have a deflation to $X$ which is projective, $\oplus \epsilon$ is in fact a split epimorphism in $\cat{F}^I$ and $X$ is a direct summand of $\bigoplus_{j \in I} j_!j^*X$. As a consequence, to prove that $\epsilon^i\colon i_!i^*X \to X$ is a split monomorphism, we can assume, without loss of generality, that $X \cong j_!Y$ for some $Y \in \cat P$ and $j \in I$. Now, we have two cases:
\begin{itemize}
\item[(1)] If $j=i$, then $\epsilon^i\colon i_!i^*i_!Y\rightarrow i_!Y$ is an isomorphism. 
\item[(2)] If $j \ne i$, then $\epsilon^j\colon i_!i^*j_!Y\rightarrow j_!Y$
is trivially a split monomorphism, because \[i^*j_!Y =\coprod_{I(j,i)}Y=0\] by Remark~\ref{remark:Kan-extensions-existence}, as $I(j,i)=\emptyset$.
\end{itemize}
In particular, $\epsilon^i$ is an inflation because we assume that $\cat F$ is weakly idempotent complete.
Moreover, as $i^*\epsilon^i$ is an isomorphism, we have that $(\coker \epsilon^i)_i=0$, and so $\coker\epsilon^i$ is in the essential image of $k_!$, which is just the left extension by zero functor. Then we can put $Y=k^*(\coker\epsilon^i)\in\cat{F}^J$. Since $k_!$ is fully faithful and exact and $k_!Y\cong k_!k^*(\coker\epsilon^i)\cong\coker\epsilon^i$ is projective in $\cat{F}^I$, it follows that $Y\in\Proj(\cat{F}^J)$.
    \end{proof}

\begin{proof}[Proof of Proposition~\ref{proposition:structure_projectives_F^I}]
    We proceed by induction on the number of objects of $I$. The case $I=e$ is trivial. Otherwise, let $i\in I$ be minimal. Then by Lemma~\ref{lemma:splitmono}, we know that we can write \[X \cong i_!i^*X \oplus k_!Z,\] for some $X\in\Proj(\cat{F})$ and $Z \in \operatorname{Proj}(\cat F^J)$. Then, we can use the inductive hypothesis to decompose $Z$ and such decomposition is then preserved by $k_!$.
\end{proof}

The key idea behind this section is the description of the category $\operatorname{K_{ac}}(\operatorname{Proj}(\cat F^I))$ which we give in the next proposition. Let us first recall the following definition.

\begin{definition}[{cf. \cite[\S 2.3]{bondal-larsen-lunts-grothendieck}}] 
    Let $\mathcal{T}$ be a triangulated category. We call \emph{semiorthogonal decomposition} a pair $\langle \mathcal{T}_1,\mathcal{T}_2\rangle$ of strictly full triangulated subcategories such that: 
    \begin{itemize}
        \item[(a)] $\Hom_{\mathcal{T}}(T_2,T_1)=0$ for any $T_1 \in \mathcal{T}_1$ and $T_2 \in \mathcal{T}_2$.
        \item[(b)] For any objects $T \in \mathcal{T}$, there exists a triangle in $\mathcal{T}$ \[T_2\to T \to T_1\to,\]
        where $T_1 \in \mathcal{T}_1$ and $T_2 \in \mathcal{T}_2$.
    \end{itemize}
\end{definition}

\begin{proposition} \label{proposition:projective_frobenius_semiorthogonal} 
    Let $\cat F$ be a weakly idempotent complete Frobenius exact category and let $I$ be a finite direct category. Denote by
        \begin{align*}
        \operatorname{incl}_1 \colon \operatorname{K_{tc}}(\operatorname{Proj}(\cat F)^I) & \hookrightarrow \operatorname{K_{ac}}(\operatorname{Proj}(\cat F)^I), \\
        \operatorname{incl}_2 \colon \operatorname{K_{ac}}(\operatorname{Proj}(\cat F^I)) &\hookrightarrow \operatorname{K_{ac}}(\operatorname{Proj}(\cat F)^I),
        \end{align*}
        the inclusion functors. They induce the following semiorthogonal decomposition:
        \begin{equation}
            \operatorname{K_{ac}}(\operatorname{Proj}(\cat F)^I) = \langle   \operatorname{K_{tc}}(\operatorname{Proj}(\cat F)^I), \operatorname{K_{ac}}(\operatorname{Proj}(\cat F^I))\rangle.
        \end{equation}
        In particular, $\operatorname{K_{ac}}(\operatorname{Proj}(\cat F^I))$ is equivalent to the following Verdier quotient:
        \[
        \operatorname{K_{ac}}(\operatorname{Proj}(\cat F^I)) \cong \operatorname{K_{ac}}(\operatorname{Proj}(\cat F)^I) / \operatorname{K_{tc}}(\operatorname{Proj}(\cat F)^I).
        \]
\end{proposition}

The proof will again use the induction on the number of objects of $I$ and,
to ease the notation, we denote the subcategory $\operatorname{Proj}(\cat F)$ of projective objects in $\cat F$ by $\cat P$.

As in Lemma~\ref{lemma:splitmono}, let $i \in I$ be a minimal object, i.e.\  such that $I(j,i)=\emptyset$ for any $j \ne i$, and let $J\coloneqq I\setminus \{i\}$ and we denote by
\begin{equation}\label{equation:minimial-object-functors}
i\colon e \to I, \hspace{1 cm} k\colon J \to I
\end{equation}
the corresponding full embeddings. As in the proof of Lemma~\ref{lemma:splitmono}, we apply Corollary~\ref{corollary:Kan-extensions-existence} and Remark~\ref{remark:Kan-extensions-existence}, and see that both $i^*$ and $k^*$ have exact left adjoint functors
 \[i_!\colon \cat F \to \cat F^I, \hspace{1 cm} k_!\colon \cat F^J \to \cat F^I, \]
respectively. Moreover, $i_!$ and $k_!$ are fully faithful and send projectives to projectives, so they induce dg-functors
\[i_!\colon \operatorname{Ch^{dg}_{ac}}(\cat P) \to \operatorname{Ch^{dg}_{ac}}(\operatorname{Proj}(\cat F^I)), \hspace{1 cm} k_!\colon \operatorname{Ch^{dg}_{ac}}(\operatorname{Proj}(\cat F^J)) \to \operatorname{Ch^{dg}_{ac}}(\operatorname{Proj}(\cat F^I)).\]
Moreover, by taking $Z^0$ or $H^0$, we respectively get 
\[i_!\colon \operatorname{Ch_{ac}}(\cat P) \to \operatorname{Ch_{ac}}(\operatorname{Proj}(\cat F^I)), \hspace{1 cm} k_!\colon \operatorname{Ch_{ac}}(\operatorname{Proj}(\cat F^J)) \to \operatorname{Ch_{ac}}(\operatorname{Proj}(\cat F^I)),\] 
and 
\[i_!\colon \operatorname{K_{ac}}(\cat P) \to \operatorname{K_{ac}}(\operatorname{Proj}(\cat F^I)), \hspace{1 cm} k_!\colon \operatorname{K_{ac}}(\operatorname{Proj}(\cat F^J)) \to \operatorname{K_{ac}}(\operatorname{Proj}(\cat F^I)).\] 
Observe that, by Proposition~\ref{proposition:structure_projectives_F^I} \[ \operatorname{K_{ac}}(\operatorname{Proj}(\cat F^I))\subseteq \operatorname{K_{ac}}(\cat P^I).\]

The inductive step will be taken care of by the following two lemmas.

\begin{lemma}\label{lemma:SOdecomp}
    There is a semiorthogonal decomposition of $\operatorname{K_{ac}}(\cat P^I)$ of the form \[\langle \operatorname{K_{i-c}}(\cat P^I), \operatorname{K_{ac}}(i_!\cat P) \rangle,\]
where $i_!\cat P$ is the essential image of 
\begin{equation*} 
        \begin{split}
            i_!\vert_{\cat P}\colon \cat P & \to \cat P^I \\
            X & \mapsto i_!X
        \end{split}
    \end{equation*}
    and \[\operatorname{K_{i-c}}(\cat P^I)\coloneqq\left\{ Z \in \operatorname{K_{ac}}(\cat P^I) \colon i^*Z \in \operatorname{K_{ac}}(\cat P) \text{ is contractible}\right\}.\]
    \end{lemma}
    \begin{proof}
        We need to prove 
    \begin{itemize}
        \item[(a)] $\Hom_{\operatorname{K_{ac}}(\cat P^I)}(\operatorname{K_{ac}}(i_!\cat P),\operatorname{K_{i-c}}(\cat P^I))=0$
        \item[(b)] For any $X \in \operatorname{K_{ac}}(\cat P^I)$ there exists a triangle \[X_{i-p} \to X \to X_{i-c} \to\]
        where $X_{i-p} \in \operatorname{K_{ac}}(i_!\cat P)$ and $X_{i-c} \in \operatorname{K_{i-c}}(\cat P^I).$
    \end{itemize}
To prove (a), let $Y \in \operatorname{K_{ac}}(i_!\cat P)$, $Z \in \operatorname{K_{i-c}}(\cat P^I)$ and $f\colon Y \to Z$ a morphism in $\operatorname{K_{ac}}(\cat P^I)$. Then, by definition, $Y \cong i_!Y'$ for $Y' \in \cat P$ and we have an adjunction \begin{equation}\label{adj_i}
    (i_!,i^*) \colon \operatorname{K_{ac}}(\cat P) \rightleftharpoons \operatorname{K_{ac}}(\cat P^I).\end{equation}
A map $f\colon i_!Y' \to Z$ in $\operatorname{K_{ac}}(\cat P^I)$ 
corresponds to a map $g \colon Y' \to i^*Z$ in $\operatorname{K_{ac}}(\cat P),$ by the adjunction. Since $i^*Z$ is contractible, the map $g$ is nullhomotopic so, by adjunction, it follows that also $f$ is nullhomotopic.

To prove (b), let $X \in \operatorname{K_{ac}}(\cat P^I)$  and consider the triangle 
\[i_!i^*X \stackrel{\epsilon^i}{\to} X \to \cone(\epsilon^i)\to \]
where $\epsilon^i$ is the counit of the adjunction $\eqref{adj_i}$, applied to $X$. Then $i_!i^*X$ is an object in $\operatorname{K_{ac}}(i_!\cat P)$, by definition. Since $i^*$ is a triangulated functor, \[i^*(\cone(\epsilon^i)) \cong \cone(i^*(\epsilon^i))\]
and, since $i^*\epsilon^i$ is an isomorphism we have that $\cone(i^*(\epsilon^i))$ and then $i^*(\cone(\epsilon^i))$ are contractible. So, by definition, $\cone(\epsilon^i) \in \operatorname{K_{i-c}}(\cat P^I)$.
    \end{proof}
To connect the semiorthogonal decomposition we found with the desired one, we need the following proposition.

\begin{lemma} \label{lemma:image-of-k-shriek}
The essential image of $k_!\colon\operatorname{K_{ac}}(\Proj(\cat F^J)) \to \operatorname{K_{ac}}(\Proj(\cat F^I))$ is equal to
$\operatorname{K_{i-c}}(\cat P^I)\cap \operatorname{K_{ac}}(\operatorname{Proj}(\cat F^I))$.
\end{lemma}
\begin{proof}
The inclusion $\Img k_!\subseteq\operatorname{K_{i-c}}(\cat P^I)\cap \operatorname{K_{ac}}(\Proj(\cat F^I))$ is easy. Any complex $X$ in the essential image of $k_!\colon\operatorname{Ch_{ac}}(\Proj(\cat F^J)) \to \operatorname{Ch_{ac}}(\Proj(\cat F^I))$ satisfies $X_i=0$, so is clearly contractible in the $i$-th component.

It remains to prove that, any $X \in \operatorname{K_{i-c}}(\cat P^I)\cap \operatorname{K_{ac}}(\Proj(\cat F^I))$ is homotopy equivalent to some $Z$ such that $i^*Z$ is strictly zero. 
So suppose that $X \in \operatorname{K_{ac}}(\Proj(\cat F^I))$ such that $i^*X$ is contractible and consider again the map
\[\epsilon^i\colon i_!i^*X \to X.\]
Since $i_!$ is an additive functor, $i_!i^*X$ is also contractible in $\operatorname{K_{ac}}(\Proj(\cat F^I))$. Moreover, thanks to Lemma~\ref{lemma:splitmono}, we know that $\epsilon^i$, when viewed as a map in $\operatorname{Ch_{ac}}(\Proj(\cat F^I))$, is a degreewise split monomorphism, and so there is a degreewise split exact sequence
\[ 0 \to i_!i^*X \stackrel{\epsilon^i}{\to} X \stackrel{\theta}{\to} \coker(\epsilon)\to 0 \]
in $\operatorname{Ch_{ac}}(\Proj(\cat F^I))$ inducing a triangle in $\operatorname{K_{ac}}(\Proj(\cat F^I))$ of the form
\[i_!i^*X \stackrel{\epsilon^i}{\to} X \stackrel{\theta}{\to} \coker(\epsilon)\to, \]
where $\coker(\epsilon)$ is the cokernel in $\operatorname{Ch_{ac}}(\Proj(\cat F^I))$. 
Since $i_!i^*X$ is contractible, $\theta$ is an isomorphism in $\operatorname{K_{ac}}(\Proj(\cat F^I))$. Moreover, since it is a cokernel in the additive category $\operatorname{Ch_{ac}}(\Proj(\cat F^I))$, $i_*(\coker(\epsilon))=\coker(i^*\epsilon)$ is strictly zero because $i^*\epsilon$ is an isomorphism in $\operatorname{Ch_{ac}}(\cat P)$, as seen in the proof of Lemma~\ref{lemma:splitmono}.
\end{proof}

\begin{proof}[Proof of Proposition~\ref{proposition:projective_frobenius_semiorthogonal}]

We proceed by induction on the number of objects of $I$. 
    The case $I=e$ is trivial.
For the inductive step, let $i \in I$ be a minimal object, $J\coloneqq I\setminus \{i\}$ and use the notation as in~\eqref{equation:minimial-object-functors}.
We again need to prove that 
    \begin{itemize}
        \item[(a)] $\Hom_{\operatorname{K_{ac}}(\cat P^I)}(\operatorname{K_{ac}}(\Proj(\cat F^I)),\operatorname{K_{tc}}(\cat P^I))=0$
        \item[(b)] For any $X \in \operatorname{K_{ac}}(\cat P^I)$ there exists a triangle \[X_{p} \to X \to X_{tc} \to\]
        where $X_{p} \in \operatorname{K_{ac}}(\Proj(\cat F^I))$ and $X_{tc} \in \operatorname{K_{tc}}(\cat P^I).$
    \end{itemize}

To prove (a), consider $Y \in \operatorname{K_{ac}}(\Proj(\cat F^I))$ and $Z \in \operatorname{K_{tc}}(\cat P^I)$. By Lemma~\ref{lemma:SOdecomp}, there is a triangle $Y_{i-p} \to Y \to Y_{i-c} \to$ with $Y_{i-p} \in \operatorname{K_{ac}}(i_!\cat P)$, so $Y_{i-p}\cong i_!i^*Y_{i-p}$, and $Y_{i-c} \in \operatorname{K_{i-c}}(\cat P^I)$. As both $Y_{i-p}$ and $Y$ lie in $\operatorname{K_{ac}}(\Proj(\cat F^I))$, so does $Y_{i-c}$. In particular, $Y_{i-c}\cong k_!k^*Y_{i-c}$ by Lemma~\ref{lemma:image-of-k-shriek}. Now
\[
\Hom_{\operatorname{K_{ac}}(\cat P^I)}(Y_{i-p},Z) \cong \Hom_{\operatorname{K_{ac}}(\cat P^I)}(i_!i^*Y_{i-p},Z) \cong
\Hom_{\operatorname{K_{ac}}(\cat P)}(i^*Y_{i-p},i^*Z)=0
\]
since $i^*Z$ is contractible, and likewise
\[
\Hom_{\operatorname{K_{ac}}(\cat P^I)}(Y_{i-c},Z) \cong \Hom_{\operatorname{K_{ac}}(\cat P^I)}(k_!k^*Y_{i-c},Z) \cong
\Hom_{\operatorname{K_{ac}}(\cat P^J)}(k^*Y_{i-c},k^*Z)=0
\]
by the inductive hypothesis, since clearly $k^*Z\in\operatorname{K_{tc}}(\cat P^J)$. Together, we have just shown that $\Hom_{\operatorname{K_{ac}}(\cat P^I)}(Y,Z) = 0$.

(b) If $X\in \operatorname{K_{ac}}(\cat P^I)$, consider again the triangle $X_{i-p}\to X\to X_{i-c}\to$ with $X_{i-p} \in \operatorname{K_{ac}}(i_!\cat P)$ and $X_{i-c} \in \operatorname{K_{i-c}}(\cat P^I)$ provided by Lemma~\ref{lemma:SOdecomp}, and let
\[\epsilon^k\colon k_!k^*X_{i-c} \to X_{i-c}\]
be the counit of the adjunction $(k_!,k^*) \colon \operatorname{K_{ac}}(\cat P^J) \rightleftharpoons \operatorname{K_{ac}}(\cat P^I)$. We know that $k^*\epsilon^k$ is an isomorphism (since $k_!$ is fully faithful) and $i^*k_!=0$ (since $k_!$ acts as the left extension by zero). Let us now invoke the inductive hypothesis and consider a triangle
\[ Y_p\overset{\theta}\to k^*X_{i-c}\to Y_{tc}\to \]
in $\operatorname{K_{ac}}(\cat P^J)$ with $Y_{p} \in \operatorname{K_{ac}}(\Proj(\cat F^J))$ and $Y_{tc} \in \operatorname{K_{tc}}(\cat P^J)$. 

We claim that the cone of the composition
$k_!Y_p \overset{k_!\theta}\to k_!k^*X_{i-c} \overset{\epsilon^k}\to X_{i-c}$, which we denote by $\zeta$, is termwise contractible. Indeed, as $k^*\epsilon^k$ is an isomorphism, we have
\[
k^*\cone(\zeta)\cong\cone(k^*\zeta)\cong\cone(k^*k_!\theta)\cong\cone(\theta)\in\operatorname{K_{tc}}(\cat P^J).
\]
Since $i^*k_!Y_p=0$, we also have
\[ i^*\cone(\zeta) \cong \cone(i^*\zeta) \cong i^*X_{i-c} = 0 \]
in $\operatorname{K_{ac}}(\cat P)$. As the essential images of $i\colon e\to I$ and $k\colon J\to I$ span all objects of $I$, it follows that $\cone(\zeta)\in\operatorname{K_{tc}}(\cat P^I)$, proving the claim.

All in all, we have triangles
\[ 
k_!Y_p\overset{\zeta}\to X_{i-c}\overset{\pi}\to X_{tc}\to
\qquad\text{and}\qquad
X_{i-p}\to X\overset{\rho}\to X_{i-c}\to
\]
with $X_{i-p}, k_!Y_p\in\operatorname{K_{ac}}(\Proj(\cat F^I))$ and $X_{tc}\in\operatorname{K_{tc}}(\cat P^I)$, and by applying the octahedral axiom, we obtain the required triangle
\[ X_p \to X \overset{\pi\rho}\to X_{tc}\to \]
with $X_p\in\operatorname{K_{ac}}(\Proj(\cat F^I))$ and $X_{tc}\in\operatorname{K_{tc}}(\cat P^I)$.
\end{proof}

Proposition \ref{proposition:projective_frobenius_semiorthogonal} can be ``enhanced'' to dg-categories (and the dg-quotient), as we show in the following side result. This will use the theory of adjoint quasi-functors (see \cite{genovese-adjunctions}) and of dg-quotients (see \cite{drinfeld-dgquotients}).
\begin{corollary}
    The inclusion dg-functor
    \[
    j_2 \colon \operatorname{Ch_{ac}^{dg}}(\operatorname{Proj}(\cat F^I)) \hookrightarrow \operatorname{Ch_{ac}^{dg}}(\operatorname{Proj}(\cat F)^I)
    \]
    has a right adjoint quasi-functor
    \[
    Q \colon \operatorname{Ch_{ac}^{dg}}(\operatorname{Proj}(\cat F)^I) \to \operatorname{Ch_{ac}^{dg}}(\operatorname{Proj}(\cat F^I))
    \]
    in the sense of \cite{genovese-adjunctions}. In particular, we may identify the dg-category $\operatorname{Ch_{ac}^{dg}}(\operatorname{Proj}(\cat F^I))$ with the following dg-quotient (cf. \cite{drinfeld-dgquotients}):
    \begin{equation}
        \operatorname{Ch_{ac}^{dg}}(\operatorname{Proj}(\cat F^I)) \cong \operatorname{Ch_{ac}^{dg}}(\operatorname{Proj}(\cat F)^I) / \operatorname{Ch_{tc}^{dg}}(\operatorname{Proj}(\cat F)^I).
    \end{equation}
\end{corollary}
\begin{proof}
    Denote, provisionally:
    \begin{align*}
        \cat A &= \operatorname{Ch_{tc}^{dg}}(\operatorname{Proj}(\cat F)^I), \\
        \cat B &=  \operatorname{Ch_{ac}^{dg}}(\operatorname{Proj}(\cat F^I)), \\
        \cat C &= \operatorname{Ch_{ac}^{dg}}(\operatorname{Proj}(\cat F)^I).
    \end{align*}
    The existence of the right adjoint quasi-functor $Q \colon \cat C \to \cat B$ follows from the existence of the right adjoint of $H^0(j_2) = \operatorname{incl}_2$, which is guaranteed by the above Proposition \ref{proposition:projective_frobenius_semiorthogonal}, see \cite[Lemma 2.1.3]{genovese-ramos-gabrielpopescu} or \cite[Remark 3.9]{lowen-julia-tensordg-wellgen}.

    Next, we take any dg-category $\cat D$ and we consider the following quasi-functor given by precomposition with $Q$:
    \[
    Q^* \colon \RHom(\cat B, \cat D) \to \RHom(\cat C, \cat D).
    \]
    We want to prove that:
    \begin{enumerate}
        \item It is quasi-fully faithful, namely, it induces quasi-isomorphisms between Hom complexes.
        \item The essential image of $H^0(Q^*)$ is the category
        \[
        \dercomp_{\cat A}^{\mathrm{qf}}(\cat C, \cat D)
        \]
        of quasi-functors $F \colon \cat C \to \cat D$ such that for any $A \in \cat A$, the image $F(A)$ is a zero object in $H^0(\cat D)$.
    \end{enumerate}
    This will ensure that $Q^*$ induces a quasi-equivalence
    \[
    \RHom(\cat B, \cat D) \cong \RHom_{\cat A}(\cat C, \cat D),
    \]
    where $\RHom_{\cat A}(\cat C, \cat D)$ is the full dg-subcategory of $\RHom(\cat C, \cat D)$ spanned by the objects in $\dercomp_{\cat A}^{\mathrm{qf}}(\cat C, \cat D)$. This precisely means that $\cat B$ is the dg-quotient of $\cat C$ modulo $\cat A$. A variant of this result is mentioned in \cite[Remark 2.1.8]{genovese-ramos-gabrielpopescu}; a version for triangulated categories and Verdier quotients is discussed in \cite[\S 4.9]{krause-localizationtheory}. Here, for completeness, we will present the main steps necessary to achieve a proof.

    Fully faithfulness of $j_2$ means that the unit
    \[
    \eta \colon 1 \to  Q \circ j_2
    \]
    is an isomorphism of quasi-functors. This implies that the corresponding counit map
    \[
    \eta^* j_2^* \circ Q^* \to 1
    \]
    of the adjunction $Q^* \dashv j_2^*$ is also an isomorphism, from which (by taking graded cohomology $H^*$ everywhere) we may deduce that $Q^*$ is indeed quasi-fully faithful. This proves (1).

    Let us now check (2). Given a quasi-functor $X$, we will denote by $\Phi_{X}$ the induced functor in $H^0$. First, if $F \colon \cat B \to \cat D$ is a quasi-functor, we have $\Phi_F \Phi_Q (A) \cong 0$ in $H^0(\cat D)$, because $\Phi_Q(A) \cong 0$ for $A \in \cat A$, thanks to the semiorthogonal decomposition $\cat C = \langle \cat A, \cat B \rangle$. On the other hand, let $G \colon \cat C \to \cat D$ be a quasi-functor such that $\Phi_G(A) \cong 0$ in $H^0(\cat D)$ for all $A \in \cat A$. We may define
    \[
    G' = G \circ j_2 \colon \cat B \to \cat D.
    \]
    We want to prove that we have an isomorphism of quasi-functors $G' \circ Q \cong G$. Thanks to the semiorthogonal decomposition $\cat C = \langle \cat A, \cat B \rangle$, we have a distinguished triangle of quasi-functors:
    \[
    j_2 \circ Q \to 1 \to j_1 \circ Q',
    \]
    where $j_1 \colon \cat A \to \cat C$ is the inclusion and $Q'$ is its left adjoint quasi-functor. Composing the above triangle with $G$, we obtain the following distinguished triangle (of quasi-functors, if we assume $\cat D$ to be pretriangulated; if not, it will still be a distinguished triangle of dg-bimodules):
    \[
    G \circ j_2 \circ Q \to G \to G \circ j_1 \circ Q'.
    \]
    Since $G$ maps $\cat A$ to zero objects, we may conclude that $G \circ j_1 \circ Q' \cong 0$, whence $G \to G \circ j_2 \circ Q$ becomes an isomorphism of quasi-functors. This concludes the proof.
\end{proof}
The above Proposition \ref{proposition:projective_frobenius_semiorthogonal} allows us to describe the derivator associated to the dg-category $\operatorname{Ch^{dg}_{ac}}(\operatorname{Proj}(\cat F))$ in terms of $\operatorname{K_{ac}}(\operatorname{Proj}(\cat F^I))$:
\begin{theorem} \label{theorem:derivator-Frobenius}
     Let $\cat F$ be a weakly idempotent complete Frobenius category and let $I$ be a finite direct category. Then, we have an equivalence of categories:
     \begin{equation}
         \QFun(\basering k[I], \operatorname{Ch^{dg}_{ac}}(\Proj(\cat F))) \cong \operatorname{K_{ac}}(\Proj(\cat F^I)) \cong
         \operatorname{K_{ac}}(\Inj(\cat{F}^I)).
     \end{equation}
\end{theorem}

\begin{proof}
    Let us consider the dg-category of dg-functors $\Fundg(\basering k[I], \operatorname{Ch^{dg}_{ac}}(\operatorname{Proj}(\cat F)))$. It is straightforward to prove:
    \begin{equation} \label{equation:complexes_projF_dgfunct}
    \Fundg(\basering k[I], \operatorname{Ch^{dg}_{ac}}(\operatorname{Proj}(\cat F))) \cong \operatorname{Ch^{dg}_{ac}}(\operatorname{Proj}(\cat F)^I). \tag{$\ast$}
    \end{equation}
    From \cite[after Theorem 4.5]{keller-dgcat}, we know that
    \[
    \QFun(\basering k[I], \operatorname{Ch^{dg}_{ac}}(\operatorname{Proj}(\cat F))) \cong H^0(\Fundg(\basering k[I], \operatorname{Ch^{dg}_{ac}}(\operatorname{Proj}(\cat F))))[\Sigma^{-1}],
    \]
    where $\Sigma$ is the family of morphisms $\varphi \colon F \to G$ such that $\varphi_i \colon F_i \to G_i$ is an isomorphism in $\operatorname{K_{ac}}(\operatorname{Proj}(\cat F))$ for all $i \in I$. It is clear that, under the equivalence \eqref{equation:complexes_projF_dgfunct}, the cones of morphisms in $\Sigma$ yield precisely $\operatorname{K_{tc}}(\operatorname{Proj}(\cat F)^I)$. Hence, using Proposition \ref{proposition:projective_frobenius_semiorthogonal}, we get:
    \[
    \QFun(\basering k[I], \operatorname{Ch^{dg}_{ac}}(\operatorname{Proj}(\cat F))) \cong \operatorname{K_{ac}}(\operatorname{Proj}(\cat F)^I) / \operatorname{K_{tc}}(\operatorname{Proj}(\cat F)^I) \cong \operatorname{K_{ac}}(\operatorname{Proj}(\cat F^I)),
    \]
    as claimed.

    For the other equivalence, we use that $\opp{\cat F}$ is again a weakly idempotent complete Frobenius exact category with $\Proj(\opp{\cat F})=\opp{\Inj(\cat F)}=\opp{\Proj(\cat F)}$ and that $\opp{I}$ is a finite direct category. By the previous part, we have an equivalence
    \[
    \QFun(\basering k[\opp{I}], \operatorname{Ch^{dg}_{ac}}(\Proj(\opp{\cat F}))) \cong \operatorname{K_{ac}}(\Proj((\opp{\cat F})^{\opp{I}})).
    \]
    It remains to note that
    \[
    \QFun(\basering k[\opp{I}], \operatorname{Ch^{dg}_{ac}}(\Proj(\opp{\cat F}))) \cong
    \QFun(\opp{\basering k[I]}, \opp{\operatorname{Ch^{dg}_{ac}}(\Proj(\cat F))}) \cong
    \opp{\QFun(\basering k[I], \operatorname{Ch^{dg}_{ac}}(\Proj(\cat F)))}
    \]
    and
    \[
    \operatorname{K_{ac}}(\Proj((\opp{\cat F})^{\opp{I}})) \cong
    \operatorname{K_{ac}}(\opp{\Inj(\cat F^I)}) \cong
    \opp{\operatorname{K_{ac}}(\Inj(\cat F^I))}.
    \qedhere
    \]
\end{proof}

\appendix

\section{A direct approach to the derivator of a Frobenius exact category}
\smallskip
\begin{center}by \textsc{Jan Šťovíček}\end{center}
\medskip

\subsection{Gorenstein exact categories}
\label{subsection:Gorenstein}

The main problem we face when trying to explicitly describe the stable derivator of a Frobenius exact category $\cat F$ is that diagram categories $\cat F^I$ (with the componentwise exact structure) are no longer Frobenius. However, as we will show in~\S\ref{subsection:diagrams_over_Frobenius}, they are Gorenstein in the following sense, which is very close to~\cite[Chapter~11, Definition~4.1]{sauter-habilitation}:

\begin{definition} \label{definition:Gorenstein-exact}
Let $\cat F$ be an exact category and $n\ge 0$ be a non-negative integer. We say that $\cat F$ is \emph{$n$-Gorenstein} if
\begin{enumerate}
\item it has enough projective objects and enough injective objects, and
\item all projective objects have injective dimension at most $n$ and all injective objects have projective dimension at most $n$.
\end{enumerate}

We say that $\cat F$ is Gorenstein if it is $n$-Gorenstein for some $n\ge 0$.
\end{definition}

Note that $\cat F$ is $0$-Gorenstein if and only if it is Frobenius. Our definition is inspired by the Iwanaga--Gorenstein rings~\cite{iwanaga-part1,iwanaga-part2} and, in particular, by the properties of their module categories. As we are going to show, key features of that setting are still valid in the generality of Definition~\ref{definition:Gorenstein-exact}. First of all, we have the following staightforward generalization of~\cite[Theorem~9]{iwanaga-part1}.

\begin{lemma} \label{lemma:finite_dimension_Gorenstein}
Let $\cat F$ be an $n$-Gorenstein exact category and $X\in\cat F$. The the following are equivalent:
\begin{enumerate}
\item $X$ has finite projective dimension,
\item $X$ has finite injective dimension,
\item the projective dimension of $X$ is at most $n$,
\item the injective dimension of $X$ is at most $n$.
\end{enumerate}
\end{lemma}
\begin{proof}
The implications $(3)\Rightarrow(1)$ and $(4)\Rightarrow(2)$ are trivial.

Let us prove that $(1)$ implies $(4)$. Suppose that $X\in\cat F$ has a finite projective resolution
\begin{equation} \label{equation:Gore_resolution}
0 \to P_m \to P_{m-1} \to \cdots \to P_1 \to P_0 \to X \to 0
\end{equation}
and denote by $Z_i$ for $i=0, \dots, m$ the objects that fit into conflations
\begin{equation} \label{equation:Gore_confl_in_resolution}
0\to Z_i\to P_{i-1}\to Z_{i-1}\to 0
\end{equation}
(where we put $Z_0=X$ by convention and necessarily $Z_m=P_m$). Then the injective dimension of each $P_i$ is at most $n$ by assumption, and it follows by induction on $i=m, m-1, \dots, 0$, using the conflations~\eqref{equation:Gore_confl_in_resolution}, that the injective dimension of each $Z_i$ is at most $n$. This implies $(4)$.

Finally, the proof of $(2)\Rightarrow(3)$ is dual.
\end{proof}

Now we define three classes of objects of interest in a Gorenstein exact category.

\begin{definition} \label{definition:Gorenstein_key_classes}
Let $\cat F$ be a Gorenstein exact category. We will define:
\begin{enumerate}
\item \emph{Weakly trivial objects} as those satisfying the equivalent conditions of Lemma~\ref{lemma:finite_dimension_Gorenstein} (the terminology is motivated by Example~\ref{example:Gorentstein} below). The class of weakly trivial objects will be denoted by $\WTriv(\cat{F})$.
\item \emph{Gorenstein projective} objects as those in the class
\[ \GProj(\cat F)=\{X\in\cat F\mid \Ext^j_{\cat F}(X,\Proj(\cat F))=0 \textrm{ for all } j>0 \}. \]
\item \emph{Gorenstein injective} objects as those in the class
\[ \GInj(\cat F)=\{X\in\cat F\mid \Ext^j_{\cat F}(\Inj(\cat F),X)=0 \textrm{ for all } j>0 \}. \]
\end{enumerate}
\end{definition}

\begin{remark} \label{remark:closure_prop_of_weakly_triv}
Lemma~\ref{lemma:finite_dimension_Gorenstein} implies that the class $\WTriv(\cat{F})$ of weakly trivial objects enjoys strong closure properties in $\cat F$. It is closed under kernels of deflations and extensions (because the class of objects of finite projective dimension always has these closure properties), and also under cokernels of inflations (because the class of objects of finite injective dimension satisfies this). It is also closed under retracts. In the terminology of~\cite{hovey_abelian_models}, such subcategories are called \emph{thick}.
\end{remark}

The three classes are linked by the notion of cotorsion pairs. This observation essentially goes back to~\cite{auslander-buchweitz}.

\begin{definition} \label{definition:cotorsion_pair}
Given an exact category $\cat F$, a pair of classes of objects $(\cat X,\cat Y)$ is called a \emph{cotorsion pair} if $\cat X = \{X\in\cat F\mid \Ext^1_{\cat F}(X,\cat Y)=0 \}$ and $\cat Y = \{Y\in\cat F\mid \Ext^1_{\cat F}(\cat X,Y)=0 \}$.
A cotorsion pair is called $(\cat X,\cat Y)$ is called
\begin{itemize}
\item \emph{hereditary} if also $\Ext^i_{\cat F}(\cat X,\cat Y)=0$ for all $i\ge 1$, and
\item \emph{complete} if for each $Z\in\cat F$, there exist conflations $0\to Y_Z\to X_Z\to Z\to 0$ and $0\to Z\to Y^Z\to X^Z\to 0$ with $X_Z,X^Z\in\cat X$ and $Y_Z,Y^Z\in\cat Y$.
\end{itemize}
\end{definition}

\begin{remark}
The conflations in the definition of a complete cotorsion pair are not required to be unique and are usually called \emph{approximation sequences} of $Z$. The reason for the terminology is the following. Given any morphism $f\colon X\to Z$ with $X\in\cat X$, the exact sequence
\[ \Hom_{\cat F}(X,X_Z)\to\Hom_{\cat F}(X,Z)\to\Ext^1_{\cat F}(X,Y_Z)=0  \]
shows that $f$ factors through the deflation $X_Z\to Z$. Thus, $X_Z\to Z$ is what is usually called a \emph{right $\cat X$-approximation} of $Z$. Dually, the morphism $Z\to Y^Z$ is a \emph{left $\cat Y$-approximation} of $Z$ (cf.~\cite{auslander-reiten-contravar-fin}).
\end{remark}

\begin{proposition}\label{proposition:Gore_cotorsion_pairs}
Let $\cat F$ be a Gorenstein exact category. Then
\[
(\GProj(\cat{F}),\WTriv(\cat{F}))
\qquad\text{and}\qquad
(\WTriv(\cat{F}),\GInj(\cat{F}))\]
are complete hereditary cotorsion pairs. Moreover, we have $\GProj(\cat{F})\cap\WTriv(\cat{F})=\operatorname{Proj}(\cat F)$ and $\WTriv(\cat{F})\cap\GInj(\cat{F})=\operatorname{Inj}(\cat F)$.
\end{proposition}

\begin{proof}
We will only prove the facts related to $(\GProj(\cat{F}),\WTriv(\cat{F}))$, the case of Gorenstein injectives is dual. The arguments are fairly standard and use the insight from~\cite{auslander-buchweitz}.

First of all, if $G\in\GProj(\cat{F})$, $X\in\WTriv(\cat{F})$ and we fix a projective resolution~\eqref{equation:Gore_resolution}, then the long exact sequences of $\Ext$-groups obtained by applying $\Hom_{\cat F}(G,-)$ to the conflations~\eqref{equation:Gore_confl_in_resolution} imply that (in that notation) $\Ext^j_{\cat F}(G,Z_i)=0$ for all $i=m, m-1, \dots, 0$ and $j>1$. That is, $\Ext^j_{\cat F}(\GProj(\cat F),\WTriv(\cat{F}))=0$ for all $j>0$. From there it follows that
\[ \GProj(\cat{F}) = \{X\in\cat F\mid \Ext^1_{\cat F}(X,\WTriv(\cat{F}))=0 \}. \]
Indeed, we have just proved the inclusion $\subseteq$ and, on the other hand, if $G$ is contained in the class on the right-hand side, then $\Ext^1_{\cat F}(G,X)=0$ for all cosyzygies of projective objects $X$ in $\cat F$ (as these belong to $\WTriv(\cat{F})$ by Remark~\ref{remark:closure_prop_of_weakly_triv}). Hence, by dimension shifting, $\Ext^j_{\cat F}(G,\operatorname{Proj}(\cat F))=0$ for all $j>0$, and so $G\in\GProj(\cat{F})$.

If $n\ge 0$ is such that $\cat F$ is $n$-Gorenstein, then any projective object has injective dimension at most $n$ by Lemma~\ref{lemma:finite_dimension_Gorenstein}.
A simple dimension-shifting argument then shows for each $Z\in\cat F$ that any $n$-th syzygy of $Z$ belongs to $\GProj(\cat{F})$. That is, there is an exact sequence
\[ 0 \to G \to P_{n-1} \to \cdots \to P_1 \to P_0 \to X \to 0 \]
in $\cat F$ such that $G, P_{n-1}, \dots, P_0\in\GProj(\cat{F})$ (of course, the objects $P_i$ are even projective).

Now, \cite[Theorem~1.1]{auslander-buchweitz} applies with $\mathbf{X}=\GProj(\cat{F})$ and $\omega=\operatorname{Proj(\cat F)}$ (in the notation of~\cite{auslander-buchweitz}), and provides us for each $Z\in\cat F$ with conflations $0\to W_Z\to G_Z\to Z\to 0$ and $0\to Z\to W^Z\to G^Z\to 0$ such that $G_Z,G^Z\in\GProj(\cat{F})$ and $W_Z,W^Z\in\WTriv(\cat{F})$. Strictly speaking, \cite[Theorem~1.1]{auslander-buchweitz} is stated under somewhat more restrictive assumptions, but its proof perfectly applies in our setting.

Finally note that given $Z\in\cat F$ is such that $\Ext^1_{\cat F}(\GProj(\cat F),Z)=0$, then the conflation $0\to Z\to W^Z\to G^Z\to 0$ from the last paragraph splits and consequently $Z\in\WTriv(\cat{F})$ by Remark~\ref{remark:closure_prop_of_weakly_triv}. All in all, we have shown that $(\GProj(\cat F),\WTriv(\cat{F}))$ is a complete hereditary cotorsion pair in~$\cat F$. Clearly also $\operatorname{Proj}(\cat F)\subseteq\GProj(\cat{F})\cap\WTriv(\cat{F})$, and the other inclusion follows from the fact that given $G\in\GProj(\cat{F})\cap\WTriv(\cat{F})$, any conflation $0\to W\to P\to G\to 0$ with $P\in\operatorname{Proj}(\cat F)$ splits as $W\in\WTriv(\cat F)$ by Remark~\ref{remark:closure_prop_of_weakly_triv}.
\end{proof}

Note that since the subcategories $\GProj(\cat{F})$ and $\GInj(\cat{F})$ are closed under extensions in $\cat F$, they themselves are naturally exact categories with the exact structure induced from that on $\cat F$.

\begin{corollary}\label{corollary:GP_and_GI_are_Frobenius}
Let $\cat F$ be a Gorenstein exact category.
Then the exact categories $\GProj(\cat{F})$ and $\GInj(\cat{F})$ are Frobenius. Moreover,
\[ \Proj(\GProj(\cat{F}))=\Inj(\GProj(\cat{F}))=\Proj(\cat F) \]
and
\[ \Proj(\GInj(\cat{F}))=\Inj(\GInj(\cat{F}))=\Inj(\cat F). \]
\end{corollary}
\begin{proof}
Since $\GProj(\cat{F})$ is a left half of a complete hereditary cotorsion pair in $\cat F$, it is closed under taking syzygies in $\cat F$ and, hence, $\GProj(\cat{F})$ has enough projectives and $\operatorname{Proj}(\GProj(\cat{F}))=\operatorname{Proj}(\cat F)$. The projective objects in $\GProj(\cat{F})$ are also injective, as seen directly from Definition~\ref{definition:Gorenstein_key_classes}. Finally, we know from the completeness of the cotorsion pair $(\GProj(\cat{F}),\WTriv(\cat{F}))$ that for any $G\in\GProj(\cat{F})$ there exists a conflation $0\to G\to P\to G'\to 0$ with $P\in\WTriv(\cat{F})$ and $G'\in\GProj(\cat{F})$ and, since $\GProj(\cat{F})$ is closed under extensions in $\cat F$, we in fact have $P\in\GProj(\cat{F})\cap\WTriv(\cat{F})=\Proj(\cat F)$. This shows that $\GProj(\cat{F})$ also has enough injectives and $\Inj(\GProj(\cat{F}))=\Proj(\cat F)$.

An analogous argument shows that $\GInj(\cat{F})$ is Frobenius with $\Proj(\GInj(\cat{F}))=\Inj(\GInj(\cat{F}))=\Inj(\cat F)$.
\end{proof}

The following is the main result of this section.

\begin{proposition}\label{proposition:GProj-vs-GInj}
The stable categories $\underline{\GProj(\cat{F})}$ and $\underline{\GInj(\cat{F})}$ are triangle equivalent. More specifically, the equivalence $F\colon\underline{\GProj(\cat{F})}\xrightarrow{\sim}\underline{\GInj(\cat{F})}$ is, up to a canonical natural isomorphism, given as follows:
\begin{enumerate}
\item For each $G\in\underline{\GProj(\cat{F})}$ we fix an approximation sequence $0\to G\xrightarrow{\eta_G} F(G)\to W\to 0$ with $F(G)\in\GInj(\cat{F})$ and $W\in\WTriv(\cat{F})$; this determines $F$ on objects.
\item Given a morphism $\bar{f}$ in $G\in\underline{\GProj(\cat{F})}$, represented by a morphism $f\colon G\to G'$ in $\GProj(\cat{F})$, the value $F(\bar f)$ is represented by any morphism $g\colon F(G)\to F(G')$ which fits into the following commutative square in $\cat F$:
\[
\begin{tikzcd}
G \arrow[d, swap, "f"] \ar[r, hookrightarrow, "\eta_G"] & F(G) \ar[d, "g"] \\
G' \arrow[r, hookrightarrow, "\eta_{G'}"] & F(G').
\end{tikzcd}
\]
\end{enumerate}
The inverse $\underline{\GInj(\cat{F})}\xrightarrow{\sim}\underline{\GProj(\cat{F})}$ is constructed dually.
\end{proposition}

Although a direct proof is not very hard either, we can explain the proposition conceptually using Hovey's insight from~\cite{hovey_abelian_models}. This has a lot to do with localization of categories and model structures. Recall that given any category $\cat F\in\kat{CAT}$ and a class of morphisms $W$ (whose elements are usually called \emph{weak equivalences}), there is always a universal functor $Q\colon\cat F\to\cat F[W^{-1}]$ making all morphisms in $W$ invertible, up to a potential set-theoretic nuisance that the Hom-sets in $\cat F[W^{-1}]$ are not small; see \cite[Lemma I.1.2]{gabriel-zisman}. In general, however, it is a challange to get any control over the category $\cat F[W^{-1}]$.

The situation is much more favorable if we have more structure on $\cat F$, so-called model category structure. This entails that we have two additional classes of morphisms $Cof$ and $Fib$ along with $W$, called \emph{cofibrations} and \emph{fibrations}, respectively, satisfying certain axioms. We will not state a complete definition here as it is not really necessary for our purpose, but it can be found, for instance, in~\cite[Definition~1.1.3]{hovey_modelcategories} or~\cite[Definition~7.1.3]{hirschhorn-modelcategories}. In this case, there is a full and dense functor $\cat F_{cf}\to\cat F[W^{-1}]$, where $\cat F_{cf}\subseteq\cat F$ is a full subcategory of $\cat F$ given by the so-called cofibrant and fibrant objects. We again will only refer to \cite[Theorem 1.2.10]{hovey_modelcategories} for details at this level of generality, but we will make all the above-mentioned classes and the required conditions on them precise in the coming proposition in the context of exact categories.
Here is a summary of results relevant for us, originating in~\cite{hovey_abelian_models} and generalized to the context of exact categories in~\cite{gillespie-exactmodels,stovicek-exactmodels} later.

\begin{proposition}\label{proposition:exact-model}
Let $\cat F$ be a weakly idempotent exact category and $\cat Cof, \cat W, \cat Fib$ be three full subcategories of $\cat F$ such that
\begin{enumerate}
\item $(\cat Cof,\cat W\cap\cat Fib)$ and $(\cat Cof\cap\cat W,\cat Fib)$ are complete cotorsion pairs in $\cat F$ and
\item $\cat W$ is closed under retracts and for any conflation $0\to X\to Y\to Z\to 0$ in $\cat F$ with two terms in $\cat W$, the third term also lies in $\cat W$.
\end{enumerate}
Then $\cat F$ together with the following classes of morphisms forms a model category:
\begin{itemize}
\item $Cof$, the class of cofibrations, consists of inflations with cokernels in $\cat Cof$,
\item $Fib$, the class of fibrations, consists of deflations with kernels in $\cat Fib$,
\item $W$, the class of weak equivalences, consists of the compositions $w=w_d\circ w_i$, where $w_d$ is a deflation with kernel in $\cat W$ and $w_i$ is an inflation with cokernel in $\cat W$.
\end{itemize}

Moreover, if the cotorsion pairs above are hereditary, then:
\begin{enumerate}
\item[(a)] The extension closed full subcategory $\cat F_{cf}:=\cat Cof\cap\cat Fib$ with the exact structure restricted from $\cat F$ is Frobenius exact, and the class of projective injective objects is precisely $\cat Cof\cap\cat W\cap\cat Fib$.
\item[(b)] The composition of functors $\cat F_{cf}\to\cat F\to\cat F[W^{-1}]$ induces equivalence
\[ \underline{\cat F_{cf}}\simeq\cat F[W^{-1}]. \]
\end{enumerate}
\end{proposition}

\begin{proof}
The existence of the model structure was proved in~\cite[Theorem 3.3]{gillespie-exactmodels} or~\cite[Theorem 6.9]{stovicek-exactmodels}, following the arguments of \cite[Theorem 2.2]{hovey_abelian_models}. The `moreover' part can be found in~\cite[Corollary 4.8 and Proposition~5.2(4)]{gillespie-exactmodels} or~\cite[Theorem 6.21]{stovicek-exactmodels}.
\end{proof}

\begin{remark}
A cautious reader might have noticed a certain mismatch in definitions: Both~\cite{hirschhorn-modelcategories,hovey_modelcategories} require that model categories be complete and cocomplete, which is certainly not true in general about a weakly idempotent exact category as in Proposition~\ref{proposition:exact-model}. However, one of the points checked in detail~\cite{gillespie-exactmodels,stovicek-exactmodels} is that even the limited repertoire of (co)limits available in $\cat F$ (i.e.\ finite (co)products and certain pullbacks and pushouts) suffices to prove the last conclusion of Proposition~\ref{proposition:exact-model}. We will abuse the terminology following~\cite{gillespie-exactmodels} here and still call $\cat F$ a model category.

For cognoscenti, there is another subtle aspect in which the above model structure differs from~\cite{hirschhorn-modelcategories,hovey_modelcategories}. Namely, there is no guarantee that the factorization of maps into a composition of a fibration with a trivial cofibration and a trivial fibration with a cofibration is functorial. However, this is also something that is not necessary anywhere in the proof of the last part of Proposition~\ref{proposition:exact-model}.
\end{remark}

\begin{example}\label{example:Gorentstein}
If $\cat F$ is a weakly idempotent complete Gorenstein exact category, then Propositions~\ref{proposition:Gore_cotorsion_pairs} and~\ref{proposition:exact-model} yield two model category category structures with the same class of weak equivalences $W$:
\begin{enumerate}
\item the so-called \emph{Gorenstein projective} model structure given by the triple $(\cat C, \cat W, \cat F) = (\GProj(\cat F), \WTriv(\cat F), \cat F)$ and
\item the so-called \emph{Gorenstein injective} model structure given by the triple of full subcategories $(\cat C, \cat W, \cat F) = (\cat F, \WTriv(\cat F), \GInj(\cat F))$ and\end{enumerate}
In particular, we have category equivalences (and even triangle equivalences by~\cite[Theorem~6.21]{stovicek-exactmodels})
\begin{equation}\label{equation:Gorenstein-equivalences}
\underline{\GProj(\cat F)} \overset{\sim}\to \cat F[W^{-1}] \overset{\sim}\leftarrow \underline{\GInj(\cat F)}.
\end{equation}
By Proposition~\ref{proposition:exact-model}, the weak equivalences are described as compositions of the form $w=w_d\circ w_i$, where $w_d$ is degreewise a retraction in $\cat F$ with projective kernel and $w_i$ is degreewise a section in $\cat F$ with projective cokernel. In particular, if $f\colon X\to Y$ is a weak equivalence, then $f_j\colon X_j\to Y_j$ induces an isomorphism in $\underline{\cat F}$ for each $j\in I$.

We claim that $W$ consists precisely of degreewise stable isomorphisms. Indeed, using the axioms for model categories applied to the model structure (1) above, each morphism $f\colon X\to Y$ in $\cat F^I$ factors as $f=g\circ w_i$, where $g$ is degreewise a deflation and $w_i$ is degreewise a section with projective cokernel. If $f$ is degreewise a stable isomorphism to start with, then so is $g$, and it is a standard fact that a deflation in a Frobenius exact category is a stable isomorphism if and only if it is a retraction with a projective kernel. Hence, all degreewise stable isomorphisms belong to $W$, and the converse inclusion was shown in the previous paragraph. This proves the claim.
\end{example}

\begin{proof}[Proof of Proposition~\ref{proposition:GProj-vs-GInj}]
This is now an immediate consequence of the equivalences from~\eqref{equation:Gorenstein-equivalences} and the fact that the map $\eta_G$ from the approximation sequence $0\to G\xrightarrow{\eta_G} F(G)\to W\to 0$ in the statement of the proposition is a weak equivalence, so becomes an isomorphism in $\cat F[W^{-1}]$.
\end{proof}

\begin{remark} \label{remark:Quillen-equivalence-and-more}
\begin{enumerate}
\item A usual way in which a homotopy theorist would explain the equivalences from~\eqref{equation:Gorenstein-equivalences} is that the identity functor on $\cat F$ is a Quillen equivalence between the Gorenstein projective and Gorenstein injective model structures.

\item We can say a little more about how the functor $Q\colon \cat F\to\cat F[W^{-1}]$ operates on Hom-groups for a Gorenstein exact category. Namely, if $X\in\GProj(\cat F)$ and $Y\in\cat F$, then there is an isomorphism
\[ Q\colon \underline{\Hom}_{\cat F}(X,Y) \overset{\sim}\to \Hom_{\cat{F}[W^{-1}]}(QX,QY), \]
where $\underline{\Hom}_{\cat F}(X,Y)$ is the factor of $\Hom_{\cat F}(X,Y)$ modulo the morphisms which factor through a projective object. Similarly, if $X\in\cat F$ and $Y\in\GInj(\cat F)$, we have an isomorphism
\[ Q\colon \overline{\Hom}_{\cat F}(X,Y) \overset{\sim}\to \Hom_{\cat{F}[W^{-1}]}(QX,QY), \]
where $\overline{\Hom}_{\cat F}(X,Y)$ is the factor of $\Hom_{\cat F}(X,Y)$ modulo the morphisms which factor through a injective object. These follow from \cite[Proposition~4.4(5)]{gillespie-exactmodels}, in view of the standard fact that given a cofibrant object $X$ and a fibrant object $Y$ in a model category, then $\Hom_{\cat F}(X,Y)\to\Hom_{\cat F[W^{-1}]}(QX,QY)$ is surjective and two morphisms $f,g\colon X\to Y$ are identified if and only if they are homotopic in the sense of model categories (see \cite[Theorem~1.2.10(ii)]{hovey_modelcategories} for details).
\end{enumerate}
\end{remark}

\subsection{Categories of diagrams in a Frobenius exact category}
\label{subsection:diagrams_over_Frobenius}

Now we get to our main source of Gorenstein exact categories: if $\cat F$ is a Frobenius exact category and $I$ a finite direct category (in the sense of Definition~\ref{definition:finite_directed}), then $\cat F^I$ with the degree-wise exact structure (where a sequence $0\to X\to Y\to Z\to 0$ is exact provided that $0\to X_i\to Y_i\to Z_i\to 0$ is exact in $\cat F$ for all $i\in I$) turns out to be Gorenstein (Proposition~\ref{proposition:Gorenstein-diagrams}).

The key step is contained in the following lemma, where we recall from Section~\ref{section:derivator_Frobenius} that $\cat F^I$ has enough projectives and injectives.

\begin{lemma}\label{lemma:finite_homological_dim_F^I_general}
Let $\cat F$ be an exact category with enough projectives $\cat P:=\Proj(\cat F)$, let $I$ be a finite direct category with a degree function $d\colon \operatorname{Ob}(I) \to \{0,\cdots n\}$, and let $X\in{\cat P}^I$. Then $X$ has projective dimension at most $n$ in ${\cat F}^I$.
\end{lemma}

\begin{proof}
We will proceed by induction on $n$. As in the proof of Lemma~\ref{lemma:splitmono}, we consider for each $j\in I$ the counit of adjunction $\epsilon^j\colon j_!(X_j)\to X$ and put them together to form a map
\[ p := (\epsilon_j)_{j\in I}\colon \bigoplus_{j\in I} j_!(X_j)\to X \]
from $P:=\bigoplus_{j\in I} j_!(X_j)$ to $X$.
Since all $X_j$ are assumed to be projective in $\cat F$, the domain of $p$ is projective in $\cat F^I$ and, as $j\colon e\to I$ are fully faithful, the maps $j^*\epsilon^j\colon (j_!(X_j))_j\to X_j$ are isomorphisms. It follows that $p$ is termwise a split epimorphism, so a deflation in $\cat F^I$. Let us denote $K:=\ker p$. Note also that for any $i,j\in I$ with $d(i)=0<d(j)$ we have $(j_!(X_j))_i=0$, and so $p_i\colon P_i\to X_i$ is an isomorphism. Consequently, $K_i=0$ for each $i\in I$ with $d(i)=0$.

Let now $J\subseteq I$ be a full subcategory given by the objects of degree at least one and denote by $k\colon J\to I$ the full embedding. Then $K\cong k_!k^*K$ and note that $k_!\colon {\cat F}^J\to{\cat F}^I$ is exact (it is the left extension by zero functor) and preserves projectives (being a left adjoint to the exact restriction functor). If $J=\emptyset$, then simply $K=0$ and $X\in\cat{F}^I$ is projective. If $J$ is not empty, it is a finite direct category with a degree function $d'\colon\operatorname{Ob}(J) \to \{0,\dots, n-1\}$ given by $d'(j)=d(j)-1$, so the projective dimension of $k^*K$ in $\cat{F}^J$ is at most $n-1$  by the inductive hypothesis, and so is the projective dimension of $K$ in $\cat{F}^I$ by the above discussion. Hence, the projective dimension of $X$ in $\cat{F}^I$ is at most $n$.
\end{proof}

As a consequence, we can characterize objects of finite homological dimension in $\cat F^I$.

\begin{lemma}\label{lemma:finite_homological_dim_F^I}
Let $\cat F$ be a Frobenius exact category and $I$ a finite direct category with a degree function $d\colon \operatorname{Ob}(I) \to \{0,\cdots n\}$. Then the following are equivalent for $X\in\cat F^I$:
\begin{enumerate}
\item $X$ has finite projective dimension in $\cat F^I$,
\item $X$ has finite injective dimension in $\cat F^I$,
\item $X_i$ is projective in $\cat F$ for each $i\in I$.
\end{enumerate}
If $X$ satisfies the equivalent conditions, then both the projective and the injective dimension of $X$ is at most $n$.
\end{lemma}

\begin{proof}
We will only prove $(1)\Leftrightarrow(3)$, as $(2)\Leftrightarrow(3)$ is dual. Note also that $(3)\Rightarrow(1)$ is an immediate consequence of Lemma~\ref{lemma:finite_homological_dim_F^I_general} and the maximum projective dimension is bounded by $n$, so it remains to prove $(1)\Rightarrow(3)$. To that end, let $X\in\cat F^I$ of finite projective dimension and consider a projective resolution
\[
0 \to P_m \to P_{m-1} \to \cdots \to P_1 \to P_0 \to X \to 0
\]
Let also $i\in I$. Note that the restriction functor $i^*\colon\cat F^I\to\cat F$ has an exact right adjoint by Corollary~\ref{corollary:Kan-extensions-existence} and Remark~\ref{remark:Kan-extensions-existence}, so $i^*$ sends projectives in $\cat F^I$ to projectives in $\cat F$. If we apply $i^*$ to the resolution above, we see that $X_i=i^*X\in\cat F$ has a finite projective dimension in $\cat F$. However, since $\cat F$ is Frobenius, this simply means that $X_i$ is projective in $\cat F$.
\end{proof}

Now we can prove the promised fact about diagram categories of Frobenius exact categories.

\begin{proposition}\label{proposition:Gorenstein-diagrams}
Let $\cat F$ be a Frobenius exact category and $I$ be a finite direct category with a degree function $d\colon \operatorname{Ob}(I) \to \{0,\cdots n\}$.
Then $\cat F^I$ is an $n$-Gorenstein exact category.
\end{proposition}

\begin{proof}
The fact that every projective object has injective dimension at most $n$ and vice versa follows from the two previous lemmas.
\end{proof}

\subsection{Presentations of stalks}
\label{subsection:stalks}

Let again $I$ be a finite direct category, and $\cat F$ any additive category.
We call a representation $X\in\cat F^I$ a \emph{stalk representation} if $X_i=0$ for all but at most one $i\in I$. Given $P\in\cat F$ and $i\in I$, we denote by $\stalk{i}(P)\in\cat F^I$ the stalk representation with $\stalk{i}(P)_i=P$ and $\stalk{i}(P)_j=0$ for all $j\ne i$. Our aim in the section will be to relate syzygies of stalk representations $\stalk{i}(P)$ in the degreewise split exact structure on $\cat F^I$ with concepts related to Reedy categories~\cite[Chapter~15]{hirschhorn-modelcategories}.

To that end, we will denote by $\partial(I/i)$ the full subcategory obtained from the slice category $I/i$ by removing the terminal object $(i, \id\colon i\to i)$; compare to~\cite[Definition~15.2.3]{hirschhorn-modelcategories}. We will further denote by $\pi_i\colon \partial(I/i)\to I$ the canonical projection functor sending $(k, f\colon k\to i)$ to $k$. Note that given any category $\cat F$ and an $I$-shaped diagram $X\in\cat F^I$, the restricted diagram $\pi_i^*(X)\in\cat F^{\partial(I/i)}$ has a canonical cocone with apex $X_i$ in $\cat F$. This is because for any $a=(k, f\colon k\to i)\in\partial(I/i)$, we have $\pi_i^*(X)_a=X_k$ by definition and the cocone is formed by the maps $X_f\colon \pi_i^*(X)_a=X_k\to X_i$. If the colimit of $\pi_i^*(X)$ exists in $\cat F$, we denote $L_i(X) := \colim_{\partial(I/i)}\pi_i^*(X)$ and call it the \emph{latching object} of $X$ at $i$, and the canonical morphism in $\cat F$,
\[ \lambda_i(X)\colon L_i(X) \to X_i, \]
is called \emph{latching map} (see~\cite[Definition 15.2.5]{hirschhorn-modelcategories}).
Dually, we denote by $\partial(i/I)$ the full subcategory obtained from the slice category $i/I$ by removing the initial object $(i, \id\colon i\to i)$. If $Y\in\cat F^I$ and the limit $\lim_{\partial(i/I)}\rho_i^*(Y)$ of an analogously defined diagram given by the projection functor $\rho_i\colon\partial(i/I)\to I$ exists, we denote it by $M_i(Y)$ and call it the \emph{matching object}. In this case, there is also a canonical map $\mu_i(Y)\colon Y_i\to M_i(Y)$ which we call the \emph{matching map}.

Now we first compute the latching map in a special case where we substitute for $\cat F$ the category of sets (which is not additive, but the definitions still makes sense for it).

\begin{lemma}\label{lemma:represented_presheaves_are_Reedy}
Let $I$ be a finite direct category, $i,j\in I$ be two objects, and consider the diagram $X:=I(i,-)\in\Set^I$. Then $\lambda_j(X)\colon L_j(X) \to X_j=I(i,j)$ is injective and
\[
\Img\lambda_j(X)=\begin{cases}
I(i,j), & \textrm{if $i\neq j$,} \\
\emptyset, & \textrm{if $i=j$.}
\end{cases}
\]
\end{lemma}
\begin{proof}
The category of sets is cocomplete, so latching objects exist for all $X\in\Set^I$ and $j\in I$. Moreover, there is a well known construction of the corresponding colimit as the disjoint union of values of $\pi_j^*(X)$ at objects of $\partial(I/j)$ modulo an equivalence relation. More precisely,
$L_j(X) = \colim_{\partial(I/j)}\pi_j^*(X) = \big(\coprod_{(k,f\colon k\to j)}X_k\big)\,/\!\sim$,
where the equivalence relation is the smallest one such that $x\sim h(x)$ whenever $(k,f\colon k\to j)\to(l,g\colon l\to j)$ is a morphism in $\partial(I/j)$ represented by a morphism $h\colon k\to l$ in $I$, and whenever $x\in X_k$ (and so $h(x)\in X_l$).

Specializing to $X=I(i,-)$, we get the formula
$L_j(X)=\big(\coprod_{(k,f\colon k\to j)}I(i,k)\big)\,/\!\sim$,
where the equivalence relation is the smallest one such that $x\sim y$ whenever $x\colon i\to k$ is in a copy of $I(i,k)$ indexed by $(k,f\colon k\to j)$ and $y\colon i\to l$ is in a copy of $I(i,l)$ indexed by $(l,g\colon l\to j)$, and there is a commutative diagram in $I$ of the form
\[
\begin{tikzcd}[row sep=small]
&& k \arrow[drr, "f"] \ar[dd] \\
i \arrow[urr, "x"] \arrow[drr, swap, "y"] &&&& j \\
&& l \arrow[urr, swap, "g"] \\
\end{tikzcd}
\]
Moreover, the latching map $\lambda_j(X)\colon L_j(X)\to X_j=I(i,j)$ then maps an equivalence class $[x]_\sim$ given by $x\colon i\to k$ from a copy of $I(i,k)$ indexed by $(k, f\colon k\to j)\in\partial(I/i)$ simply to the composition $f\circ x\in I(i,j)$.

In order to prove injectivity of $\lambda_j(X)$ for $X=I(i,-)$, suppose that $\lambda_j(X)([x]_\sim)=\lambda_j(X)([y]_\sim)\in I(i,j)$, where $(x\colon i\to k)\in X_a=I(i,k)$ for $a=(k, f\colon k\to j)$ and $(y\colon i\to l)\in X_b=I(i,l)$ for $b=(l, g\colon l\to j)$. In other words, $f\circ x=g\circ y$. If we denote this composition by $h\colon i\to j$ and put $c=(i,h\colon i\to j)$, then $x$ is the image of $1_i\in X_c$ under the morphism $X_c\to X_a$ in $\partial(I/i)$ given by $x\colon i\to k$, and similarly $y$ is the image of $1_i$ under the morphism $X_c\to X_b$ given by $y\colon i\to k$. It follows that $[x]_\sim=[1_i]_\sim=[y]_\sim$ in $L_j(X)$, so $\lambda_j(X)$ is injective, as required.

The same argument shows that the image of $\lambda_j(X)$ consists of compositions $f\circ x$, where $f\colon k\to j$ is a non-identity morphism in $I$ and $x\colon i\to k$ is any morphism in $I$. Such compositions (even when restricting only to $x=1_i$) give precisely all non-identity morphisms $i\to j$.
\end{proof}

In order to make use of the latter lemma for an additive category $\cat F$, observe that we have a functor $-\times P\colon\Set_{\mathsf{fin}}\to\cat F$, where $\Set_{\mathsf{fin}}\subseteq\Set$ is the full subcategory of finite sets, which sends a finite set $S$ to an $S$-fold coproduct of $P$ in $\cat F$ and satisfies the adjunction formula
\[ \Hom_{\cat F}(S\times P,Y)\cong\Hom_{\Set}(S,\Hom_{\cat F}(P,Y)). \]
In particular, $-\times P\colon\Set_{\mathsf{fin}}\to\cat F$ preserves finite colimits. This also proves the existence of such finite colimits in $\cat F$.
Analogously, the functor $P^{(-)}\colon\Set_{\mathsf{fin}}^\mathrm{op}\to\cat F$ which sends $S$ to the $S$-fold product of $P$, satisfies the adjunction formula
\[ \Hom_{\cat F}(Y, P^S)\cong\Hom_\Set(S,\Hom_{\cat F}(Y,P)). \]
In particular, $P^{(-)}\colon\Set_{\mathsf{fin}}^\mathrm{op}\to\cat F$ sends finite colimits of finite sets to limits in $\cat F$ (and such limit then exist in $\cat F$).

\begin{lemma}\label{lemma:represented_functors_are_Reedy}
Let $\cat F$ be an additive category, $I$ be a finite direct category and $i,j\in I$.
\begin{enumerate}
\item Let $P\in\cat F$ and $X=i_!(P)\in\cat F^I$ (where we use the notation from Proposition~\ref{proposition:Kan-extensions-existence}). Then the latching object $L_j(X)$ exists in $\cat F$. Moreover, the latching morphism $\lambda_j(X)\colon L_j(X)\to X_j$ is an isomorphism for $i\ne j$, and a section with cokernel $P$ when $i=j$.
\item Let $Q\in\cat F$ and $Y=i_*(Q)\in\cat F^I$. Then the matching object $M_j(Y)$ exists in $\cat F$. Moreover, the matching morphism $\mu_j(Y)\colon Y_j\to M_j(Y)$ is an isomorphism for $i\ne j$, and a retraction with kernel $Q$ when $i=j$.
\end{enumerate}
\end{lemma}

\begin{proof}
In order to prove part (1), consider the following diagrams in $\Cat$:
\[
\begin{tikzcd}
\cat F \arrow[r, <-, "i^*"] \arrow[d, swap, "{\Hom_{\cat F}(P,-)}"] & \cat F^I \arrow[d, "{\Hom_{\cat F}(P,-)}"] &&
\cat F \arrow[r, "i_!"] \arrow[d, swap, <-, "-\times P"] & \cat F^I \arrow[d, <-, "-\times P"]
\\
\Set \arrow[r, <-, "i^*"] & \Set^I, &&
\Set_{\mathsf{fin}} \arrow[r, "i_!"] & \Set_{\mathsf{fin}}^I.
\end{tikzcd}
\]
Here, the rightmost arrows are obtained by componentwise applying the functors in the left arrow to $I$-shaped diagrams and, obviously, the left-hand side square commutes. Since the right-hand side square consists of left adjoint functors (defined only on $\Set_{\mathsf{fin}}$), it commutes up to a natural equivalence as well. As a consequence, if $[0]\in\Set$ stands for the singleton set, we have
\[ I(i,-)\times P = i_!([0])\times P\cong i_!([0]\times P) \cong i_!(P) = X \]
in $\cat F^I$. That is, $X$ is obtained up to natural equivalence as the composition
\[ I\xrightarrow{I(i,-)}\Set_{\mathsf{fin}}\xrightarrow{-\times P}\cat F. \]
Then, using that left adjoints preserve all existing colimits, we get that
\begin{align*}
L_j(X) &\cong L_j(I(i,-)\times P) \\
&= \colim_{\partial(I/i)}(I(i,\pi_i(-))\times P) \\
&\cong (\colim_{\partial(I/i)} I(i,\pi_i(-)))\times P \\
&= L_j(I(i,-))\times P.
\end{align*}
In fact, the argument even shows that $\lambda_j(X)\cong \lambda_j(I(i,-))\times P$. Then the conclusion immediately follows from Lemma~\ref{lemma:represented_presheaves_are_Reedy}.

The proof of part (2) is similar. We just notice that $\mu_j(X)\cong P^{\lambda_j(I(i,-))}$.
\end{proof}

Now we can prove the desired statement on (co)syzygies of stalk representations. Let $P\in\cat F$ and $i\in I$. 

\begin{proposition}\label{proposition:presentation_of_stalks}
Let $\cat F$ be an additive category, $I$ a finite direct category and $j\in I$.
\begin{enumerate}
\item Let $P\in\cat F$ and consider the compositions of functors
\[
\opp{I}\times I\xrightarrow{I(-,-)} \Set_{\mathsf{fin}} \xrightarrow{-\times P} \cat F.
\]
This gives us a functor $X\colon \opp{I}\to\cat F^I$ such that $X_i=I(i,-)\times P \;(\cong i_!(P))$ and there is a degreewise split short exact sequence in $\cat F^I$ of the form
\[ 0\longrightarrow L_j(X)\overset{\lambda_j(X)}\longrightarrow j_!(P) \longrightarrow \stalk{j}(P)\longrightarrow 0. \]

\item Let $Q\in\cat F$ and consider the compositions of functors
\[
I\times \opp{I}\xrightarrow{\opp{I(-,-)}} \opp{\Set}_{\mathsf{fin}} \xrightarrow{Q^{(-)}} \cat F.
\]
This gives us a functor $Y\colon \opp{I}\to\cat F^I$ such that $Y_i=Q^{I(-,i)} \;(\cong i_*(Q))$ and there is a degreewise split short exact sequence in $\cat F^I$ of the form
\[ 0\longrightarrow \stalk{j}(Q)\longrightarrow j_*(Q) \overset{\mu_j(Y)}\longrightarrow M_j(Y)\longrightarrow 0. \]
\end{enumerate}
\end{proposition}

\begin{proof}
Let us prove part (1), part (2) is similar.
The functor $X$ is just an adjoint form of the composition in (1), using the Cartesian closed monoidal structure on $\Cat$.
By the previous lemma, the latching morphism for $X_i\cong i_!(P)\in\cat{F}^I$ exists for each $i\in\opp{I}$. It follows that the latching morphism also exists for $X\in(\cat F^I)^{\opp{I}}$ and is computed component-wise, that is, $(\lambda_j(X))_i=\lambda_j(X_i)$. Finally, Lemma~\ref{lemma:represented_functors_are_Reedy} also tells us that $\lambda_j(X)$ is a degreewise split monomorphism and what the cokernel is.
\end{proof}

\subsection{Kan extensions for Frobenius exact categories}
\label{subsection:Kan-extensions}

Our next aim is to prove that given a weakly idempotent complete Frobenius exact category $\cat{F}$, then for any functor $u\colon I\to J$ between finite direct categories, there is a partially defined left Kan extension functor $u_!\colon \GProj(\cat{F}^I)\to\GProj(\cat{F}^J)$ and dually a partially defined right Kan extension functor $u_*\colon \GInj(\cat{F}^I)\to\GInj(\cat{F}^J)$.

We first prove that Gorenstein projectives and Gorenstein injectives restrict along certain functors $u\colon I\to J$. Here is the key observation.

\begin{lemma}\label{lemma:restriction-and-Ext}
Let $\cat{F}$ be an exact category, $\cat{W}\subseteq\cat{F}$ be a subcategory closed under finite direct sums, and $u\colon I\to J$ be a functor in $\kat{Cat}$.
\begin{enumerate}
\item If, for each $j\in J$, the slice $u/j$ is a disjoint union of finitely many categories with terminal objects, and if $Y\in\cat{F}^J$ is such that $\Ext^1_{\cat{F}^J}(\cat{W}^J,Y)=0$, then $\Ext^1_{\cat{F}^I}(\cat{W}^I,u^*Y)=0$.
\item If, for each $j\in J$, the slice $j/u$ is a disjoint union of finitely many categories with initial objects, and if $X\in\cat{F}^J$ is such that $\Ext^1_{\cat{F}^J}(X,\cat{W}^J)=0$, then $\Ext^1_{\cat{F}^I}(u^*X,\cat{W}^I)=0$.
\end{enumerate}
\end{lemma}

\begin{proof}
We will prove only the first statement; the other is dual. By Corollary~\ref{corollary:Kan-extensions-existence}, the exact restriction functor $u^*\colon \cat{F}^J\to\cat{F}^I$ has a left adjoint $u_!\colon \cat{F}^I\to\cat{F}^J$, which is also exact by Remark~\ref{remark:Kan-extensions-existence}. The exact adjunction induces a natural isomorphism $\Ext^1_{\cat{F}^J}(u_!X,Y)\cong\Ext^1_{\cat{F}^I}(X,u^*Y)$ for each $X\in\cat{F}^I$ and $Y\in\cat{F}^J$. The conclusion follows as Remark~\ref{remark:Kan-extensions-existence} also implies that $u_!(\cat{W}^I)\subseteq\cat{W}^J$.
\end{proof}

The main consequence of the previous lemma is as follows.

\begin{lemma}\label{lemma:latching-and-restriction}
Let $\cat F$ be a Frobenius exact category and $I$ be a finite direct set.
\begin{enumerate}
\item Let $X\in\GProj(\cat F^I)$. Then for any $i\in I$ and the corresponding projection functor $\pi_i\colon \partial(I/i)\to I$ from \S\ref{subsection:stalks}, we have $\pi_i^*(X)\in\GProj(\partial(I/i))$. Moreover, if $u\colon J\to I$ is a sieve (recall Proposition~\ref{proposition_(co)sieves}), then $u^*(X)\in\GProj(\cat F^I)$.
\item Let $Y\in\GInj(\cat F^I)$. Then for any $i\in I$ and the projection functor $\rho_i\colon \partial(i/I)\to I$ from \S\ref{subsection:stalks}, we have $\rho_i^*(Y)\in\GInj(\partial(i/I))$. Moreover, if $u\colon J\to I$ is a cosieve , then $u^*(Y)\in\GInj(\cat F^I)$.
\end{enumerate}
\end{lemma}

\begin{proof}
In part (1), the category $j/\pi_i$, where $j\in I$, is isomorphic to the full subcategory of the double slice category $j/I/i$, whose objects are triples $(k\in I, f\colon j\to k, g\colon k\to i)$ with $g\ne 1_i$. So $j/\pi_i$ is a disjoint union of categories $\cat C_h$, $h\in I(j,i)$, where $\cat C_h\subseteq j/I/i$ consists of all objects $(k, f\colon j\to k, g\colon k\to i)$ with $g\circ f=h$ and $g\ne 1_i$. Moreover, each $\cat C_h$, if non-empty, has the initial object $(j, 1_j\colon j\to j, h\colon j\to i)$. It follows from Lemmas~\ref{lemma:finite_homological_dim_F^I} and~\ref{lemma:restriction-and-Ext}(2) that, for any $X\in\GProj(\cat F^I)$, we have $\Ext^1_{\cat F^J}(\pi_i^*(X),\Proj(\cat F)^J)=0$, and so $\pi_i^*(X)\in\GProj(\cat F^J)$.

If $u\colon J\to I$ is a sieve and $j\in I$, then $j/u$, if non-empty, has the initial object $(u^{-1}(j), 1_j)$. So Lemma~\ref{lemma:restriction-and-Ext}(2) again applies to prove that $u^*(X)\in\GProj(\cat F^J)$.

The proof of part (2) is dual analogous.
\end{proof}

We also record the following easy characterization of Gorenstein projectives and injectives using stalk representations.

\begin{lemma} \label{lemma:Gorenstein-via-stalks}
Let $\cat F$ be a Frobenius exact category and $I$ a finite direct category.
\begin{enumerate}
\item A representation $X\in\cat F^I$ is Gorenstein projective if and only if $\Ext^1_{\cat F^I}(X,\stalk{j}(Q))=0$ for all $j\in I$ and $Q\in\Proj(\cat F)$.
\item A representation $Y\in\cat F^I$ is Gorenstein injective if and only if $\Ext^1_{\cat F^I}(\stalk{j}(P),Y)=0$ for all $j\in I$ and $P\in\Proj(\cat F)$.
\end{enumerate}
\end{lemma}

\begin{proof}
The `only if' part in both (1) and (2) follows from Proposition~\ref{proposition:Gore_cotorsion_pairs} and Lemma~\ref{lemma:finite_homological_dim_F^I}.

For the `if' part, let $d\colon \operatorname{Ob}(I) \to \{0,\cdots n\}$ be the degree function as in Definition~\ref{definition:finite_directed} and consider $W\in\WTriv(\cat F^I)$. For each $0\leq m\leq n+1$, let $W(m)\subseteq W$ be a subrepresentation such that $W(m)_i=0$ for all $d(i)<m$ and $W(m)_i=W_i$ for $d(i)\ge m$. Then we have a degreewise split finite filtration
\[ 0=W(n+1)\subseteq W(n) \subseteq W(n-1) \subseteq\cdots\subseteq W(1)\subseteq W(0)=W \]
and $W(m)/W(m+1)\cong \bigoplus_{i\in I,\, d(i)=m}\stalk{i}(W_i)$ in $\cat F^I$. Hence, if $\Ext^1_{\cat F^I}(X,\stalk{j}(P))=0$ for all $j\in I$ and $P\in\Proj(\cat F)$, then $\Ext^1_{\cat F^I}(X,\WTriv(\cat F^I))=0$, so $X\in\GProj(\cat F)$ by Proposition~\ref{proposition:Gore_cotorsion_pairs} and Lemma~\ref{lemma:finite_homological_dim_F^I}. The argument for Gorenstein injectives is analogous.
\end{proof}

Now we give an inductive criterion which will in the end allow us to prove existence of colimits of Gorenstein projective diagrams. To this end, dually to the concept of minimal object of $I$ introduced in \S\ref{section:derivator_Frobenius}, we call an object $j \in I$ a \emph{maximal object} if $I(j,i)=\emptyset$ for any $i \ne j$. Any non-empty finite direct category clearly has a maximal object.

\begin{lemma}\label{lemma:colimit-existence-inductive}
Let $\cat F$ be a Frobenius exact category and let $I$ be a non-empty finite direct category.
Furthermore, let $j\in I$ be a maximal object, let $J = I\setminus \{j\}$, and denote by $u\colon J\to I$ the inclusion functor.
If $G\in\cat F^I$ is an $I$-shaped diagram in $\cat F$ such that
\begin{enumerate}
\item both the colimit of $u^*(G)\in\cat F^J$ and the latching object $L_j(G)$ exist in $\cat F$, and
\item the latching morphism $\lambda_j(G)\colon L_j(G)\to G_j$ is an inflation,
\end{enumerate}
then $\colim_IG$ exists in $\cat F$.
\end{lemma}

\begin{proof}
Let us denote $G_J:=\colim_Ju^*(G)\in\cat F$ for brevity and recall from the beginning of~\S\ref{subsection:stalks} that $L_j(G)=\colim_{\partial(I/j)}\pi_j^*(G)$.
Since $I$ is finite direct and $f\ne 1_i$ for all $(k,f\colon k\to j)\in\partial(I/j)$, the image of the projection $\pi_j\colon \partial(I/j)\to I$ is contained in $J$.
In particular, the diagram $\pi_j^*(G)$ is obtained as the composition of functors
$\partial(I/j) \overset{\pi_j}\longrightarrow J \overset{u}\longrightarrow I \overset{G}\longrightarrow \cat F$
and there is a canonical comparison map
\begin{equation}\label{equation:colimit-comparison-map}
r\colon L_j(G)=\colim_{\partial(I/j)}\pi_j^*u^*(G)\to\colim_Ju^*(G)=G_J.
\end{equation}
To construct $r$ explicitly, let $e$ be the terminal category, and let $\operatorname{pt}_J\colon J\to e$ and $\operatorname{pt}_{\partial(I/j)}\colon \partial(I/j)\to e$ be the unique functors to $e$, so that $\operatorname{pt}_J^*\colon\cat F\to\cat F^J$ and $\operatorname{pt}_{\partial(I/j)}^*\colon\cat F\to\cat F^{\partial(I/j)}$ are the constant diagram functors. Then the colimiting cocone of $G_J$, in the form of a morphism $u^*(G)\to\operatorname{pt}_J^*(G_J)$ in $\cat F^J$, restricts along $\pi_j\colon\partial(I/i)\to J$ to a cocone $\pi_j^*u^*(G)\to\operatorname{pt}_{\partial(I/j)}^*(G_J)$. Then $r$ is obtained from the latter cocone, using the universal property of the colimit defining the latching object.

Now we use the assumption that $\lambda_j(G)$ is an inflation, so that the following pushout square exists in~$\cat F$,
\begin{equation}\label{equation:colimit-GProj}
\begin{tikzcd}
L_j(G) \arrow[d, swap, "\lambda_j(G)"] \arrow[r, "r"] & \colim_J u^*(G) \arrow[d]  \\
G_j \arrow[r] & C.
\end{tikzcd}
\end{equation}
We claim that $C$ is a colimit object of the diagram $G\colon I\to\cat F$. To that end, let $Y\in\cat F$ be any object. If we apply $\Hom_{\cat F}(-,Y)$ to~\eqref{equation:colimit-GProj}, we obtain a pullback square of abelian groups,
\[
\begin{tikzcd}
\Hom_{\cat F}(L_j(G),Y) \arrow[d, <-, swap, "{\Hom_{\cat F}(\lambda_j(G),Y)}"] \arrow[r, <-, "{\Hom_{\cat F}(r,Y)}"] & \Hom_{\cat F}(\colim_J u^*(G), Y) \arrow[d, <-]
\\
\Hom_{\cat F}(G_j, Y) \arrow[r, <-] & \Hom_{\cat F}(C, Y).
\end{tikzcd}
\]
Using that homomorphism groups from a colimit classify cocones of the corresponding diagrams and writing the cocones in the form of morphisms to constant diagrams as above, we can replace some terms of the latter square by their adjoint forms,
\[
\begin{tikzcd}
\Hom_{\cat F^{\partial(I/j)}}(\pi_j^*u^*(G),\operatorname{pt}^*_{\partial(I/j)}(Y)) \arrow[d, <-] \arrow[r, <-, "\pi_j^*"] & \Hom_{\cat F^J}(u^*(G), \operatorname{pt}^*_J(Y)) \arrow[d, <-]
\\
\Hom_{\cat F}(G_j, Y) \arrow[r, <-] & \Hom_{\cat F}(C, Y).
\end{tikzcd}
\]
By the construction of the morphism $r$ in \eqref{equation:colimit-comparison-map}, the upper horizontal map amounts simply to applying the restriction functor $\pi_j^*\colon\cat F^J\to\cat F^{\partial(I/j)}$, as indicated. On the other hand, the left vertical map assigns to  a morphism $G_j\to Y$ in $\cat F$ the composition $\pi^*_j(G)\to\operatorname{pt}^*_{\partial(I/j)}(G_j)\to \operatorname{pt}^*_{\partial(I/j)}(Y)$, where the first map is the canonical cocone given by $G$, as explained at the beginning of~\S\ref{subsection:stalks}.

It is straightforward from the explicit computations that there is also another pullback square of abelian groups of the form
\[
\begin{tikzcd}
\Hom_{\cat F^{\partial(I/j)}}(\pi_j^*(G),\operatorname{pt}^*_{\partial(I/j)}(Y)) \arrow[d, <-] \arrow[r, <-, "\pi_j^*"] & \Hom_{\cat F^J}(u^*(G), \operatorname{pt}^*_J(Y)) \arrow[d, <-, "u^*"]
\\
\Hom_{\cat F}(G_j, Y) \arrow[r, <-, "j^*"] & \Hom_{\cat F^I}(G, \operatorname{pt}_I^*(Y)),
\end{tikzcd}
\]
where $\operatorname{pt_I}\colon I\to e$ is the unique functor to the terminal category, the vertical map on the right-hand side is induced by $u^*$ and the lower horizontal arrow is the restriction to the $j$-th component. Hence, from the uniqueness of pullbacks, we obtain a natural isomorphism
\[ \Hom_{\cat F}(C,-)\cong \Hom_{\cat F^I}(G,\operatorname{pt}_I^*(-)), \]
which shows that the colimit $\colim_I G$ exists in $\cat F$ and is isomorphic to $C$.
\end{proof}

In the setting of the previous lemma, we can also prove the exactness of $I$-shaped colimits, if assuming exactness of latching objects and exactness of colimits of restricted diagrams. We will again use this as an inductive criterion for exactness later in Theorem~\ref{theorem:Gorenstein_projectives}.

\begin{lemma}\label{lemma:colimit-exactness-inductive}
Let $\cat F$, $j\in I$ and $u\colon J\to I$ be as in Lemma~\ref{lemma:colimit-existence-inductive} and suppose that $\varepsilon\colon 0\to G\to H\to K\to 0$ is a conflation in $\cat F^I$ in the degreewise exact structure such that 
\begin{enumerate}
\item for each term $X$ of $\varepsilon$, the colimit of $u^*(X)\in\cat F^J$ and the latching object $L_j(X)$ exist in $\cat F$,
\item for each term $X$ of $\varepsilon$, the latching morphism $\lambda_j(X)\colon L_j(X)\to X_j$ is an inflation, and
\item the sequences $\colim_Ju^*(\varepsilon)$ and $L_j(\varepsilon)$ are conflations in $\cat F$.
\end{enumerate}
Then the colimit sequence
\[ \colim_I(\varepsilon)\colon \; 0 \to \colim_I G \to \colim_I H \to \colim_I K \to 0 \]
 (which exists in $\cat F$ by Lemma~\ref{lemma:colimit-existence-inductive}) is a conflation.
\end{lemma}

\begin{proof}
Using the pushout diagram~\eqref{equation:colimit-GProj} from the proof of Lemma~\ref{lemma:colimit-existence-inductive}, the fact that $C=\colim_I G$ there,
and \cite[Proposition~2.12]{buehler} which turns such pushouts to conflations, we obtain the following diagram in $\cat F$ with conflations in rows,
\[
\begin{tikzcd}
L_j(G) \arrow[r, >->] \arrow[d] & G_j\oplus \colim_J u^*(G) \arrow[r, ->>] \arrow[d] & \colim_I G \arrow[d] \\
L_j(H) \arrow[r, >->] \arrow[d] & H_j\oplus \colim_J u^*(H) \arrow[r, ->>] \arrow[d] & \colim_I H \ar[d] \\
L_j(K) \arrow[r, >->] & K_j\oplus \colim_J u^*(K) \arrow[r, ->>] & \colim_I K
\end{tikzcd}
\]
Now assumption (3) tells us that we also have conflations in the first two columns, so we can conclude by the $3\times3$-Lemma \cite[Corollary 3.6]{buehler}, which tells us that the last column is also a conflation.
\end{proof}

Now we are ready to give a description of Gorenstein projective objects in $\cat F^I$ for a weakly idempotent complete Frobenius exact category $\cat F$. A corresponding result for Gorenstein injectives follows as well by formal duality, as the category of Gorenstein injective objects of $\cat F^I$ is opposite to the category of Gorenstein projective objects of $(\opp{\cat F})^{\opp{I}}$.

\begin{theorem}\label{theorem:Gorenstein_projectives}
Let $\cat F$ be a weakly idempotent complete Frobenius exact category, let $I$ be a finite direct category, and let $G\in\cat F^I$. Then the following are equivalent:
\begin{enumerate}
\item $G$ is Gorenstein projective in $\cat F^I$.
\item The latching object $L_j(G)$ exists and the latching morphism
$\lambda_j(G)\colon L_j(G)\to G_j$
is an inflation for each $j\in I$.
\end{enumerate}
Moreover, for each $G\in\GProj(\cat F^I)$, the colimit $\colim_IG$ exists in $\cat F$, and the induced colimit functor
$\colim_I\colon \GProj(\cat F^I)\to\cat F$
is exact.
\end{theorem}

\begin{proof}
Let $d\colon \operatorname{Ob}(I) \to \{0,\dots n\}$ be a degree function as in Definition~\ref{definition:finite_directed}. We will prove the statement by induction on $n\ge 0$ and, for finite direct categories with the same codomain of $d$, by the number of objects.

If $n=0$ or $I$ is empty, the category $I$ is a finite discrete category, so $\cat F^I$ is Frobenius again. The latching objects are trivial (they are computed as empty colimits) and $\colim_I$ is just the $I$-fold coproduct. Hence, the result follows trivially.

Let $n>0$ and $I$ be non-empty. We first prove the equivalence between (1) and (2).
To start with, note that if $G\in\GProj(\cat F^I)$ and $j\in J$, then $\pi_j^*(G)\in\GProj(\cat F^{\partial(I/j)})$ by Lemma~\ref{lemma:latching-and-restriction} and the finite direct category $\partial(I/j)$ is smaller in the sense of our induction. In particular, all the latching objects $L_j(X)$, $j\in J$, exist for Gorenstein projectives by the inductive hypothesis.

Suppose now that $G\in\cat F^I$ is such that all the latching objects $L_j(G)$ exist. Then we have
\begin{align*}
\Hom_{\cat F}(L_j(G),Q) &= \Hom_{\cat F}(\colim_{(i\to j)\in\partial(I/j)} G_i,Q) \\
&\cong \lim_{(i\to j)\in\partial(I/j)} \Hom_{\cat F}(G_i,Q) \\
&\cong \lim_{(i\to j)\in\partial(I/j)} \Hom_{\cat F^I}(G,i_*Q) \\
&\cong \Hom_{\cat F^I}(G,\lim_{(i\to j)\in\partial(I/i)} i_*Q)
\end{align*}
 for each $j\in J$ and $Q\in\Proj(\cat F)$. Indeed, the limit in $\cat F^I$ in the last row exists by Proposition~\ref{proposition:presentation_of_stalks}(2) and it coincides with the mathching object $M_j(Y)$ in the conflation in its statement.
 In fact, the argument shows that the homomorphism of abelian groups
 \[ \Hom_{\cat F}(\lambda_j(G),Q)\colon \Hom_{\cat F}(G_j, Q) \to \Hom_{\cat F}(L_j(G), Q) \]
 is isomorphic to
 \[ \Hom_{\cat F^I}(G,\mu_j(Y))\colon \Hom_{\cat F^I}(G,j_*Q) \to \Hom_{\cat F^I}(G,M_j(Y)), \]
where $Y\in (\cat F^I)^{\opp{I}}$ and $\mu_i(Y)\colon j_*Q\to M_j(Y)$ are again as in Proposition~\ref{proposition:presentation_of_stalks}(2). Since $j_*Q$ is injective with respect to the degreewise exact structure on $\cat F^I$ because $Q\in\Proj(\cat F)=\Inj(\cat F)$, an application of $\Hom_{\cat F^I}(G,-)$ to the short exact sequence from Proposition~\ref{proposition:presentation_of_stalks}(2) shows that
\[
\coker \Hom_{\cat F}(\lambda_j(G),Q) \cong \coker \Hom_{\cat F^I}(G,\mu_j(Y)) \cong \Ext^1_{\cat F^I}(G,\stalk{j}(Q)).
\]
Thus, by Lemma~\ref{lemma:Gorenstein-via-stalks}(1), $G$ is Gorenstein projective if and only if $\Hom_{\cat F}(\lambda_j(G),Q)$ is surjective for each $j\in I$ and $Q\in\Proj(\cat F)$.
Since we assume that $\cat F$ is weakly idempotent complete, this is further equivalent to $\lambda_j(G)$ being an inflation for all $j\in J$ by the so-called Obscure Axiom (for a dual version, see~\cite[Proposition~7.6 and Exercise~11.11]{buehler}). This completes the proof of $(1)\Leftrightarrow(2)$.

To prove the moreover part, let $j\in I$ be a maximal object (in the sense described just above Lemma~\ref{lemma:colimit-existence-inductive}). If we put $J=I\setminus\{j\}$ and let $u\colon J\to I$ be the inclusion functor, then $u$ is clearly a sieve and $J$ is a finite direct category which is smaller in the sense of our induction. So, given $G\in\GProj(\cat F^I)$, we have $u^*(G)\in\GProj(\cat F^J)$ by Lemma~\ref{lemma:latching-and-restriction} and $\colim_J u^*(G)$ exists in $\cat F$ by the inductive hypothesis. Then $\colim_I G$ also exists in $\cat F$ by Lemma~\ref{lemma:colimit-existence-inductive}.

It remains to show that $\colim_I\colon \GProj(\cat F^I)\to \cat F$ is exact.
If $\varepsilon\colon 0\to G\to H\to K\to 0$ is a conflation in $\GProj(\cat F^I)$, the restrictions $u^*(\varepsilon)$ and $\pi_j^*(\varepsilon)$ are conflations in $\GProj(\cat F^J)$ and $\GProj(\cat F^{\partial(I/j)})$, respectively, by Lemma~\ref{lemma:latching-and-restriction}, and $\colim_J u^*(\varepsilon)$ and $L_j(\varepsilon)=\colim_{\partial(I/j)}\pi_j^*(\varepsilon)$ are conflations in $\cat F$ by the inductive hypothesis. Finally, the conclusion that $\colim_I(\varepsilon)$ is a conflation in $\cat F$ follows from Lemma~\ref{lemma:colimit-exactness-inductive}.
\end{proof}

To deal with the derivator of $\cat F$, we must also treat Kan extensions. We again state the result only for Gorenstein projectives, since the Gorenstein injective version is formally dual.

\begin{corollary}\label{corollary:Gorenstein_left_Kan}
Let $\cat F$ be a weakly idempotent complete Frobenius exact category and let $u\colon I\to J$ be a functor between finite direct categories. Then there is an exact functor
\[ u_!\colon \GProj(\cat F^I) \to \GProj(\cat F^J) \]
and an isomorphism $\Hom_{\cat F^I}(G,u^*(Y))\cong\Hom_{\cat F^J}(u_!(G),Y)$ natural in both variables, where $G\in\GProj(\cat F^I)$ and $Y\in\cat F^J$ (one can say that $u_!$ is a partially defined left adjoint to $u^*\colon\cat F^J\to\cat F^I$).
\end{corollary}

\begin{proof}
The functor $u_!$ can be constructed using \cite[Theorem~X.3.1]{maclane-categories} (the same argument as in Prosposition~\ref{proposition:Kan-extensions-existence}), since all the necessary colimits exist for Gorenstein projective $I$-shaped diagrams by Theorem~\ref{theorem:Gorenstein_projectives}. Moreover, $u_!$ is exact since all the colimits are exact by the theorem, so we have an exact functor $u_!\colon \GProj(\cat F^I) \to \cat F^J$ and the required natural isomorphism.

It remains to show that $u_!(G)\in\GProj(\cat F^J)$ for $G\in\GProj(\cat F^I)$.
Note that $u_!$ sends projectives to projectives (being a left adjoint to an exact functor, the fact that it only is partially defined does not matter), and so also projective resolutions to projective resolutions. Hence, we also have a natural isomorphism
\[ \Ext^1_{\cat F^I}(G,u^*(Y))\cong\Ext^1_{\cat F^J}(u_!(G),Y). \]
Since clearly $u^*$ sends weakly trivial objects to weakly trivial objects by Lemma~\ref{lemma:finite_homological_dim_F^I} (recall Definition~\ref{definition:Gorenstein_key_classes}), we have
$\Ext^1_{\cat F^J}(u_!(G),W)\cong\Ext^1_{\cat F^I}(G,u^*(W))=0$ for any $W\in\WTriv(\cat F^J)$. Hence $u_!(G)\in\GProj(\cat F^J)$ by Proposition~\ref{proposition:Gore_cotorsion_pairs}, as required.
\end{proof}

\subsection{The derivator of a Frobenius exact category revisited}
\label{subsection:derivator_Frobenius}

Now we are in a position to state the main result, which provides a more direct desription of the derivator of a weakly idempotent complete Frobenius exact category, without using complexes. We will essentially use the model structures from Example~\ref{example:Gorentstein}, and informed readers will recognize them as suitable Reedy model structures \cite[\S15.3]{hirschhorn-modelcategories} of the corresponding (trivial) model structures in $\cat F$ itself. The functors $u_!\colon \GProj(\cat F^I) \to \GProj(\cat F^J)$ and $u_*\colon \GInj(\cat F^I) \to \GInj(\cat F^J)$ are then none other than versions of (partially defined) left and right Quillen functors \cite[\S8.5]{hirschhorn-modelcategories}, respectively.

\begin{theorem}\label{theorem:derivator_Frobenius_direct}
Let $\cat F$ be a weakly idempotent complete Frobenius exact category and let $\derivator_{\cat F} \colon \opp{\kat{FinDir}} \to \kat{CAT}$ be the derivator given by the assignment~\eqref{equation:derivatordg_Frobenius} (which is a special case of Theorem~\ref{theorem:derivator_dg_finite}). Then $\derivator_{\cat F}$ is equivalent to one given by the assignment
\[
I \longmapsto \cat F^I[W_I^{-1}],
\]
where $W_I$ consists of the morphisms of diagrams $f\colon X\to Y$ in $\cat F^I$ such that $f_i\colon X_i\to Y_i$ is a stable isomorphism in $\cat F$ (i.e.\ the coset of $f_i$ is an isomorphism in $\underline{\cat F}$) for each $i\in I$.
Moreover, for any functor $u\colon I\to J$ between finite direct categories, the left and right adjoints to $u^*$ are given by the diagram
\begin{equation}\label{equation:Frobenius-Kan-extensions}
\begin{tikzcd}
\underline{\GProj(\cat F^I)} \ar[d, bend right, swap, pos=0.45, "u_!"] & \derivator_{\cat F}(I) \arrow[l, swap, <-, "\simeq"] \arrow[r, <-, "\simeq"] & \underline{\GInj(\cat F^I)}  \ar[d, bend left, pos=0.45, "u_*"] \\
\underline{\GProj(\cat F^J)} & \derivator_{\cat F}(J) \arrow[u, "u^*"] \arrow[l, swap, <-, "\simeq"] \arrow[r, <-, "\simeq"] & \underline{\GInj(\cat F^J)},
\end{tikzcd}
\end{equation}
where $u_!$ and $u_*$ are induced by Corollary~\ref{corollary:Gorenstein_left_Kan} and its dual.
\end{theorem}

\begin{proof}
We have an obvious commutative diagram of categories
\[
\begin{tikzcd}
\operatorname{C_{ac}}(\Proj(\cat F^I)) \arrow[r, "\subseteq"] \arrow[d, swap, "Z^0"] &
\operatorname{C_{ac}}(\Proj(\cat F)^I) \arrow[d, swap, "Z^0"] &
\operatorname{C_{ac}}(\Inj(\cat F^I)) \arrow[l, swap, "\supseteq"] \arrow[d, "Z^0"]
\\
\GProj(\cat F^I) \arrow[r, "\subseteq"] &\cat F^I &
\GInj(\cat F^J), \arrow[l, swap, "\supseteq"]
\end{tikzcd}
\]
which induces the following commutative diagram.
\begin{equation}\label{equation:Frobenius-derivator-comparison}
\begin{tikzcd}
\operatorname{K_{ac}}(\Proj(\cat F^I)) \arrow[r] \arrow[d, swap, "Z^0"] &
\operatorname{K_{ac}}(\Proj(\cat F)^I) / \operatorname{K_{tc}}(\Proj(\cat F)^I) \arrow[d, swap, "Z^0"] &
\operatorname{K_{ac}}(\Inj(\cat F^I)) \arrow[l] \arrow[d, "Z^0"]
\\
\underline{\GProj(\cat F^I)} \arrow[r] &\cat F^I[W_I^{-1}] &
\underline{\GInj(\cat F^J)}. \arrow[l]
\end{tikzcd}
\end{equation}
Here, the middle column is well defined by Lemma~\ref{lemma:finite_homological_dim_F^I}, the upper horizontal arrows are equivalences by the proof of Theorem~\ref{theorem:derivator-Frobenius}, and the lower horizontal arrows are equivalences by Example~\ref{example:Gorentstein}. Moreover, the outer vertical arrows are equivalences thanks to Corollary~\ref{corollary:GP_and_GI_are_Frobenius} and~\eqref{equation:Frobenius-dg}, so all the functors in~\eqref{equation:Frobenius-derivator-comparison} are equivalences. Finally, the diagram is natural in $I\in\opp{\kat{FinDir}}$ in the sense that for any functor $u\colon I\to J$ in $\opp{\kat{FinDir}}$, the restrictions along $u$ commute with the two instances of~\eqref{equation:Frobenius-derivator-comparison} for $I$ and $J$, forming a three-dimensional commutative diagram. 
In summary, together with Theorem~\ref{theorem:derivator-Frobenius}, we have the equivalences
\[ \derivator_{\cat F}(I) \simeq \operatorname{K_{ac}}(\Proj(\cat F)^I) / \operatorname{K_{tc}}(\Proj(\cat F)^I) \simeq \cat F^I[W_I^{-1}] \]
which are natural in $I$, so induce the required equivalence of (pre)derivators. The description of Kan extensions immediately follows from the above and Corollary~\ref{corollary:Gorenstein_left_Kan} and its dual.
\end{proof}

\begin{remark}
Suppose that $u\colon I\to J$ is a functor between finite direct categories. Since the partial Kan extensions $u_!$, $u_*$ given by Corollary~\ref{corollary:Gorenstein_left_Kan} are exact and preserve projective, resp.\ injective diagrams, they also induce functors $u_!\colon \operatorname{K_{ac}}(\Proj(\cat F^I))\to \operatorname{K_{ac}}(\Proj(\cat F^J))$ and $u_*\colon \operatorname{K_{ac}}(\Inj(\cat F^I))\to \operatorname{K_{ac}}(\Inj(\cat F^J))$. Hence, axiom (Der3) for the derivator of a weakly idempotent complete Frobenius exact category can be also expressed by the following commutative diagram similar to~\eqref{equation:Frobenius-Kan-extensions}, using the description of the derivator given in Theorem~\ref{theorem:derivator-Frobenius}: 
\[
\begin{tikzcd}
\operatorname{K_{ac}}(\Proj(\cat F^I)) \arrow[r, "\simeq"] \arrow[d, swap, bend right, "u_!"] &
\frac{\operatorname{K_{ac}}(\Proj(\cat F)^I)}{\operatorname{K_{tc}}(\Proj(\cat F)^I)} \arrow[d, swap, <-, "u^*"] &
\operatorname{K_{ac}}(\Inj(\cat F^I)) \arrow[l, "\simeq"] \arrow[d, bend left, "u_*"]
\\
\operatorname{K_{ac}}(\Proj(\cat F^J)) \arrow[r, "\simeq"] &
\frac{\operatorname{K_{ac}}(\Proj(\cat F)^J)}{\operatorname{K_{tc}}(\Proj(\cat F)^J)} &
\operatorname{K_{ac}}(\Inj(\cat F^J)) \arrow[l, "\simeq"]
\end{tikzcd}
\]
\end{remark}

\begin{remark}
For some functors $u\colon I\to J$ between finite direct categories, the situation is even easier. If for each $j\in J$, the slice $j/u$ is a disjoint union of finitely many categories with initial objects, then the restriction functor $u^*$ preserves (Gorenstein) projectives by the same argument as for Lemma~\ref{lemma:latching-and-restriction}. That is, we have honest adjunctions
\[
u_!\colon \underline{\GProj}(\cat F^I)\rightleftharpoons\underline{\GProj}(\cat F^J):\!u^*
\quad\text{and}\quad
u_!\colon \operatorname{K_{ac}}(\Proj(\cat F^I))\rightleftharpoons\operatorname{K_{ac}}(\Proj(\cat F^J)):\!u^*.
\]
A dual remark applies if $u/j$ is a disjoint union of finitely many categories with terminal objects.
\end{remark}

\bibliographystyle{plain}

\end{document}